\documentclass[oneside]{amsart}
\newcommand{\formatswitch}{preprint}

\usepackage{amssymb}
\usepackage{ifthen}
\usepackage{verbatim}
\usepackage{epsfig}
\usepackage[all]{xy}

\newcommand{\tref}[1]{(\ref{#1})}

\ifthenelse{\equal{\formatswitch}{big}}{
\setlength{\textwidth}{432pt}
\setlength{\oddsidemargin}{18.8775pt}

\setcounter{tocdepth}{1}
}{
\ifthenelse{\equal{\formatswitch}{paper}}{

\setcounter{tocdepth}{1}
}{
\setcounter{tocdepth}{2}
}
}
\DeclareMathAlphabet\EuScript{U}{eus}{m}{n}
\DeclareMathAlphabet\EuScriptb{U}{eus}{b}{n}

\newcommand{\claimenum}{\renewcommand{\theenumi}{\alph{enumi}}
 \renewcommand{\labelenumi}{\textit{(\theenumi)}}
 \renewcommand{\theenumii}{\roman{enumii}}
 \renewcommand{\labelenumii}{\textit{(\theenumii)}}
 \begin{enumerate}}
\newcommand{\claimenumend}{\end{enumerate}}

\newcommand{\romanenum}{\renewcommand{\theenumi}{\roman{enumi}}
 \renewcommand{\labelenumi}{\textit{(\theenumi)}}
 \renewcommand{\theenumii}{\alph{enumii}}
 \renewcommand{\labelenumii}{\textit{(\theenumii)}}
 \begin{enumerate}}
\newcommand{\romanenumend}{\end{enumerate}}

\newtheorem{dummy}{realdumb}[section]
\newtheorem{thm}[dummy]{Theorem}

\newtheorem{lemma}[dummy]{Lemma}
\newtheorem{prop}[dummy]{Proposition}
{\theoremstyle{definition} }

{\theoremstyle{definition} }
\newtheorem{cor}{Corollary}[dummy]

\renewcommand{\text}{\mathrm}

\newcommand{\strutdepth}{\dp\strutbox}
\newcommand{\marginalnote}[1]
   {\strut\vadjust{\kern-\strutdepth\domarginalnote{#1}}}
\newcommand{\domarginalnote}[1]{\vtop to \strutdepth{
  \baselineskip\strutdepth
   \vss\llap{ #1\ \ }\null}}  
\newcounter{showlabelflag}
\newcounter{makelabelflag}
\newcommand{\showlabels}{\setcounter{showlabelflag}{1}}

\newcommand{\makelabels}{\setcounter{makelabelflag}{1}}
\newcommand{\hidelabels}{\setcounter{showlabelflag}{2}}
\newcommand{\mylabel}[1]{
  \ifthenelse{\value{makelabelflag}=1}
    {\label{#1}}{}
  \ifthenelse{\value{showlabelflag}=1}
    {\marginpar{#1}}{}\relax}

\newcommand{\R}{{\mathbf R}}

\newcommand{\Z}{{\mathbf Z}}

\ifthenelse{\equal{\formatswitch}{draft}}{
\showlabels
}{\relax}

\newtheorem{question}[dummy]{Question}
\newtheorem{conj}[dummy]{Conjecture}

\newcommand{\Saaty}{\cite{MR863420}}
\newcommand{\Gravier}{\cite{MR1932681}}
\newcommand{\Carpentier}{\cite{MR2764810}}
\newcommand{\Kryu}{\cite{Kryuchkov:rotations}}
\newcommand{\EliaOne}{\cite{MR1721923}}
\newcommand{\Zeilberger}{\cite{MR2873886}}
\newcommand{\txst}{\textstyle}

\newcommand{\scs}{\scriptstyle}

\newcommand{\mymargin}[1]{
  \ifthenelse{\value{showlabelflag}=1}
    {\marginpar{#1}}{}\relax}
\newcommand{\cprime}{{\mathsurround=0pt\({}'\)}}

\begin{document}

\bibliographystyle{amsplain}
\title[Coloring planar graphs]{Coloring Planar graphs \\ via \\
Colored Paths in the Associahedra}
\thanks{AMS
Classification (2010): primary 05C15, secondary 20F05, 05C10, 05C30,
05C20.}
\author{GARRY BOWLIN}
\author{MATTHEW G. BRIN}
\date{January 12, 2013}


\makelabels
\hidelabels
\maketitle


{\itshape Abstract.}  Hassler Whitney's theorem of 1931 reduces the
task of finding proper, vertex 4-colorings of triangulations of the
2-sphere to finding such colorings for the class \(\mathfrak H\) of
triangulations of the 2-sphere that have a Hamiltonian circuit.
This has been used by Whitney and others from 1936 to the present to
find equivalent reformulations of the 4 Color Theorem (4CT).
Recently there has been activity to try to use some of these
reformuations to find a shorter proof of the 4CT.  Every
triangulation in \(\mathfrak H\) has a dual graph that is a union of
two binary trees with the same number of leaves.  Elements of a
group known as Thompson's group \(F\) are equivalence classes of
pairs of binary trees with the same number of leaves.  This paper
explores this resemblance and finds that some recent reformulations
of the 4CT are essentially attempting to color elements of
\(\mathfrak H\) using expressions of elements of \(F\) as words in a
certain generating set for \(F\).  From this, we derive information
about not just the colorability of certain elements of \(\mathfrak
H\), but also about all possible ways to color these elements.
Because of this we raise (and answer some) questions about
enumeration.  We also bring in an extension \(E\) of the group \(F\)
and ask whether certain elements ``parametrize'' the set of all
colorings of the elements of \(\mathfrak H\) that use all four
colors.


\contentsline {section}{\tocsection {}{1}{Introduction}}{2}
\contentsline {subsection}{\tocsubsection {}{1.1}{Background}}{2}
\contentsline {subsubsection}{\tocsubsubsection {}{1.1.1}{Two reductions}}{2}
\contentsline {subsubsection}{\tocsubsubsection {}{1.1.2}{Whitney's theorem}}{3}
\contentsline {subsubsection}{\tocsubsubsection {}{1.1.3}{The setting and the dual views}}{3}
\contentsline {subsubsection}{\tocsubsubsection {}{1.1.4}{Previous work}}{4}
\contentsline {subsection}{\tocsubsection {}{1.2}{On the current paper}}{5}
\contentsline {subsubsection}{\tocsubsubsection {}{1.2.1}{Subgroups and monoids}}{6}
\contentsline {subsubsection}{\tocsubsubsection {}{1.2.2}{Associahedra and color graphs}}{6}
\contentsline {subsubsection}{\tocsubsubsection {}{1.2.3}{Enumeration}}{7}
\contentsline {subsubsection}{\tocsubsubsection {}{1.2.4}{Other}}{7}
\contentsline {subsection}{\tocsubsection {}{1.3}{In the paper}}{7}
\contentsline {subsection}{\tocsubsection {}{1.4}{Thanks}}{9}
\contentsline {section}{\tocsection {}{2}{Starting from Whitney's theorem}}{9}
\contentsline {subsection}{\tocsubsection {}{2.1}{The equator}}{9}
\contentsline {subsection}{\tocsubsection {}{2.2}{Rooting and orienting the trees}}{10}
\contentsline {subsection}{\tocsubsection {}{2.3}{Reversing the process}}{11}
\contentsline {subsection}{\tocsubsection {}{2.4}{Triangulations}}{11}
\contentsline {section}{\tocsection {}{3}{Trees}}{12}
\contentsline {subsection}{\tocsubsection {}{3.1}{General trees}}{12}
\contentsline {subsection}{\tocsubsection {}{3.2}{Orders in trees}}{12}
\contentsline {subsection}{\tocsubsection {}{3.3}{Standing assumptions on trees}}{13}
\contentsline {subsection}{\tocsubsection {}{3.4}{Subtrees}}{13}
\contentsline {subsection}{\tocsubsection {}{3.5}{Standard model for binary trees}}{13}
\contentsline {subsection}{\tocsubsection {}{3.6}{Inducting with trees}}{14}
\contentsline {subsection}{\tocsubsection {}{3.7}{Projections}}{15}
\contentsline {subsection}{\tocsubsection {}{3.8}{Root shifts and reflections}}{15}
\contentsline {section}{\tocsection {}{4}{The associahedra}}{16}
\contentsline {subsection}{\tocsubsection {}{4.1}{The faces}}{17}
\contentsline {subsection}{\tocsubsection {}{4.2}{Edges and rotations}}{17}
\contentsline {subsection}{\tocsubsection {}{4.3}{Examples}}{18}
\contentsline {subsection}{\tocsubsection {}{4.4}{Symmetries}}{20}
\contentsline {section}{\tocsection {}{5}{Color}}{20}
\contentsline {subsection}{\tocsubsection {}{5.1}{The colors}}{20}
\contentsline {subsection}{\tocsubsection {}{5.2}{Face and edge colorings}}{20}
\contentsline {subsection}{\tocsubsection {}{5.3}{Coloring binary trees}}{20}
\contentsline {subsection}{\tocsubsection {}{5.4}{Coloring binary tree pairs}}{21}
\contentsline {subsection}{\tocsubsection {}{5.5}{Signs}}{22}
\contentsline {subsection}{\tocsubsection {}{5.6}{Small examples}}{23}
\contentsline {subsection}{\tocsubsection {}{5.7}{Permuting colors}}{23}
\contentsline {subsection}{\tocsubsection {}{5.8}{The dual view}}{24}
\contentsline {section}{\tocsection {}{6}{Colored rotations and colored paths}}{24}
\contentsline {subsection}{\tocsubsection {}{6.1}{Colored rotations}}{24}
\contentsline {subsection}{\tocsubsection {}{6.2}{A converse to Proposition 6.2\hbox {}}}{26}
\contentsline {subsection}{\tocsubsection {}{6.3}{The first signed path conjecture}}{26}
\contentsline {subsection}{\tocsubsection {}{6.4}{Rigidity basics}}{26}
\contentsline {subsection}{\tocsubsection {}{6.5}{Colored associahedra and color graphs}}{27}
\contentsline {section}{\tocsection {}{7}{Groups}}{27}
\contentsline {subsection}{\tocsubsection {}{7.1}{Thompson's group \(F\)}}{27}
\contentsline {subsection}{\tocsubsection {}{7.2}{Carets}}{29}
\contentsline {subsection}{\tocsubsection {}{7.3}{Reduced pairs}}{29}
\contentsline {subsection}{\tocsubsection {}{7.4}{The multiplication}}{29}
\contentsline {subsection}{\tocsubsection {}{7.5}{Unions, intersections and differences of trees}}{30}
\contentsline {subsection}{\tocsubsection {}{7.6}{Another representation}}{31}
\contentsline {subsection}{\tocsubsection {}{7.7}{Actions on finite trees}}{31}
\contentsline {subsection}{\tocsubsection {}{7.8}{Rotation as action}}{32}
\contentsline {subsection}{\tocsubsection {}{7.9}{Edge paths in the associahedra}}{32}
\contentsline {subsection}{\tocsubsection {}{7.10}{Presentations}}{32}
\contentsline {subsection}{\tocsubsection {}{7.11}{Sizes of trees in tree pairs}}{33}
\contentsline {subsection}{\tocsubsection {}{7.12}{The group \(E\)}}{34}
\contentsline {section}{\tocsection {}{8}{Colors and the group operations}}{34}
\contentsline {subsection}{\tocsubsection {}{8.1}{Coloring representatives of a single element of \(F\)}}{34}
\contentsline {subsection}{\tocsubsection {}{8.2}{Extending the set of colored elements}}{35}
\contentsline {subsection}{\tocsubsection {}{8.3}{The compatibility lemma}}{35}
\contentsline {subsection}{\tocsubsection {}{8.4}{Vines}}{36}
\contentsline {subsection}{\tocsubsection {}{8.5}{Action of rotations on colored pairs}}{36}
\contentsline {subsection}{\tocsubsection {}{8.6}{Application to specific subsets of \(F\)}}{38}
\contentsline {subsection}{\tocsubsection {}{8.7}{Further considerations}}{38}
\contentsline {section}{\tocsection {}{9}{Rigid colorings}}{39}
\contentsline {subsection}{\tocsubsection {}{9.1}{Depth condition}}{39}
\contentsline {subsection}{\tocsubsection {}{9.2}{The positive rigid pattern}}{39}
\contentsline {subsection}{\tocsubsection {}{9.3}{The group \(F_4\)}}{40}
\contentsline {subsection}{\tocsubsection {}{9.4}{Characterizing positive rigid color vectors}}{40}
\contentsline {section}{\tocsection {}{10}{Color graphs, zero sets, shadow patterns, and long paths}}{40}
\contentsline {subsection}{\tocsubsection {}{10.1}{Zero sets}}{40}
\contentsline {subsection}{\tocsubsection {}{10.2}{A separating example}}{41}
\contentsline {subsection}{\tocsubsection {}{10.3}{Long paths}}{43}
\contentsline {section}{\tocsection {}{11}{Sign structures}}{44}
\contentsline {subsection}{\tocsubsection {}{11.1}{The signed graph}}{44}
\contentsline {subsection}{\tocsubsection {}{11.2}{The second signed path conjecture}}{46}
\contentsline {subsection}{\tocsubsection {}{11.3}{The vertex set of a sign structure}}{46}
\contentsline {subsection}{\tocsubsection {}{11.4}{Relations among the paths}}{47}
\contentsline {subsubsection}{\tocsubsubsection {}{11.4.1}{Moving paths across faces}}{48}
\contentsline {subsubsection}{\tocsubsubsection {}{11.4.2}{Moving paths across pentagons}}{48}
\contentsline {subsubsection}{\tocsubsubsection {}{11.4.3}{Adding or removing canceling pairs}}{48}
\contentsline {subsubsection}{\tocsubsubsection {}{11.4.4}{Moving paths across squares}}{48}
\contentsline {subsubsection}{\tocsubsubsection {}{11.4.5}{Shortest paths are not always the best}}{49}
\contentsline {subsubsection}{\tocsubsubsection {}{11.4.6}{Parallel edges in the sign structure}}{49}
\contentsline {subsubsection}{\tocsubsubsection {}{11.4.7}{Paths with identical sign structures}}{49}
\contentsline {subsection}{\tocsubsection {}{11.5}{On the relevance of the group \(E\)}}{50}
\contentsline {subsection}{\tocsubsection {}{11.6}{On a converse to Theorem 11.5\hbox {}}}{51}
\contentsline {subsection}{\tocsubsection {}{11.7}{Non-prime maps}}{52}
\contentsline {section}{\tocsection {}{12}{Acceptable color vectors}}{52}
\contentsline {section}{\tocsection {}{13}{Patterns}}{53}
\contentsline {subsection}{\tocsubsection {}{13.1}{Patterns and multiplication}}{53}
\contentsline {subsection}{\tocsubsection {}{13.2}{The positive pattern and group}}{54}
\contentsline {subsection}{\tocsubsection {}{13.3}{Neighborhoods of positive colorings}}{55}
\contentsline {section}{\tocsection {}{14}{Higher genera and Thompson's group \(V\)}}{56}
\contentsline {subsection}{\tocsubsection {}{14.1}{Tree pairs on the torus}}{56}
\contentsline {subsection}{\tocsubsection {}{14.2}{Thompson's group \(V\)}}{57}
\contentsline {subsection}{\tocsubsection {}{14.3}{Tree pairs on the projective plane}}{58}
\contentsline {section}{\tocsection {}{15}{Enumeration}}{59}
\contentsline {subsection}{\tocsubsection {}{15.1}{Trees}}{59}
\contentsline {subsection}{\tocsubsection {}{15.2}{Colorings of a single tree}}{59}
\contentsline {subsection}{\tocsubsection {}{15.3}{Recursive sequences}}{59}
\contentsline {subsection}{\tocsubsection {}{15.4}{Acceptable colorings}}{60}
\contentsline {subsection}{\tocsubsection {}{15.5}{Rigid colorings}}{61}
\contentsline {subsection}{\tocsubsection {}{15.6}{Flexible colorings}}{62}
\contentsline {subsection}{\tocsubsection {}{15.7}{Highly colorable maps}}{62}
\contentsline {subsubsection}{\tocsubsubsection {}{15.7.1}{The estimates}}{62}
\contentsline {subsubsection}{\tocsubsubsection {}{15.7.2}{The maps}}{63}
\contentsline {subsubsection}{\tocsubsubsection {}{15.7.3}{The counts}}{65}
\contentsline {subsubsection}{\tocsubsubsection {}{15.7.4}{Questions}}{66}
\contentsline {subsection}{\tocsubsection {}{15.8}{Zero sets and color graphs}}{66}
\contentsline {section}{\tocsection {}{16}{The end}}{68}
\contentsline {section}{\tocsection {}{}{References}}{69}

\section{Introduction}\mylabel{IntroSec}

This paper combines algebraic structures (that of a specific
finitely presented group) with the problem of coloring planar maps.
This paper can be viewed either as a search for a shorter proof of
the Four Color Theorem or as an exploration of the colorings that
must exist given the truth of the theorem.  We present a mechanism
that finds all colorings.  We describe the mechanism, present some
of its properties, and discuss questions of enumeration that it
leads to.

In this introduction, Section \ref{BackgroundSec} gives enough
background to tie the items in the title of the paper together,
Section \ref{ResultSumSec} describes some of our results, and
Section \ref{SectionListSec} lists the topics in the order that they
are developed in the paper.

\subsection{Background}\mylabel{BackgroundSec}

A map \(M\) is an embedding of a finite graph (whose image we will
call the underlying graph of the map and which we will still denote
by \(M\)) into the 2-sphere \(S^2\).  A proper, face 4-coloring of
\(M\) is a function from the faces of \(M\) to a set of four colors
so that any two different faces that share an edge are assigned
different colors.  Finding a proper, face 4-coloring for \(M\) is
equivalent to finding a proper, vertex 4-coloring (no two adjacent
vertices receive the same color) of the graph dual to \(M\) in
\(S^2\).  The 4 Color Theorem (4CT) says that such a 4-coloring can
always be found.

The 4CT has been proven in \cite{MR0543792}, \cite{MR0543793} and
\cite{MR1441258} and the proof has been verified by
\cite{MR2463991}.  Since these proofs and verifications use large
amounts of computer calculations, shorter proofs have been sought.

It is standard to start a discussion of the 4CT by reducing the
set of graphs to be considered.  One quick reduction is to the
class of cubic graphs.  A further reduction using the shrinking of
edges eliminates loops in the dual graph in \(S^2\).  See Section
3-2 of \Saaty.  Shrinking edges also can eliminate parallel edges in
the dual graph, but we want to do this elimination using a different
technique since the different technique is more relevant to us.

\subsubsection{Two reductions}\mylabel{TwoRedSec}

For the rest of this introduction we will assume that all graphs
have no loops.

Let  \(T\) be a finite graph on \(S^2\) all of whose
faces are triangles, but that fails to be a triangulation because it
has parallel edges.  Let \(C\) be a circuit consisting of two
parallel edges.  Each complementary domain of \(C\) in \(S^2\)
contains a vertex of \(T\) since every face of \(T\) is a triangle.
There are two graphs we can get by deleting all vertices and edges
of \(T\) that are in one complementary domain of \(C\).  Call these
graphs the result of ``cutting \(T\) along \(C\).''  Each of these
can be properly vertex 4-colored by an inductive hypothesis.  In
each of these graphs, the two colors given to the vertices of \(C\)
are different.  Thus there are exactly two permutations of the
colors of one of the graphs that allow us to put the two colorings
together to color all of \(T\).

The phrase ``exactly two permutations'' is relevant in the following
way.  We can clearly reduce \(T\) by a finite number of such cuts to
a finite set of pieces that are true triangulations in that they
have no parallel edges by identifying the two edges of a cut each
time we make a cut.  We will later call such pieces ``prime.''  The
original \(T\) can be recovered from such pieces if it is remembered
where the cuts were and how the pieces fit together.  It is also
clear that if there are \(k\) pieces and the various pieces have,
respectively, \(n_1, n_2, \ldots, n_k\) different proper, vertex
4-colorings up to permutations of the colors, then \(T\) has
\mymargin{PrimeProdFormula}\begin{equation}\label{PrimeProdFormula}
2^{k-1} \prod_{i=1}^k n_i \end{equation} different proper, vertex
4-colorings up to permutations of the colors.

Just as a circuit of length 2 must use two different colors for its
vertices, a circuit of length 3 must use three different colors for
its vertices.  We base a second reduction on this.

If we ignore parallel edges and turn our attention to separating
circuits of length 3, then a similar analysis takes place and we
reduce \(T\) to a finite number ``irreducble'' pieces that have no
separating circuits of length 3.  No identification of edges is
needed since the edges in the cut will form a triangular face in
each of the two subgraphs created by a cut.  If there are \(j\) such
pieces and they have (up to permutations of the colors),
respectively, \(m_1, m_2, \ldots, m_j\) different proper, vertex
4-colorings, then \(T\) has \(m_1m_2\cdots m_j\) different proper
vertex 4-colorings up to permutations of the colors.

It follows that if we completely understand proper, vertex
4-colorings of graphs in \(S^2\) whose faces are triangles and which
are both irreducible and prime, then we understand all proper, vertex
4-colorings of all all graphs in \(S^2\) whose faces are triangles.

\subsubsection{Whitney's theorem}\mylabel{IntroWhitSec}

In the previous setting, a piece that is both ``prime'' and
``irreducible'' is a finite graph in \(S^2\) with triangular faces,
with no loops, no parallel edges and no separating circuits of
length 3.  Let \(P\) be such a piece.  By the main theorem of
\cite{MR1503003}\index{Whitney!theorem}, \(P\) must have a
Hamiltonian circuit \(C\).  A restatement of Whitney's theorem is
given as Theorem 4-7 in \Saaty.  If \(P\) has \(n\) vertices, then
cutting \(P\) along \(C\) will give two triangulated polygons
of \(n\) sides and edges.

We take this conclusion to be the starting point of the paper.  We
consider maps \(M\) in \(S^2\) whose underlying graph is cubic (so
that the dual has only triangular faces) and whose dual has a
Hamiltonian circuit and no loops.  Since we base our definition on
the conclusion of Whitney's theorem, we lose the parts of the
hypotheses that are not mentioned.  In particular, the dual may have
parallel edges and separating circuits of length 3.  It turns out
that the parallel edges will concern us in what follows, but we will
not worry at all about separating circuits of length 3.  Perhaps
further studies taking these circuits into account could be
fruitful.

\subsubsection{The setting and the dual
views}\mylabel{IntroSettingSec}

In the following, ``triangulated polygon'' will always assume that
all the vertices are on the boundary of the polygon.  That is, no
extra vertices were introduced by the triangulation.

We study those maps \(M\) in \(S^2\) that are in a certain class
\(\mathfrak W\)\index{\protect\(\protect\mathfrak W\protect\)}.  For
\(M\) to be in \(\mathfrak W\), its dual graph in \(S^2\) must have
no loops, must have triangles for faces and must have a Hamiltonian
circuit.  The graph \(M\) is thus cubic and connected.

The Hamiltonian circuit in the dual of \(M\) will be called the
equator.  It cuts through each face of \(M\) in a single arc and we
require that it passes through no vertex of \(M\) and be transverse
to the edges of \(M\).  The intersection \(D\) of \(M\) with a
complementary domain of the equator will be a tree since if \(D\) is
not connected, then some face of \(M\) will hit the equator twice
and if \(D\) is not simply connected, then some face of \(M\) will
not hit the equator.  The resulting tree is binary since all
vertices other than the leaves have degree three.  Thus \(M\)
consists of two binary trees that meet each other in their leaves
along the equator.

The dual graph \(T\) of \(M\) will fall into two triangulated
polygons \(C_1\) and \(C_2\) bounded by the equator.  While \(T\)
may have parallel edges, no two parallel edges will reside in one of
the \(C_i\) since there would then have to be a face that is not a
triangle.  Thus two parallel edges must have one edge in \(C_1\) and
the other in \(C_2\).  The restriction of \(T\) to each polygon will
then be a true triangulation of that polygon.

We make the setting more specific.  If we make \(S^2\) the set of
points in \(\R^3\) of distance one from the origin, we can insist
that the Hamiltonian circuit lie on the circle in \(S^2\) on the
\(xy\)-plane.  Projecting to the \(xy\)-plane lets us talk of two
different triangulations of the same polygon.  In the dual view, we
have two binary trees in the polygon whose leaves are the same set
of points, one leaf in the center of each edge of the Hamiltonian
circuit.

Thus our setting is either two triangulations of the same polygon,
or two binary trees in a disk with the leaves in the boundary of
disk and having the same leaf set.

This paper was motivated by the resemblance of the conclusions of
Whitney's theorem to a group known as Thompson's group \(F\).
Elements of \(F\) are built from pairs of binary trees having the
same number of leaves.  The group \(F\) has been studied extensively
since its introduction by R.~J.~Thompson around 1970.

The obvious resemblance was first noticed on the appearance of the
expository article \cite{MR1633714} by Robin Thomas.  It was felt at
the time that the proof of the 4CT, based as it was on hundreds of
cases that needed to be checked by computer, could benefit from a
little extra organization, and that this organization might just be
supplied by the group structure of \(F\).  Unfortunately, first
looks by several people found that the multiplication on the group
and the colorings of the maps seemed to have absolutely nothing to
do with one another.

In June of 2011, the authors of the current paper decided to take
another look and found (with the help of the computer) that while
the multiplication and colorings did not cooperate, there is a
strong relationship between a presentation of \(F\) by generators
and relations on the one hand, and machinery that attempts to create
4-colorings of maps in \(\mathfrak W\) on the other hand.

Because of our desire to introduce the group \(F\), and since \(F\)
is typically defined using pairs of binary trees, we will
concentrate on pairs of trees in this paper.  However, the
triangulated polygons point of view has its uses and will be used
when it is advantageous to do so.

\subsubsection{Previous work}\mylabel{IntroPrevWkSec}

Starting with Whitney \cite{MR1550643} in 1936, Whitney's theorem
\cite{MR1503003} has been used to formulate statements equivalent to
the 4CT.  Two from \cite{MR1550643} are repeated as \(\mathbf
C_{19}\) and \(\mathbf C_{20}\) in Section 5-5 of \Saaty, along with
statements due to several others.  A later formulation by Kauffman
\cite{kauffman:four-color} was discussed by R.~Thomas in
\cite{MR1633714}.  A formulation \cite{Loday:YY} by Loday is more
recent.  Most work relevant to the current article started to appear
after \cite{kauffman:four-color} and \cite{MR1633714}.

Given two triangulations of the same polygon, it has been known that
there is a sequence of ``diagonal flips'' (terminology of
\EliaOne) that will sequentially modify one triangulation
until it is identical to the other.  Dually, given two finite binary
trees with the same leaf set, there is a sequence of ``rotations''
(called transplantations in \Kryu) that will
sequentially modify one tree until it is isomorphic to the other.
Neither of these obsevations has the power to give a proof of the
4CT.

For trees, Kryuchkov \Kryu{} in 1992, and for
triangulated polygons, Eliahou \EliaOne{} in 1999 showed that
if a sequence of modifications which preserves some extra structure
could always be found, then the 4CT would follow.  (Structure
preserving moves for trees are called admissible transplantations in
\Kryu.  For polygons they are called signed
diagonal flips in \EliaOne.  Details about the extra
structures that must be preserved by the sequences of moves will be
given later.)  Specifically, if a there is an appropriate structure
that a sequence of moves preserves, then the sequence of moves leads
to a desired coloring.  This gives a pair of statements that imply
the 4CT but that might be false.  In \Gravier, Gravier and
Payan give an elegant proof that this is not the case and that these
statements follow from the 4CT.  There is actually more detail in
the result of \Gravier{} and this will be discussed shortly.

Since the trees and polygons are finite, and the extra structures
demanded by the results of \Kryu{} and
\EliaOne{} come from a finite set of possible structures, an
exhaustive search will show if a given sequence of modifications
satifies the hypotheses of these papers.  In \Carpentier,
Carpentier gives a practical algorithm to decide if a given sequence
of modifications has a corresponding structure that it preserves.
The paper \Carpentier{} uses triangulated polygons, but is easily
recast to use trees.

Many of the reformulations of the 4CT using Whitney's theorem have
an algebraic flavor.  Eliahou and Lecouvey in \cite{MR2567026}
introduce the symmetric group and Loday in \cite{Loday:YY}
introduces Lie algebras into discussion of coloring the maps in
\(\mathfrak W\).

The work of the papers discussed to this point is to establish
equivalences with the 4CT.  In \Zeilberger, Cooper, Rowland and
Zeilberger calculate proper, vertex 4-colorings for certain infinite
sets of maps represented as pairs of trees.  Decomposition into
primes is also discussed.  Their results are extended by Csar,
Sengupta and Suksompong in \cite{csar:tamari} to pairs of trees that
lie in an interval in a certain lattice.

\subsection{On the current paper}\mylabel{ResultSumSec}

We consider maps in the class \(\mathfrak W\) and we translate the
sequences of structure preserving moves of Eliahou and Kryuchkov
into words in the generators of elements of Thompson's group \(F\).
We also derive a criterion that is structurally different from (but
one assumes equivalent to) the criterion of Carpentier for a word in
the generators to have an associated structure that is preserved by
the word and thus lead to a coloring.  

Implicit in this description is the possibility that there are words
that cooperate with no structure and that lead to no coloring.  This
turns out to be the case and we call words that cooperate with some
structure ``successful.''  It is also possible that for a given map,
different words that successfully lead to colorings lead to
different colorings of the same map.  Further, it is possible that a
given word does not uniquely determine a coloring for a map.

We show that for prime maps, a successful word determines a unique
coloring of the map (modulo permutations of the colors).  We also
show that modifying a word by certain of the relations of \(F\) does
not change the coloring determined by the word.  This leads to a
presentation of an extension \(E\) of \(F\) that only uses the
relations of \(F\) that do not result in a change of coloring when
used to modify a word.  We ask whether there is a well defined
function from certain elements of \(E\) to the set of all colorings
of planar maps.

We note that it is possible to give a well defined notion of a
coloring of an element of \(F\) and it is elementary that the 4CT is
equivalent to the statement that every element of \(F\) has a
coloring.  This ignores the specific representation of the elements
of \(F\) as words in a certain generating set, and only asks that
there be some word that succeeds.  The results of \Zeilberger{}
essentially discuss colorings of elements of \(F\) without reference
to their representation as specific words in the generators.

\subsubsection{Subgroups and monoids}\mylabel{IntroMonoidSec}

There is a positive monoid \(P\) of \(F\) that is so named because
\(F\) is the group of fractions of \(P\) in that every element of
\(F\) is of the form \(pn^{-1}\) with \(p\) and \(n\) in \(P\).  We
show that all elements in \(P\) have colorings and that the prime
elements of \(P\) have unique colorings.  This duplicates, possibly
extends, and to some extent unifies several of the results in
\Zeilberger{} and \cite{csar:tamari}.  See Section \ref{RotAppSec}.

An extra detail about the results of Kryuchkov and Eliahou is that
structure preserving sequences of moves lead to 4-colorings that use
all four colors.  The converse by Gravier and Payan \Gravier{}
assumes a coloring using all four colors and finds for that coloring
a structure preserving sequence leading to the coloring.

Colorings using only three colors have no corresponding
sequences.  We refer to such colorings as rigid.  (Non-rigid
colorings are called flexible.)  There is an easy characterization
of rigid colorings in \Gravier{} and the well known fact that
a map in \(S^2\) can have a proper, face 3-coloring if and only if
each face has an even number of edges is presented as Theorem 2-5 of
\Saaty.  The possibility of a 3-coloring is not a large problem
since every map with more than 3 faces having a proper, face
3-coloring also has a proper, face 4-coloring that uses all four
colors.

We show that the maps having 3-colorings correspond to a pleasantly
described subgroup \(F_4\) of \(F\).  

The subgroup \(F_4\) corresponds to a certain ``pattern'' that
assigns a fixed coloring to all finite binary trees.  We call it the
rigid pattern since it allows for no moves of the type considered by
Kryuchkov.  We briefly discuss what it takes for a pattern to lead
to a subgroup of \(F\) and briefly consider the pattern that is the
``least rigid.''  This leads to a subgroup of \(F\) that seems not
to have been considered before.

\subsubsection{Associahedra and color graphs}\mylabel{IntroAssocSec}

Associahedra are convex polytopes whose vertices are finite binary
trees or (dually) triangulated polygons.  There is one associahedron
in each dimension and we denote the associahedron of dimension \(d\)
by \(A_d\).  The vertices of \(A_d\) are triangulated polygons
having \(d+3\) edges or binary trees with \(d+3\) leaves (or \(d+2\)
leaves if one leaf has been promoted to the title of ``root'' of the
tree).  Sequences of moves considered by Kryuchkov or Eliahou are
walks along the edges of \(A_d\) with each move corresponding to a
traversal of one edge.

Let us denote by \(\mathbf v\) a 4-coloring (with no requirement
that it be proper) of a polygon with \(d+3\) edges.  We can refer to
\(\mathbf v\) as a ``color cycle.''  (When working with trees, the
``color cycle'' will be replaced by a ``color vector.''  See Section
\ref{ColBinTreeSec}.)  Then \(\mathbf v\) is either a proper, vertex
4-coloring or not for a given triangulation of the polygon depending
on the triangulation.  Thus certain vertices of the associhedron
\(A_d\) can be said to be valid for \(\mathbf v\) and certain
vertices are not.  The structure preserving moves are the ones
corresponding to edges of \(A_d\) in the subgraph of the 1-skeleton
of \(A_d\) spanned by the vertices that are valid for \(\mathbf v\).

This makes the subgraph spanned by the vertices that are valid for
\(\mathbf v\) an interesting subgraph to look at.  We refer to it as
the color graph of \(\mathbf v\).  It follows from the results and
discussion in \Gravier{} that a color graph is either
connected or has no edges.

The diameter of \(A_d\) is well understood.  From
\cite{pournin:assocdiam}, we know that the
diameter\index{associahedron!diameter}\index{diameter!of
assciahedron} is \(2d-4\) for \(d\ge 10\).  The diameter of the
color graphs of the various \(\mathbf v\) can be considerably
higher, and we show examples where the diameter is \(\lfloor
d^2/4\rfloor\).  Interestingly, the examples involve pairs of
vertices of the associahedron considered in \Zeilberger.

We call the complement of the color graph of \(\mathbf v\) the zero
set of \(\mathbf v\) for reasons that we will explain later.  We
raise questions about the relative sizes and structures of a color
graph and its complementary zero set.

We present examples that show that the structure of the color graphs
can be rather complicated, independent of the examples of large
diameter.  A prime map with a rigid coloring represents two isolated
vertices in a color graph.  However, if a rigidly colored prime map
is combined with a flexibly colored prime map, then the coloring of
the resulting non-prime map is flexible.  The ``disolving'' of the
rigidity of the rigid factor by the presence of the flexible factor
can be very inefficient and we give an example of how this takes
place.

\subsubsection{Enumeration}\mylabel{IntroEnumSec}

The fact that we consider all proper, 4-colorings gives us an
opportunity to try to count them.  It is known that counting the
colorings of a map is \(\#P\)-complete \cite{MR2201454}, but that
does not make counting them impossible.  We give examples of maps
with very large numbers of colorings and ask whether these maps have
the largest numbers of colorings for maps of a given size.  One
question has been since answered by P.~D.~Seymour
\cite{pdseymour:biwheel}.  See the note after Question
\ref{BiwheelQuestion}.

We also characterize certain structures and this allows us to count
them as well.  

A proposition in \Zeilberger{} gives properties of any color
cycle \(\mathbf v\) that has a non-empty color graph.  We supply a
converse to this proposition and use it to count those color cycles
with non-empty color graphs.

We can also characterize those color cycles that have non-empty
color graphs but which are rigid.  This lets us count such color
cycles and thus the complementary set of color cycles that are
flexible.

Many questions about enumeration are raised.

\subsubsection{Other}\mylabel{IntroOtherSec}

It is well known that coloring behavior changes as the genus of the
surface containing a map goes up.  The effects of this on move
sequences is discussed in \EliaOne, and we give a brief
discussion on its effects on the current paper's point of view.  In
particular, there is a larger group of Thompson's known as \(V\) and
we show that there is no counterpart for \(V\) of the statement that
``all elements of \(F\) have a coloring.''  Interestingly, the
standard example on the torus (an embedding of \(K_7\)) does not
supply an example, but the standard example on the projective plane
(an embedding of \(K_6\)) does.

\subsection{In the paper}\mylabel{SectionListSec}

We describe the flow of the paper and add a bit of detail to the
above list.

Section \ref{WhitneySec} establishes definitions and notation.

Trees are the central objects of the paper and all definitions of
trees and related structures are given in Section \ref{TreeSec}.
The machinery needed to define the faces of the associahedra are
introduced.  The action of the dihedral groups on trees is defined.

The associahedra are introduced in Section \ref{AssocSec}.  Vertices
are binary trees, and a rotation is defined as a move along one edge
and thought of as a modification of the tree at one end of the edge
to make it look like the tree at the other end.

Color (not mentioned in the sections before) is added in Section
\ref{ColorSec}.  Since we view maps as made of pairs of trees, there
is extensive discussion of colorings of trees and tree pairs.  The
easy consequence of Whitney's theorem that the 4CT is equivalent to
the colorability of pairs of trees is pointed out.  The orientation
of colors used around a single vertex (called its sign) is
introduced and discussed.

Section \ref{MainColRotSec} combines color with rotations.  The use
of paths along edges to find colorings is introduced and we restate
the conjectures of \Kryu{} and \EliaOne{} to read that every pair of
vertices in an associahedron has an edge path between them that
succeeds in carrying a coloring from the first vertex to the second.
This conjecture implies the 4CT, and the converse follows from the
main result of \Gravier.  We introduce the rigid colorings that are
carried by no edge path.  These correspond to triangulations of the
2-sphere that have proper, vertex 3-colorings.  We also introduce
the notion of a color graph on an associahedron.  Such a graph has
as vertices all trees colored by a single ``vector'' of colors and
the edges are all the rotations between the vertices that stay
within the colorings given by the vector.  By the result of
\Gravier, this graph is either connected or has no edges.

Thompson's group \(F\) and its extension \(E\) are described in
Section \ref{GroupSec}.  Algebra is combined with rotations to
represent elements of \(F\) as edge paths in an associahedron.

Section \ref{GpOpSec} goes as far as it can combining color with the
multiplication on Thompson's group.  The equivalence of the 4CT to
the colorability of tree pairs is given the immediate translation
into the equivalence of the 4CT to the colorability of elements of
\(F\).  Since \(F\) is finitely generated, this can be reformulated
as an equivalence, Theorem \ref{FGenEquivThm}, that sounds simpler
but that adds little true knowledge.  A compatibility lemma is given
that finds cooperation between the multiplication and colorings in
certain situations and is exploited to show that all elements in a
certain submonoid of \(F\) have colorings and that some of them have
unique colorings.  These results have other proofs.  They also unify
a number of the results in \Zeilberger.

Section \ref{RigidSec} covers thoroughly the uncooperative
3-colorings.  These correspond to a nice subgroup of \(F\), are easy
to characterize, and thus, in Section \ref{EnumSec}, easy to count.

Section \ref{ZeroSetSec} further considers color graphs and also
looks at their complements.  We show that color graphs (which live
in the graph of edges of associahedra) can be more contorted in
appearance than the full edge graph of the associahedron that they
live in.

Section \ref{SignStructSec} develops the criterion that detects the
ability of an edge path to successfully drag a coloring from one
tree to another.  It is similar in spirit, but different in detail
from the criterion in \Carpentier.  This test for success builds,
from the edge path, a signed graph which will be balanced if and
only the edge path gives a coloring.  Prime maps have connected
signed graphs, and (up to permuting the colors) a successfull edge
path leads to a unique coloring of the prime map.  Colorings of
non-prime graphs are combinations of colorings of prime graphs.  The
combination of a face 4-coloring using four colors on one prime
factor with a face 3-coloring on another prime factor results in a
coloring that can still be created by an edge path in spite of the
fact that the face 3-coloring could not have been so created on its
own.  When this happens, the signed graph constructed from the path
has components that correspond to more than one prime factor.
Restatements of earlier conjectures are given and the discussion is
given that brings the group \(E\) into the picture.  Conjectures
about the relevance of the group \(E\) are given.

The remaining sections cover material that is not part of the main
flow of the paper.

Section \ref{AcceptColSec} separates those color arrangements that
color no trees from those that color some trees by characterizing
the latter.  (Except for miniscule trees, there are no arrangements
that color all trees.)  This supplies a converse to a proposition in
\Zeilberger.

In the discussion about face 3-colorings, a pattern for coloring
all trees comes up.  This does not automatically color tree pairs,
since a coloring of a tree pair has a condition to be met (colors
must agree on the leaves of the two trees).  The pattern considered
in Section \ref{RigidSec} produces face 3-colorings when a match
occurs.  This pattern also leads to the subgroup \(F_4\) of \(F\)
that we mentioned above.  Section \ref{PatternSec} briefly looks, in
some generality, at conditions that make a pattern lead to subgroups
of \(F\).  Section \ref{PosPatSec} looks at a particularly simple
pattern that seems to lead to a not so simple subgroup of \(F\).
Section \ref{PosNbhdSec} demonstrates the limit of usefulness of the
pattern of Section \ref{PosPatSec} to the 4CT.

Section \ref{HGV} takes a brief look at standard examples of maps on
the torus and projective plane from the point of view of the current
paper and relates the examples to another of Thompson's groups known
as \(V\).  As expected, the situation is quite different.  In
particular, the obvious generalization to \(V\) of the statement
that all elements of \(F\) have a coloring is false.

Section \ref{EnumSec} gathers all our results and questions about
counts.  Formulas are given that count the colorings that are valid
for at least one tree, that count colorings that lead to face
3-colorings of maps, and that count colorings that lead to ``true''
face 4-colorings of maps (all four colors are used).  From our
computer calculations, we also noticed that certain maps have many
more colorings than others.  One map in particular (the biwheel from
the dual point of view) seemed to have approximately twice as many
as any other.  We raise the question of which maps of a given size
have the most colorings and have candidates for the top four counts
of colorings for a given size.  Our candidate for the top count and
the fact that it corresponds precisely to the biwheel has been since
proven correct by P.~D.~Seymour \cite{pdseymour:biwheel}.  The
section ends with other questions of enumeration motivated by other
parts of the paper.

Section \ref{EndSec} raises questions that fit nowhere else.

\subsection{Thanks} The authors extend their thanks to Jeff Nye for
all of his help and ideas with the programming.

\section{Starting from Whitney's theorem}\mylabel{WhitneySec}

We establish some notation to accompany the setting described in
Section \ref{IntroSettingSec}.  Certain terms are defined carefully.
However, others are to be interpreted intuitively by the reader
during this section and they will be defined more carefully in the
remaining sections of the paper.

\subsection{The equator}

For a map \(M\) in \(\mathfrak W\) we know the dual graph of \(M\)
in \(S^2\) has a Hamiltonian circuit \(E\) that we will call the
equator.  We can take this circuit to be the circle \(z=0\) in the
2-sphere \(\{(x,y,z)\mid x^2+y^2+z^2=1\}\).  We can also insist that
some point of the intersection of \(M\) with \(E\) is the point
\((0,1,0)\).  From now on we will insist that an \(M\in \mathfrak
W\) is arranged in this way on the 2-sphere.

We call the points in \(S^2\) with \(z\ge0\), the northern
hemisphere, and the points with \(z\le0\) the southern hemisphere.
From Section \ref{IntroSettingSec}, we know that the intersection of
\(M\) with each hemisphere is a binary tree.

For an \(M\in \mathfrak W\), we will usually use \(D\) to denote the
intersection of \(M\) with the northern hemisphere, and \(R\) to
denote the intersection of \(M\) with the southern hemisphere.  The
reason for these letters will be clarified in Section
\ref{FDefsSec}.

If the dual graph of \(M\) in \(S^2\) has no parallel edges, then we
follow Loday in \cite{Loday:YY} and call \(M\) {\itshape
prime}\index{prime!map}\index{map!prime}.  Such maps are called
irreducible in \Zeilberger.

\subsection{Rooting and orienting the trees}\mylabel{WhitMotivSec}

Given \(M\in \mathfrak W\) with \(D\) and \(R\) as above, we declare
that the common root of \(D\) and \(R\) is the point \((0,1,0)\).
The leaves, which are common to both \(D\) and \(R\), are all other
points in \(M\) on the equator.  Note that in each of \(D\) and
\(R\), the root has degree 1 and the root and leaves together
account for all the vertices of degree 1.

We then embed each of \(D\) and \(R\) (separately) in the
\(xy\)-plane by vertical projection.  If the fact that the two
images intersect is bothersome, then \(D\) can be moved two units in
the negative \(x\) direction, and \(R\) can be moved two units in
the positive \(x\) direction.  The purpose of the projection is to
discuss orientation in a well defined way.

With \(D\) and \(R\) projected to the \(xy\)-plane, we associate to
each vertex \(v\) of degree 3 a cyclic order on the edges that
impinge on \(v\) by using the counterclockwise order on the edges as
viewed from above the plane.  We also get a linear order on the
leaves.  The order is obtained from a counterclockwise walk (as viewed
from above) that starts and ends at the root around the circle of
radius 1 that contains the root and leaves.

Since the vertices that are neither root nor leaf are all of degree
3, we can call each of \(D\) and \(R\) a binary tree.  Thus we have
a process that takes each \(M\in \mathfrak W\) to a pair \((D,R)\)
of finite, rooted, oriented, binary trees that have the same number
of leaves.  We will later be very careful with definitions, but the
reader can infer enough definitions at this point for us to
continue.

Below left is a map \(M\) with equator (dotted line) labeled \(E\)
and chosen root labeled \(*\).  To the right are the two trees \(D\)
and \(R\) obtained from \(M\).
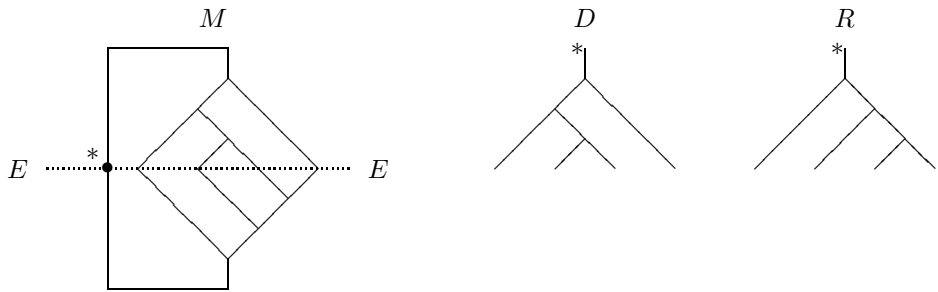
\[
\xy
(10,20)*{M};
(0,0); (12,12)**@{-}; (24,0)**@{-}; (12,-12)**@{-}; (0,0)**@{-};
(8,8); (20,-4)**@{-}; (12,4); (8,0)**@{-}; (16,-8)**@{-};
(-12,0); (28,0)**@{.};
(12,12); (12,16)**@{-}; (-4,16)**@{-}; (-4,-16)**@{-}; 
(12,-16)**@{-}; (12,-12)**@{-};
(-16,0)*{E}; (32,0)*{E}; (-4,0)*{\bullet}; (-6,2)*{*};
(10,-20)*{};
\endxy
\qquad\qquad
\xy
(12,20)*{D}; (12,12); (12,16)**@{-}; (11,16)*{*};
(0,0); (12,12)**@{-}; (24,0)**@{-};
(8,8); (16,0)**@{-}; (12,4); (8,0)**@{-}; 
\endxy
\qquad\quad
\xy
(4,20)*{R}; (4,12); (4,16)**@{-}; (3,16)*{*};
(-8,0); (4,12)**@{-}; (8,8)**@{-};
(0,0); (8,8)**@{-}; (12,4)**@{-};
(8,0); (12,4)**@{-}; (16,0)**@{-};
\endxy
\]

Note that a pair of trees \((D,R)\) that comes from a map \(M\in
\mathfrak W\) comes with a one-to-one correspondence from the leaves
of \(D\) to the leaves of \(R\) that respects the linear orders on
the leaves.  If \(D\) and \(R\) are regarded as subsets of \(M\) (or
the embedding of \(M\) in \(S^2\)), then this one-to-one
correspondence is just the identity map.  If this correspondence is
extended to the roots, then we also have a one-to-one correspondence
from the vertices of degree 1 in \(D\) to the corresponding set in
\(R\).  We will refer to these correspondences often.

From now on a pair\index{pair!of finite trees}\index{tree!finite
pair} of finite trees \((D,R)\) will always carry the assumption
that the two trees have the same number of leaves.  We will remind
the reader of this from time to time.

\subsection{Reversing the process}\mylabel{TreePairToMapSec}

Conversely, we can take a pair \((D,R)\) of finite, rooted,
oriented, binary trees with the same number of leaves and create a
map in \(\mathfrak W\).  The method of creating a map should be
clear.  We will sketch the steps, and definitions made later will
fill in details.

The orientations of the vertices induce a cyclic order of the
vertices of degree 1 (all leaves and the root), which becomes a
linear order if the root is declared to come first.  

We map \(D\) to the northern hemisphere of \(S^2\) and \(R\) to the
southern hemisphere so that only the leaves and roots end up on the
equator, so that the two roots go to the point \((0,1,0)\), so that
the leaves of \(D\) occupy the same points as the leaves of \(R\),
and so that a counterclockwise walk (viewed from above) around the
equator starting at the root visits the leaves of the two trees in
an order that agrees with the order induced from the orientations.

Note that the number of faces is equal to the number of vertices
on either \(D\) or \(R\) of degree 1.

The failure to always get a dual with no parallel edges is more
easily illustrated than discussed.  The two trees shown below will
produce a map that fails this property.  Recall that \(D\) will live
on the northern hemisphere so that its vertical projection to the
\(xy\)-plane looks as shown below and that \(R\) will live in the
southern hemisphere so that its vertical projection to the
\(xy\)-plane looks as shown below.

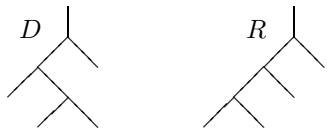
\[
\xy
(-1,5)*{D};
(-4,-4); (0,0)**@{-}; (4,-4)**@{-};
(0,-8); (4,-4)**@{-}; (8,-8)**@{-};
(0,0); (4,4)**@{-}; (8,0)**@{-};
(4,4); (4,8)**@{-};
\endxy
\qquad\qquad
\xy
(-1,5)*{R};
(4,-4); (0,0)**@{-}; (-4,-4)**@{-};
(0,-8); (-4,-4)**@{-}; (-8,-8)**@{-};
(0,0); (4,4)**@{-}; (8,0)**@{-};
(4,4); (4,8)**@{-};
\endxy
\]

We follow \cite{Loday:YY} again and say that a pair of trees is
{\itshape prime}\index{prime!finite tree pair}\index{tree!finite
pair!prime} if the process above produces a map whose dual has no
parallel edges.

It is just as easy to build a pair of trees that creates a map that
has a separating circuit of length 3.  We leave this for the reader.

The working class of maps for the rest of this paper are the maps in
the class \(\mathfrak W\).

\subsection{Triangulations}\mylabel{DualTriSec}

Dual to the pairs of trees are the triangulated polygons described
in Section \ref{IntroSettingSec}.  The equator \(E\) breaks the
triangulation dual to a map \(M\in \mathfrak W\) into two
triangulated
polygons\index{triangulated!polygon!pair}\index{polygon!triangulated!pair}\index{pair!triangulated!polygon}
whose common boundary is \(E\).

This dual view works with pairs \((P_1, P_2)\) for which the \(P_i\)
are triangulated polygons with the same number of boundary edges.
Each \(P_i\) is triangulated using no vertices other than the
vertices of the polygon, and there is a cyclic order preserving
bijection between the vertices of the two polygons.  The number of
edges in the boundary of each \(P_i\) is the number of faces of
the map \(M\).  This view makes it clear that the dihedral group of
order \(2n\) acts on the set of maps in \(\mathfrak W\) with \(n\)
faces.  See Section \ref{RootShiftSec}.

\section{Trees}\mylabel{TreeSec}

We do a lot with trees, and we use this section to define the terms
and structures that we need.  Some of the oddness in our definitions
is motivated by the discussion in Section \ref{WhitneySec}.

\subsection{General trees}

A {\itshape tree}\index{tree!definition} is a connected graph with
no circuits.  A finite tree\index{tree!finite} is a tree with
finitely many vertices.  A locally finite tree\index{tree!locally
finite} is a tree in which the degree of every vertex is finite.

Every tree must have a root\index{root!of tree}\index{tree!root},
and we will be very restrictive about where the root can be.  For
us, the root of a tree must have degree 1.  We will emphasize this
by saying that such a tree is {\itshape
rooted}\index{rooted!tree}\index{tree!rooted}.

Each vertex in a tree has a distance to the root defined as the
length (in edges) of the unique simple walk from the vertex to the
root.  If \(v\) is a vertex in a tree with a root, then the
{\itshape children}\index{children!in tree}\index{tree!children} of
\(v\) are those vertices adjacent to \(v\) that are farther from the
root than \(v\) is.  The vertex adjacent to the root is the only
child of the root, and we insist that the root have a child.  The
{\itshape leaves}\index{leaves!tree}\index{tree!leaves} of a tree
are the vertices with no children.  The {\itshape trivial
tree}\index{trivial!tree}\index{tree!trivial} is the unique tree with
one leaf.

The inverse of ``child'' is ``parent.''\index{parent!in
tree}\index{tree!parent} Every vertex except the root has a unique
parent, and the root has no parent.  The transitive closure of
``child'' is ``descendant''\index{descendant!in
tree}\index{tree!descendant} and the transitive closure of
``parent'' is ``ancestor.''\index{ancestor!in
tree}\index{tree!ancestor} The root is an ancestor of all other
vertices in a tree.

In a {\itshape binary tree}\index{binary!tree}\index{tree!binary}
every vertex that is neither root nor leaf has degree 3.

If \(T\) is a tree, then the {\itshape internal
vertices}\index{tree!internal vertex}\index{vertex!of tree!internal}
of \(T\) are the vertices of degree greater than 1.  The {\itshape
internal edges}\index{tree!internal edge}\index{edge!internal in
tree} are the edges where both endpoints are internal vertices.
While not standard terminology, we need to call the edges that are
not internal edges the {\itshape external edges}\index{edge!external
in tree}\index{tree!external edge}.  It will be convenient to
separate the external edges into the {\itshape root
edge}\index{tree!root edge}\index{edge!root in tree}, the edge that
impinges on the root, and the {\itshape leaf
edges}\index{tree!leaf!edge}\index{edge!leaf in tree}, those edges
that impinge on leaves.

The root and the root edge are motivated by Section
\ref{WhitMotivSec}, and are there for sometimes technical, sometimes
formal, and sometimes practical reasons.  Since the root and root
edge must always be there, it adds no information in a drawing to
include them.  So the occasional drawing of a tree with no root and
root edge shown will be assumed to be ``augmented'' by these items.

The root edge will be inconvenient at times.  However it would be
inconvenient more often if it were omitted.

\subsection{Orders in trees}

A tree is {\itshape locally ordered}\index{order!local in
tree}\index{tree!locally ordered}\index{local order!in tree} if for
every vertex \(v\) there is a cyclic order defined on the set of
edges that impinge on \(v\).  The cyclic order on the edges gives a
cyclic order on the vertices adjacent to \(v\).  For any vertex
\(v\) other than the root, this leads to a linear order on the
children of \(v\) by starting the linear order with the vertex that
follows the parent in the cyclic order.  The children of the root
have a linear order since the number of children of the root is 1.

If a tree \(T\) is embedded in the plane, then the embedding induces
a local order\index{local order!from planar embedding} by using the
counterclockwise cyclic order of the edges that impinge on a vertex
\(v\) as described in Section \ref{WhitMotivSec}.

Conversely, every finite, ordered tree can be embedded in the plane
so that the local order induced from the embedding agrees with the
given order.

An embedding can be referred to as a drawing\index{drawing!of
tree}\index{tree!drawing}.  We will draw trees with the roots at the
top.  See the paragraphs at the bottom of Page 8 and top of Page 9
in \cite{MR1898414} for a discussion of this choice.

An isomorphism of locally ordered trees is required to respect both
root and local order.  It follows that the only automorphism of a
locally ordered tree is the identity.

A locally ordered tree can be given the prefix\index{prefix order!on
tree}\index{tree!prefix order} total order on all the vertices.  This
is the unique linear order that restricts to the local linear order
on the children of each vertex and that for two children \(v\) and
\(w\) of a single vertex, if \(v<w\), then the descendants of \(v\)
precede \(w\) which in turn precedes the descendants of \(w\).  The
restriction of this total order to the leaves gives a total order on
the leaves.

In a locally ordered binary tree, every vertex with two children has
the first child designated as the left child and the other child is
the right child.  A locally ordered binary tree can also be given
the infix\index{infix order!on tree}\index{tree!infix order} total
order.  This is the unique linear order that restricts to the local
linear on the children and that places each vertex after its left
child and the descendants of the left child and before its right
child and the descendants of the right child.  The restriction of
the infix order to the leaves is the same as the restriction of the
prefix order to the leaves.

The induced order on the leaves (from prefix or infix order) will
usually be referred to as the left-right\index{left-right order!on
leaves}\index{leaves!left-right order} order on the leaves since,
with the tree drawn with the root at the top, visiting the leaves
moving left to right in the drawing follows the order.

We will often have need to refer to the \(i\)-th leaf of a finite,
locally ordered tree, and this will refer to the \(i\)-th leaf in
the left-right order of the leaves, with the count starting from 1.

\subsection{Standing assumptions on trees}

In the rest of this paper, all trees will be locally finite, rooted,
locally ordered, and no vertex will have degree 2.  These
assumptions will not be repeated in statements.

\subsection{Subtrees}

A subtree\index{subtree} \(S\) of a tree \(T\) is a subgraph that is
also a tree and that satisfies the extra condition that if \(v\) is
a vertex of \(S\) other than the root of \(S\), then either all the
children of \(v\) in \(T\) are in \(S\), or none of the children of
\(v\) in \(T\) are in \(S\).  The root of \(S\) must have only one
child.

A local order on a tree is inherited by a subtree.

Since we assume \(T\) has a root, there is a unique vertex in \(S\)
that is closest to the root of \(T\) and we will insist that this
vertex be the root of \(S\).

\subsection{Standard model for binary trees}\mylabel{MathcalTSec}

Binary trees can be given a very standard\index{binary!tree!standard
model}\index{standard model!binary tree} structure that makes them
all subtrees of a common structure.  This will be convenient in many
places.

The vertex set of the infinite, locally ordered, binary tree
\(\mathcal T\)\index{\protect\(\protect\mathcal T\protect\)}
consists of a single special element that we will denote \(*\),
together with the set \(\{0,1\}^*\) of all finite words (including
the empty word \(\emptyset\)) in the alphabet 0 and 1.  In
\(\{0,1\}^*\), we take {\itshape concatenation, prefix, suffix} to
have their usual meanings.  Concatenation will be denoted by
juxtaposition.  The left child of a vertex \(u\) in \(\{0,1\}^*\)
will be \(u0\) and the right child of \(u\) will be \(u1\).  The
only child of \(*\) will be \(\emptyset\).  Words like {\itshape
parent, ancestor, child, descendant} will have their usual meanings.
We order the children of \(u\in \{0,1\}^*\) by \(u0<u1\).  The edges
of \(\mathcal T\) are all connections from parent to child.
Declaring \(*\) to be the root makes \(\mathcal T\) a rooted,
locally ordered, binary tree.  The infix order will be used on
\(\mathcal T\) at times.

Throughout the paper we will need to discuss words.  We will use
notation similar to regular expressions\index{regular expressions}
and write \(0^n\) for a string of \(n\) zeros, \(0^+\) for a
non-empty finite string of zeros and \(0^*\) for a possibly empty
finite string of zeros.  Thus \(1^+0^+1^*\) represents at least one
1, followed by at least one 0, and then possibly followed (or not)
by more ones.  The expression \((abc)\) refers to the string
\(abc\), and \((abc)^n\), \((abc)^*\), \((abc)^+\) refer,
respectively, to \(n\) appearances, zero or more appearances, and
one or more appearances of the string \(abc\).  Lastly \([xyz]\)
refers to a choice of one of \(x\), or \(y\), or \(z\), so that
\([xyz]^n\), \([xyz]^*\), \([xyz]^+\) refer, respectively, to \(n\),
zero or more, one or more appearances of independent choices from
\(x\), or \(y\), or \(z\).

Recall that all trees are are locally ordered.  It is standard that
any finite, binary tree can be realized uniquely as a finite subtree
of \(\mathcal T\) so that its root is \(*\), so that its vertex set
is closed under the taking of prefixes and for which \(u0\) is in
\(T\) if and only if \(u1\) is in \(T\).  It will be convenient to
view all finite, binary trees this way.

Each \(v\in \{0,1\}^*\) is a vertex of \(\mathcal T\) and of any
tree in \(\mathcal T\) that contains \(v\).  As a word, it also
contains the information on how to get to \(v\) as a walk from
\(\emptyset\).  As such it can be thought of as an
``address.''\index{address!of vertex in tree}
This makes for the odd statement that the word \(v\) is the address
for the vertex \(v\), but at times this will be convenient.

Subtrees are defined as for subtrees of general trees.  There are
some special notations for certain subtrees.

If \(T\) is a binary tree and \(v\) is a vertex in \(T\) with parent
\(w\), then \(T_v\)\index{\protect\(T_v\protect\)} will denote the
subtree \(\{vu\mid vu\in T\} \cup \{w\}\).  The root of \(T_v\) is
\(w\) and every leaf of \(T_v\) is a leaf of \(T\).  If \(T\) is a
binary tree, then \(T_0\) will be referred to as the {\itshape left
subtree}\index{subtree!left}\index{left!subtree} of \(T\) and
\(T_1\) will be referred to as the {\itshape right
subtree}\index{subtree!right}\index{right!subtree} of \(T\).  We will
be consistent with the words ``tree'' and ``subtree'' and insist
that a finite, binary tree in \(\mathcal T\) have \(*\) as its root,
while subtrees can have any vertex in \(\mathcal T\) as its root.

The advantage of putting finite binary trees in \(\mathcal T\) is
that if \(T\) and \(S\) are binary trees, then so are \(S\cap T\)
and \(S\cup T\).  Recall that trees are rooted at \(*\).

\subsection{Inducting with trees}\mylabel{TreeIndSec}

We will frequently appeal to either of two inductive arguments
without giving details.

Finite trees are well founded\index{well
founded!order}\index{order!well founded} (every set has a minimal
element where vertices of a tree are partially ordered by declaring
descendants less than ancestors) and as such can have constructions
or statements about them argued by well founded\index{well
founded!induction}\index{Noetherian!induction}\index{induction!well
founded}\index{induction!Noetherian}  (Noetherian)
induction which lets a statement or construction be applied to a
vertex if it applies to all of its descendants.  We will often
simply say that something can be defined or argued inductively.  For
example, {\itshape height} can be defined inductively by declaring
the leaves of a finite tree to have height zero, and any vertex
whose descendants all have defined heights is defined to have height
one greater than the largest height of its children.

\newcommand{\caret}{\hbox to .06 in{}
\widehat{\hbox to .01 in{}}
\hbox to .05 in{}
}

Another argument is based on a binary operation on binary trees.
Given two binary trees, \(S\) and \(T\), there is a unique binary
tree \(S\caret T\)\index{\protect\(S\caret T\protect\)} whose left
subtree is isomorphic to \(S\) and whose right subtree is isomorphic
to \(T\).  Every binary tree that has more than one leaf is uniquely
of the form \(S\caret T\).  Since \(S\caret T\) has more vertices
than either of \(S\) or \(T\), we can induct on the size of a binary
tree.

\subsection{Projections}\mylabel{ContrSec}

Since we will study pairs of binary trees, we will be comparing
structures of trees.  The following will help organize trees into
related groups of various sizes.

If \(T\) is a finite, binary tree then a
projection\index{projection!of tree}\index{tree!projection} of \(T\)
can be defined based on certain data.  We say that two subtrees
\(S_1\) and \(S_2\) of \(T\) are {\itshape edge disjoint}\index{edge
disjoint!subtrees} if removing the root and root edge from each of
\(S_1\) and \(S_2\) leaves them with no edge in common.

The data for the projection consists of a finite set \(T_1, \ldots,
T_n\) of subtrees in \(T\), which are pairwise edge disjoint.  To
avoid trivialities, we require that each \(T_i\) have at least three
leaves.  

We form the {\itshape projection determined by} \(T_1, \ldots, T_n\)
as follows.  For each \(i\) with \(1\le i\le n\), let \(w_i\) be the
child of the root of \(T_i\).  We remove from \(T\) all the
non-root edges of \(T_i\) and all the internal vertices of \(T_i\)
except for \(w_i\).  Then for each leaf of \(T_i\), we add an edge
that connects that leaf to \(w_i\).

We show an example below.

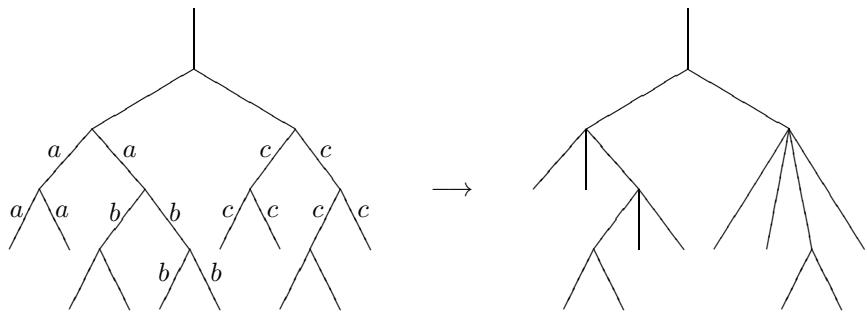
\[
\xy
(24.5,16); (24.5,24)**@{-};
(0,-8); (4,0)**@{-}; (11,8)**@{-}; (24.5,16)**@{-}; 
(38,8)**@{-}; (44,0)**@{-}; (48,-8)**@{-};
(1,-3)*{a}; (6,5)*{a}; (16,5)*{a}; (7,-3)*{a};
(8,-8); (4,0)**@{-}; 
(8,-16); (12,-8)**@{-}; (18,0)**@{-}; (24,-8)**@{-}; (28,-16)**@{-};
(20.5,-11)*{b}; (27.5,-11)*{b};
(14,-3)*{b}; (22,-3)*{b};
(18,0); (11,8)**@{-}; (16,-16); (12,-8)**@{-}; (20,-16); (24,-8)**@{-};
(28,-8); (32,0)**@{-}; (38,8)**@{-};
(36,-16); (40,-8)**@{-}; (44,0)**@{-};
(29,-3)*{c}; (34,5)*{c}; (42,5)*{c}; (35,-3)*{c};
(41,-3)*{c}; (47,-3)*{c};
(36,-8); (32,0)**@{-}; (44,-16); (40,-8)**@{-};
\endxy
\qquad \longrightarrow \qquad
\xy
(24.5,16); (24.5,24)**@{-};
(4,0); (11,8)**@{-}; (24.5,16)**@{-}; 
(38,8)**@{-}; (48,-8)**@{-};
(11,0); (11,8)**@{-}; 
(8,-16); (12,-8)**@{-}; (18,0)**@{-}; (24,-8)**@{-};
(18,0); (11,8)**@{-}; (16,-16); (12,-8)**@{-}; (18,-8); (18,0)**@{-};
(28,-8); (38,8)**@{-};
(37,-16); (41,-8)**@{-}; (38,8)**@{-};
(35,-8); (38,8)**@{-}; (45,-16); (41,-8)**@{-};
\endxy
\]

In the illustration above, three subtrees are used.  Each of the
three subtrees is pictured in the tree on the left by having all of
its non-root edges given a common label.  Edges labeled \(a\) are of
a subtree with one internal edge and three leaves.  Edges labeled
\(b\) are of a similar subtree (but not isomorphic as an ordered
subtree).  The edges labeled \(c\) are of a subtree with two
internal edges and four leaves.  The subtree with edges labeled
\(a\) is edge disjoint from the subtree with edges labeled \(b\)
because of the technicalities in the definition of edge disjoint.

Note that a projection determined by subtrees \(T_1, \ldots, T_n\)
can be thought of as the result of shrinking all the internal edges
of the \(T_i\) to have length zero.

A projection inherits a local order from the infix order of the
original tree.

Note that if \(T\) is any finite tree, then there is a finite set of
binary trees that have \(T\) as a projection.  This will be used in
Section \ref{AssocSec}.

\subsection{Root shifts and reflections}\mylabel{RootShiftSec}

An operation on trees is easiest to describe using the dual view.  A
rooted tree embedded in the plane is naturally dual to a
triangulated polygon.  A finite binary tree \(T\) with \(n\)
external vertices (\(n-1\) leaves and one root) can be embedded in a
regular, triangulated \(n\)-gon \(P\) in the plane so that it is
dual to the triangulation, so that the local order is preserved, and
with one external vertex of \(T\) in the interior of each edge of
\(P\).  We also insist that a horizontal edge in the polygon is at
the top and that the root of \(T\) is in this top edge.  Given \(P\)
and \(T\) as described, we say that \(P\) is the triangulated
\(n\)-gon dual\index{dual!polygon to tree}\index{polygon!dual to
tree} to \(T\) and that \(T\) is the tree dual\index{dual!tree to
polygon}\index{tree!dual to polygon} to the triangulated \(n\)-gon
\(P\).

Below are two examples.  Each line gives a tree in a hexagon, then
the dual triangulation, and finally the tree and triangulation
together.
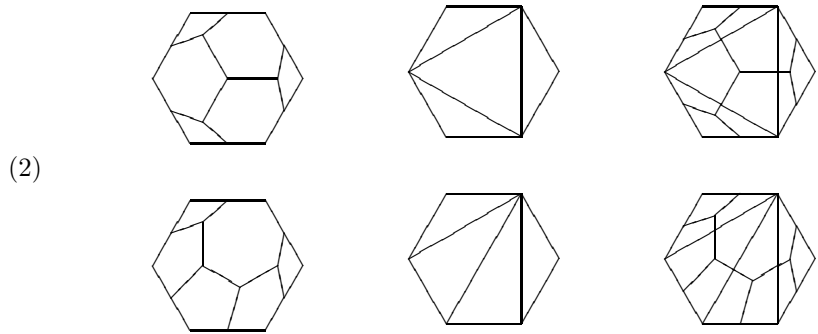
\mymargin{TriPolyExmpls}\begin{equation}\label{TriPolyExmpls}
\begin{split}
&\xy
(-10,0); (-5,8.66)**@{-}; (5, 8.66)**@{-}; (10,0)**@{-};
(5,-8.66)**@{-}; (-5,-8.66)**@{-}; (-10,0)**@{-};
(0,0); (-3.33, 5.77)**@{-};
(0,0); (6.66,0)**@{-};
(0,0); (-3.33, -5.77)**@{-};
(-7.5,4.33); (-3.33, 5.77)**@{-}; (0,8.66)**@{-};
(-7.5,-4.33); (-3.33, -5.77)**@{-}; (0,-8.66)**@{-};
(7.5,4.33); (6.66,0)**@{-}; (7.5,-4.33)**@{-};;
\endxy
\qquad\qquad
\xy
(-10,0); (-5,8.66)**@{-}; (5, 8.66)**@{-}; (10,0)**@{-};
(5,-8.66)**@{-}; (-5,-8.66)**@{-}; (-10,0)**@{-};
(5,8.66)**@{-}; (5,-8.66)**@{-}; (-10,0)**@{-};
\endxy
\qquad\qquad
\xy
(-10,0); (-5,8.66)**@{-}; (5, 8.66)**@{-}; (10,0)**@{-};
(5,-8.66)**@{-}; (-5,-8.66)**@{-}; (-10,0)**@{-};
(5,8.66)**@{-}; (5,-8.66)**@{-}; (-10,0)**@{-};
(0,0); (-3.33, 5.77)**@{-};
(0,0); (6.66,0)**@{-};
(0,0); (-3.33, -5.77)**@{-};
(-7.5,4.33); (-3.33, 5.77)**@{-}; (0,8.66)**@{-};
(-7.5,-4.33); (-3.33, -5.77)**@{-}; (0,-8.66)**@{-};
(7.5,4.33); (6.66,0)**@{-}; (7.5,-4.33)**@{-};;
\endxy
\\ \\
&\xy
(-10,0); (-5,8.66)**@{-}; (5, 8.66)**@{-}; (10,0)**@{-};
(5,-8.66)**@{-}; (-5,-8.66)**@{-}; (-10,0)**@{-};
(-7.5,4.33); (-3.33, 5.77)**@{-}; (0,8.66)**@{-};
(7.5,4.33); (6.66,0)**@{-}; (7.5,-4.33)**@{-};;
(-3.33, 5.77); (-3.33,0)**@{-}; (-7.5,-4.33)**@{-};
(-3.33,0); (1.66, -2.88)**@{-}; (6.66,0)**@{-};
(0,-8.66); (1.66, -2.88)**@{-};
\endxy
\qquad\qquad
\xy
(-10,0); (-5,8.66)**@{-}; (5, 8.66)**@{-}; (10,0)**@{-};
(5,-8.66)**@{-}; (-5,-8.66)**@{-}; (-10,0)**@{-};
(5,8.66)**@{-}; (5,-8.66)**@{-};
(5,8.66); (-5,-8.66)**@{-};
\endxy
\qquad\qquad
\xy
(-10,0); (-5,8.66)**@{-}; (5, 8.66)**@{-}; (10,0)**@{-};
(5,-8.66)**@{-}; (-5,-8.66)**@{-}; (-10,0)**@{-};
(5,8.66)**@{-}; (5,-8.66)**@{-};
(5,8.66); (-5,-8.66)**@{-};
(-7.5,4.33); (-3.33, 5.77)**@{-}; (0,8.66)**@{-};
(7.5,4.33); (6.66,0)**@{-}; (7.5,-4.33)**@{-};;
(-3.33, 5.77); (-3.33,0)**@{-}; (-7.5,-4.33)**@{-};
(-3.33,0); (1.66, -2.88)**@{-}; (6.66,0)**@{-};
(0,-8.66); (1.66, -2.88)**@{-};
\endxy
\end{split}
\end{equation}
The trees corresponding to these examples are drawn below in the
usual way.
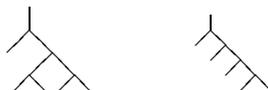
\[
\xy
(0,0); (2,2)**@{-}; (5,5)**@{-}; (8,2)**@{-}; (10,0)**@{-};
(2,2); (4,0)**@{-};
(8,2); (6,0)**@{-};
(5,5);(2,8)**@{-};(-1,5)**@{-}; (2,8); (2,11)**@{-};
\endxy
\qquad\qquad
\xy
(6,0);(8,2)**@{-};(10,0)**@{-};
(4,2);(6,4)**@{-};(8,2)**@{-};
(2,4);(4,6)**@{-};(6,4)**@{-};
(0,6);(2,8)**@{-};(4,6)**@{-}; (2,8); (2,10)**@{-};
\endxy
\]

If symmetries of the \(n\)-gon \(P\) are now applied and we adhere
to the convention that the top edge contains the root of the tree,
then we get new trees that are isomorphic as unrooted, unordered
trees to the original tree \(T\) but which might or might not be
isomorphic as rooted, ordered trees.  If all possible symmetries are
applied to the top figure in \tref{TriPolyExmpls}, then the set of
trees obtained this way contains only two trees.  These are shown
below.
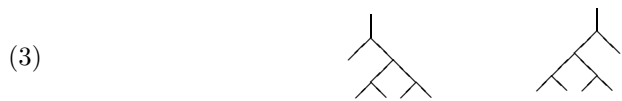
\mymargin{SmallSymExmpl}\begin{equation}\label{SmallSymExmpl}
\begin{split}
\xy
(0,0); (2,2)**@{-}; (5,5)**@{-}; (8,2)**@{-}; (10,0)**@{-};
(2,2); (4,0)**@{-};
(8,2); (6,0)**@{-};
(5,5);(2,8)**@{-};(-1,5)**@{-}; (2,8); (2,11)**@{-};
\endxy
\qquad\qquad
\xy
(0,0); (-2,2)**@{-}; (-5,5)**@{-}; (-8,2)**@{-}; (-10,0)**@{-};
(-2,2); (-4,0)**@{-};
(-8,2); (-6,0)**@{-};
(-5,5);(-2,8)**@{-};(1,5)**@{-}; (-2,8); (-2,11)**@{-};
\endxy
\end{split}
\end{equation}

However if all possible symmetries are applied to the bottom figure
in \tref{TriPolyExmpls}, then the set of tree obtained is larger and
is shown below.
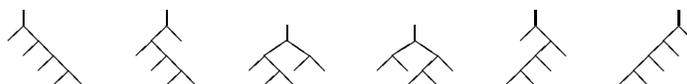
\[
\xy
(6,0);(8,2)**@{-};(10,0)**@{-};
(4,2);(6,4)**@{-};(8,2)**@{-};
(2,4);(4,6)**@{-};(6,4)**@{-};
(0,6);(2,8)**@{-};(4,6)**@{-}; (2,8); (2,10)**@{-};
\endxy
\qquad
\xy
(6,0);(8,2)**@{-};(10,0)**@{-};
(4,2);(6,4)**@{-};(8,2)**@{-};
(2,4);(4,6)**@{-};(6,4)**@{-};
(8,6);(6,8)**@{-};(4,6)**@{-}; (6,8); (6,10)**@{-};
\endxy
\qquad
\xy
(6,1);(8,3)**@{-};(10,1)**@{-};
(4,3);(6,5)**@{-};(8,3)**@{-};
(12,5);(9,7)**@{-};(6,5)**@{-}; (9,7); (9,9)**@{-};
(10,3); (12,5)**@{-}; (14,3)**@{-};
\endxy
\qquad
\xy
(-6,1);(-8,3)**@{-};(-10,1)**@{-};
(-4,3);(-6,5)**@{-};(-8,3)**@{-};
(-12,5);(-9,7)**@{-};(-6,5)**@{-}; (-9,7); (-9,9)**@{-};
(-10,3); (-12,5)**@{-}; (-14,3)**@{-};
\endxy
\qquad
\xy
(-6,0);(-8,2)**@{-};(-10,0)**@{-};
(-4,2);(-6,4)**@{-};(-8,2)**@{-};
(-2,4);(-4,6)**@{-};(-6,4)**@{-};
(-8,6);(-6,8)**@{-};(-4,6)**@{-}; (-6,8); (-6,10)**@{-};
\endxy
\qquad
\xy
(-6,0);(-8,2)**@{-};(-10,0)**@{-};
(-4,2);(-6,4)**@{-};(-8,2)**@{-};
(-2,4);(-4,6)**@{-};(-6,4)**@{-};
(-0,6);(-2,8)**@{-};(-4,6)**@{-}; (-2,8); (-2,10)**@{-};
\endxy
\]

Note that because of the symmetries present in the examples chosen,
all of the above modifications could have been accomplished by a
rotation of the \(n\)-gon \(P\).  This is not always the case.  

We use the term {\itshape root shift}\index{root shift!of
tree}\index{tree!root shift} to refer to the modification of a
finite, binary tree \(T\) obtained by a rotation of its dual
triangulated \(n\)-gon.  We do not use the term rotation since the
word rotation will be used for completely different modification of
a tree.  An alteration of a finite, binary tree obtained by an
orientation reversing symmetry of its dual \(n\)-gon will be called
a {\itshape
reflection}.\index{reflection!tree}\index{tree!reflection}.

\section{The associahedra}\mylabel{AssocSec}

The associahedra\index{associahedron} are cellular complexes whose
vertices are finite, binary trees.  There is one associahedron in
each dimension.  The vertices of the \(d\)-dimensional
associahedron\index{associahedron!dimension}\index{dimension!associahedron}
are all the trees with \(d\) internal edges, \(d+1\) internal nodes,
\(d+2\) leaves and pairs of such vertices (trees) create maps with
\(d+3\) faces.  Thus there are several integer values associated
to one associahedron and will require care to keep track of.  The
books \cite{sternberg:quantum} and \cite{MR1898414} refer to the
\(d\)-dimensional associahedron by \(K_{d+2}\) so their notational
parameter (subscript) is the number of leaves.

We will denote the \(d\)-dimensional associahedron by
\(A_d\).\index{\protect\(A_d\protect\)} 

The number of vertices in \(A_d\) is number of trees with \(n=d+1\)
internal nodes which is
\[
C(n) = \frac{1}{n+1}\binom{2n}{n} = \frac{(2n)!}{n!(n+1)!}
\]
the \(n\)-th Catalan number\index{Catalan number}.  This is sequence
A000108 in the Online Encyclopedia of Integer Sequences.  It will be
relevant that the order of growth of this sequence is \(4^n\).  This
is revisited in Section \ref{TreeCountSec}.

The associahedra are also convex polytopes.  Their role as polytopes
is not needed here.  However the notions of faces, especially edges
(faces of dimension one), will be particularly important.  We will
define faces formally without justifying the use of the word face.
That the \(A_d\) are also cubical complexes does not concern us.

\subsection{The faces}  

A face\index{face!of associahedron}\index{associahedron!face} in
\(A_d\) is determined by a projection (Section \ref{ContrSec}) of a
vertex of \(A_d\) (finite binary tree with \(d+2\) leaves) and each
isomorphism type of a projection gives a face.  Isomorphisms of
projections are required to preserve local order.  The face
determined by a projection \(S\) is the set of all vertices of
\(A_d\) that have a projection isomorphic (as ordered trees) to
\(S\).

The dimension of the face\index{face!of
associahedron!dimension}\index{dimension!associahedron!face} is the
total number of internal edges in the selected subtrees of the
projection.

The structure of a face\index{face!of associahedron!structure} is a
product of various \(A_i\).  If \(T\) is a vertex in a face, then
there are selected subtrees \(T_1, T_2, \ldots, T_n\) of \(T\) that
specify the projection that defines the face.  Any other vertex
\(T'\) in the same face can be obtained from \(T\) by replacing each
\(T_i\) by a binary subtree with the same number \(k_i\) of leaves
as \(T_i\).  For each \(i\), this ``space of replacements'' gives an
isomorphic copy of \(A_{k_i-2}\).  The structure of the face is the
product \[ A_{k_1-2} \times A_{k_2-2} \times \cdots \times
A_{k_n-2}.  \] Note that if \(S\) is a projection determining a face
(namely \(S\) is a locally odered, rooted, finite tree), then the
dimension of the face determined by \(S\) is the sum over the
internal vertices \(v\) of \(S\) of the quantity \(d(v)-3\) where
\(d(v)\) is the degree of \(v\).

In the projection illustrated in Section \ref{ContrSec}, the tree
on the left is a vertex in \(A_9\), and the face determined by the
projection shown has the structure \(A_1\times A_1\times A_2\).

\subsection{Edges and rotations}\mylabel{PlainRotSec}

The simplest non-trivial faces of \(A_d\) are the edges.  An
edge\index{edge!associahedron}\index{associahedron!edge} with vertex
\(T\) at one end is determined by a subtree \(T_1\) of \(T\) having
one internal edge and three leaves.  There are exactly two binary
trees having three leaves and the other vertex \(T'\) of the edge is
obtained from \(T\) by removing \(T_1\) and replacing it by ``the
other'' tree with three leaves.

\newcommand{\rot}[1]{\lceil {#1} \rfloor}

Two trees that form the vertices of an edge are shown below.  The
arrow labeled \(\rot{u}\) will be explained in the paragraphs below.
Recall that we view all finite binary trees as living in \(\mathcal
T\).
\mymargin{ARot}
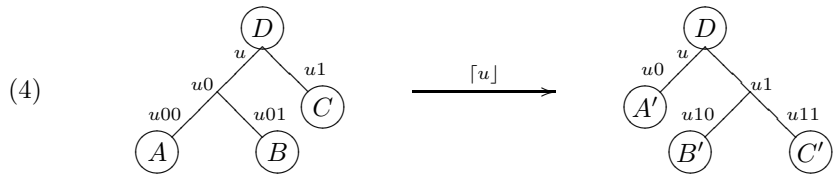
\begin{equation}\label{ARot}
\xy
(0,-6); (12,6)**@{-}; (18,0)**@{-};
(6,0); (12,-6)**@{-};
(-2,-8)*{A}*\cir<8pt>{}; (-1,-3)*{\scs u00};
(14,-8)*{B}*\cir<8pt>{}; (13,-3)*{\scs u01};
(20,-2)*{C}*\cir<8pt>{}; (19,3)*{\scs u1}; (4,1)*{\scs u0};
(12,8.5)*{D}*\cir<8pt>{}; (9,5)*{\scs u};
\endxy
\qquad 
\xymatrix{
{}\ar[rr]^{\rot{u}}&&{}
}
\qquad
\xy
(0,-6); (-12,6)**@{-}; (-18,0)**@{-};
(-6,0); (-12,-6)**@{-};
(2,-8)*{C'}*\cir<8pt>{}; (1,-3)*{\scs u11};
(-14,-8)*{B'}*\cir<8pt>{}; (-13,-3)*{\scs u10};
(-20,-2)*{A'}*\cir<8pt>{}; (-19,3)*{\scs u0};  (-4.5,1)*{\scs u1};
(-12,8.5)*{D}*\cir<8pt>{}; (-15,5)*{\scs u};
\endxy
\end{equation}

In each figure, the edges shown are the all the edges in the
selected subtree of three leaves that specify the relevant
projection.  The circles shown contain the rest of the tree.  The
vertices are labeled with their addresses.

The letters in the circles indicate corresponding subtrees in the
two trees.  For example, the circle containing \(A\) indicates a
subtree whose root is \(u0\).  However, it is more convenient to
leave the root edge outside the circle \(A\) in the figure.  The
subtree \(D\) is literally the same subtree in both trees and
contains the root.  The subtrees \(A\) and \(A'\) are isomorphic but
not identical since they do not use exactly the same vertex set.
Similar comments apply to \(B\) and \(B'\) as well as \(C\) and
\(C'\).

We refer to the transition from one of the trees shown to the other
as a {\itshape
rotation}\index{rotation!tree}\index{tree!rotation}\index{edge!rotation
along}\index{rotation!along edge}.  Thus traveling along an edge in
an associahedron is performing a rotation of one tree into another.
Our rotations are the transplantations of \Kryu{}
and are dual to the (unsigned) diagonal flips of \EliaOne.

The particular rotation illustrated above will be referred to as the
rotation of the left tree (or from the left tree to the right tree)
at \(u\) and will be denoted
\(\rot{u}\)\index{\protect\(\rot{u}\protect\)}.  The fact that trees
with many different shapes might use the vertex with address \(u\)
will not be a problem, and in fact the common notation for all
rotations using a subtree with root at \(u\) will be a benefit.

We invent the notation \(\overline
u\)\index{\protect\(\overline{u}\protect\)} formally so that
\(\rot{\overline u}=\rot{u}^{-1}\) is the rotation from the right
tree to the left tree.  We can also abuse the notation a bit and
write \(\rot{v}\) to mean a rotation in either direction depending
on whether \(v\) is of the ``form'' \(u\) or \(\overline u\), so
that \(\rot{v}=\rot{u}^{-1}\) if \(v=\overline u\).

The vertices \(u\) and \(u0\) in the left tree in \tref{ARot} will
be called the {\itshape pivot vertices}\index{pivot vertex!of
rotation}\index{rotation!pivot vertex} of the rotation \(\rot u\)
shown there.  The vertices \(u\) and \(u1\) in the right tree in
\tref{ARot} will be the pivot vertices of the rotation
\(\rot{\overline u}\).

\subsection{Examples}\mylabel{AssocExmplSec}

The simplest non-trivial associahedron is \(A_1\) and is a single
edge shown below.
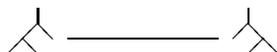
\[
\xy
(-2,-2); (2,2)**@{-}; (4,0)**@{-}; (0,0); (2,-2)**@{-};
(6,0); (26,0)**@{-};  (2,2); (2,4)**@{-};
(34,-2); (30,2)**@{-}; (28,0)**@{-}; (32,0); (30,-2)**@{-};
(30,2); (30,4)**@{-};
\endxy
\]

We can embed the trees in the figure above into larger trees so as
to recover the figure \tref{ARot}.  This would map the vertex
\(\emptyset\) in the trees above to the vertex \(u\) in the trees in
\tref{ARot}.  If we think of the figure above as a simplification of
\tref{ARot} by forgetting to put in the trees \(A, A', B, B', C,
C'\) and \(D\), then the figure above can represent an arbitrary
edge (1-dimensional face) in some \(A_d\).  With \(\emptyset\)
mapped to a vertex \(u\), in \tref{ARot}, we can redraw the figure
above as follows (with the root edge omitted for simplicity) to show
such an edge.
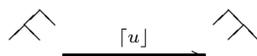
\[
\xymatrix
{
{\xy
(6,2); (4,4)**@{-}; (6,6)**@{-}; (8,4)**@{-};
(4,4); (2,2)**@{-};
\endxy}
\ar[rr]^{\rot{u}} &&
{\xy
(-6,2); (-4,4)**@{-}; (-6,6)**@{-}; (-8,4)**@{-};
(-4,4); (-2,2)**@{-};
\endxy}
}
\]
This does not give a unique edge in a higher dimensional
associahedron since the rotation \(\rot{u}\) might apply to many
vertices in an associahedron.  For example, the rotation
\(\rot{\emptyset}\) applies to any tree (vertex of an associahedron)
whose left subtree is not trivial.

In a similar spirit, we show a typical 2-dimensional face that is
isomorphic to \(A_2\) below.
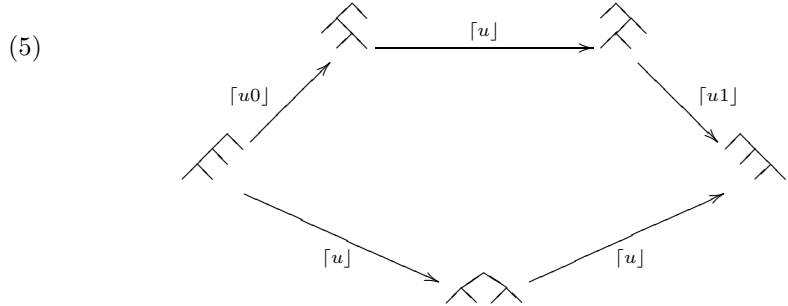
\mymargin{ThePentagon}\begin{equation}\label{ThePentagon}
\xymatrix
{
&
{\xy
(8,0); (4,4)**@{-}; (6,6)**@{-}; (8,4)**@{-};
(6,2); (4,0)**@{-};
(4,4); (2,2)**@{-};
\endxy}
\ar[rr]^{\rot{u}} &&
{\xy
(-8,0); (-4,4)**@{-}; (-6,6)**@{-}; (-8,4)**@{-};
(-6,2); (-4,0)**@{-};
(-4,4); (-2,2)**@{-};
\endxy}
\ar[dr]^{\rot{u1}} \\
{\xy
(0,0); (6,6)**@{-}; (8,4)**@{-};
(2,2); (4,0)**@{-};
(4,4); (6,2)**@{-};
\endxy}
\ar[ur]^{\rot{u0}} \ar[drr]_{\rot{u}} &&&& 
{\xy
(0,0); (-6,6)**@{-}; (-8,4)**@{-};
(-2,2); (-4,0)**@{-};
(-4,4); (-6,2)**@{-};
\endxy} 
\\
&&
{\xy
(0,0); (2,2)**@{-}; (5,4)**@{-}; (8,2)**@{-}; (10,0)**@{-};
(2,2); (4,0)**@{-};
(8,2); (6,0)**@{-};
\endxy}
\ar[urr]_{\rot{u}}
}
\end{equation}
In the figure above, \(u\) is the address of the 
top vertex drawn in each of the subtrees shown.

There can also be 2-dimensional faces of the form \(A_1\times
A_1\).  There are three of these in the figure below which shows a
typical 3-dimensional face of the form \(A_3\).
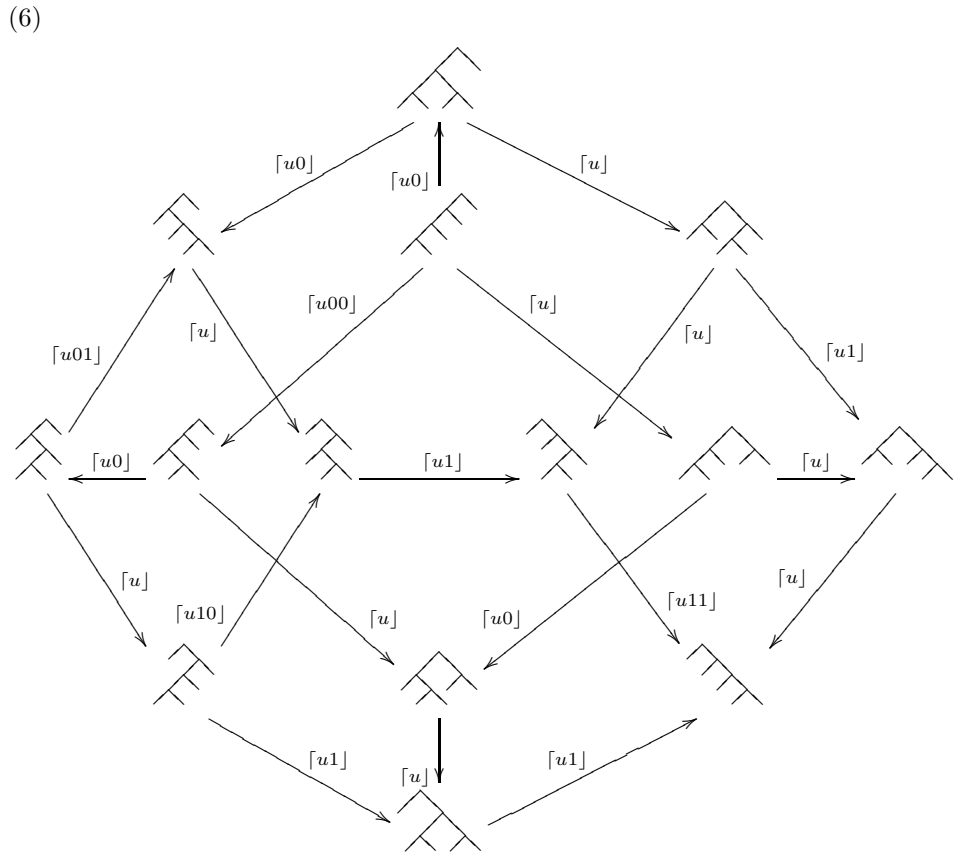
\mymargin{TheAssoc}\begin{align}\label{TheAssoc}
\xymatrix@C=11.5pt
{
&&&&&{\xy
(0,0); (2,2)**@{-}; (5,5)**@{-}; (8,2)**@{-}; (10,0)**@{-};
(2,2); (4,0)**@{-};
(8,2); (6,0)**@{-};
(5,5);(8,8)**@{-};(11,5)**@{-};
\endxy}
\ar[dlll]_{\rot{u0}} \ar[drrr]^{\rot{u}}&&&&& \\
&&{\xy
(4,6);(6,8)**@{-};(8,6)**@{-};
(2,4);(4,6)**@{-};(6,4)**@{-};
(4,2);(6,4)**@{-};(8,2)**@{-};
(6,0);(8,2)**@{-};(10,0)**@{-};
\endxy} \ar[ddrr]_(.3){\rot{u}}&&&{\xy
(6,6);(8,8)**@{-};(10,6)**@{-};
(4,4);(6,6)**@{-};(8,4)**@{-};
(2,2);(4,4)**@{-};(6,2)**@{-};
(0,0);(2,2)**@{-};(4,0)**@{-};
\endxy} \ar[ddlll]_(.3){\rot{u00}} \ar[u]^{\rot{u0}}
\ar[ddrrr]^(.3){\rot{u}}&&&{\xy
(4,4);(7,7)**@{-};(10,4)**@{-};
(2,2);(4,4)**@{-};(6,2)**@{-};
(6,0);(8,2)**@{-};(10,0)**@{-};
(8,2);(10,4)**@{-};(12,2)**@{-};
\endxy} \ar[ddll]^(.3){\rot{u}}
\ar[ddrr]^{\rot{u1}}&& \\ \\
{\xy
(2,6);(4,8)**@{-};(6,6)**@{-};
(0,4);(2,6)**@{-};(4,4)**@{-};
(6,2);(4,4)**@{-};(2,2)**@{-};
(0,0);(2,2)**@{-};(4,0)**@{-};
\endxy}\ar[uurr]^{\rot{u01}} \ar[ddrr]^{\rot{u}}&&{\xy
(4,6);(6,8)**@{-};(8,6)**@{-};
(2,4);(4,6)**@{-};(6,4)**@{-};
(4,2);(2,4)**@{-};(0,2)**@{-};
(2,0);(4,2)**@{-};(6,0)**@{-};
\endxy} \ar[ddrrr]^(.7){\rot{u}}
\ar[ll]_{\rot{u0}}&&{\xy
(0,6);(2,8)**@{-};(4,6)**@{-};
(2,4);(4,6)**@{-};(6,4)**@{-};
(4,2);(2,4)**@{-};(0,2)**@{-};
(2,0);(4,2)**@{-};(6,0)**@{-};
\endxy}\ar[rr]^{\rot{u1}}&&{\xy
(0,6);(2,8)**@{-};(4,6)**@{-};
(2,4);(4,6)**@{-};(6,4)**@{-};
(4,2);(6,4)**@{-};(8,2)**@{-};
(2,0);(4,2)**@{-};(6,0)**@{-};
\endxy}  \ar[ddrr]^(.6){\rot{u11}}&&{\xy
(4,4);(7,7)**@{-};(10,4)**@{-};
(2,2);(4,4)**@{-};(6,2)**@{-};
(0,0);(2,2)**@{-};(4,0)**@{-};
(8,2);(10,4)**@{-};(12,2)**@{-};
\endxy} \ar[ddlll]_(.7){\rot{u0}}
\ar[rr]^{\rot{u}}&&{\xy
(4,4);(7,7)**@{-};(10,4)**@{-};
(2,2);(4,4)**@{-};(6,2)**@{-};
(10,0);(12,2)**@{-};(14,0)**@{-};
(8,2);(10,4)**@{-};(12,2)**@{-};
\endxy} \ar[ddll]_{\rot{u}} \\ \\
&&{\xy
(2,6);(4,8)**@{-};(6,6)**@{-};
(4,4);(6,6)**@{-};(8,4)**@{-};
(2,2);(4,4)**@{-};(6,2)**@{-};
(0,0);(2,2)**@{-};(4,0)**@{-};
\endxy}
\ar[drrr]^{\rot{u1}} \ar[uurr]^(.35){\rot{u10}} &&&
{\xy
(4,4);(7,7)**@{-};(10,4)**@{-};
(2,2);(4,4)**@{-};(6,2)**@{-};
(4,0);(6,2)**@{-};(8,0)**@{-};
(8,2);(10,4)**@{-};(12,2)**@{-};
\endxy}
\ar[d]_{\rot{u}}&&&{\xy
(6,0);(8,2)**@{-};(10,0)**@{-};
(4,2);(6,4)**@{-};(8,2)**@{-};
(2,4);(4,6)**@{-};(6,4)**@{-};
(0,6);(2,8)**@{-};(4,6)**@{-};
\endxy}&& \\
&&&&&{\xy
(0,0); (2,2)**@{-}; (5,5)**@{-}; (8,2)**@{-}; (10,0)**@{-};
(2,2); (4,0)**@{-};
(8,2); (6,0)**@{-};
(5,5);(2,8)**@{-};(-1,5)**@{-};
\endxy}\ar[urrr]^{\rot{u1}}&&&&& \\
}
\end{align}
In the figure above, there are six 2-dimensional faces of the form
\(A_2\).

Recall that each face corresponds to a projection.  The central
square face corresponds to the tree
\(
\xy
(-3,-3); (3,3)**@{-}; (6,0)**@{-}; 
(3,3); (3,0)**@{-};
(0,-3); (0,0)**@{-}; (3,-3)**@{-};
\endxy
\) and
the tree
\(
\xy
(-4.5,0); (0,3)**@{-}; (4.5,0)**@{-};
(-1.5,0); (0,3)**@{-}; (1.5,0)**@{-};
(0,-2); (1.5,0)**@{-}; (3,-2)**@{-};
\endxy
\)
corresponds to the top central pentagonal face.

\subsection{Symmetries}

Note that the dihedral group\index{dihedral group!action on
associahedron}\index{associahedron!dihedral group action} of order
\(2(d+3)\) acts on the associahedron \(A_d\) by root shifts and
reflections on the individual trees.  See Section
\ref{RootShiftSec}.  That the action preserves the edges of \(A_d\)
is easy to see from the dual view of triagulated polygons.  In that
setting an edge in \(A_d\) corresponds to an unsigned diagonal flip
of \EliaOne.  This replaces two triangles in a triangulated
\(n\)-gon that share an edged by ``the other two triangles'' with
the same union as the original two triangles.  The reader can check
that the two examples in \tref{TriPolyExmpls} are connected by an
edge in \(A_3\).

\section{Color}\mylabel{ColorSec}

\subsection{The colors} 

All our face and edge colors will come from \(\Z_2\times
\Z_2\).\index{color!element of group}\index{group!of colors} We use
\(\Z_2\times \Z_2\) since addition and subtraction are identical in
this group.  Notations such as \((0,0)\) and \((0,1)\) and so forth
are cumbersome, and we will ``code''\index{color!code
for}\index{code!for colors} the elements of \(\Z_2\times \Z_2\) by
the mapping \((0,0)\rightarrow 0\), \((0,1)\rightarrow 1\),
\((1,0)\rightarrow 2\) and \((1,1)\rightarrow 3\).  In other words,
the 0 and 1 in each \((x,y)\) is thought of as a binary digit.  In
spite of this coding, the view that the colors reside in
\(\Z_2\times \Z_2\) will be retained since we will use the
arithmetic in \(\Z_2\times \Z_2\) to add colors.  Thus 0 is the
identity, we have \(1+2=3\), \(1+3=2\), \(2+3=1\), and we have
\(n+n=0\) for any \(n\in \{0,1,2,3\}\).

\subsection{Face and edge colorings}

It is standard\index{face-edge!coloring
correspondence}\index{edge-face!coloring correspondence} that for
planar cubic maps, a proper, face 4-coloring exists if and only if
there is a corresponding proper, edge 3-coloring.  This observation
goes back to Tait \cite{12.0409.01}.  One uses colors from a group
of four elements (such as \(\Z_2\times \Z_2\)).  Given a proper,
face 4-coloring, one colors the edges with the difference of the
colors of the two faces that impinge on the edge.  The identity is
never the color of an edge.  That the result is a proper, edge
3-coloring is a trivial exercise.  The reverse direction depends on
the planarity of the map.  See Theorem 4-3 of \Saaty.  The planarity
is essential.  This is discussed in Section \ref{TreeTorusSec}.

Now let \(M\) be in \(\mathfrak W\) with \((D,R)\) the associated
pair of trees.  Recall the one-to-one correspondence from the
vertices of degree 1 in \(D\) to the similar set in \(R\).  Since we
are going to discuss edges, we can also think of this correspondence
as between the set of external edges in \(D\) to the set of external
edges in \(R\).

If \(M\) has a proper, edge 3-coloring, then the coloring restricts
to a proper, edge 3-coloring of both \(D\) and \(R\).  The important
feature of these colorings of \(D\) and \(R\) is that they give the
same color to corresponding external edges of \(D\) and \(R\).  This
leads to the next discussion.

\subsection{Coloring binary trees}\mylabel{ColBinTreeSec}

Let \(T\) be a finite, binary tree with \(n\) leaves.  A {\itshape
color vector}\index{color!vector}\index{vector!color} for \(T\) is
an \(n\)-tuple \(\mathbf c=(c_1, \ldots, c_n)\) with values in
\(\Z_2\times \Z_2\).  The vector \(\mathbf c\) is thought of as an
assignment of colors to the edges impinging on the leaves of \(T\).
The left-right order of the leaves of \(T\) is used to number the
leaves from left to right starting at 1, and the color \(c_i\) is
assigned to the edge impinging on the \(i\)-th leaf of \(T\).

We also think of the color \(c_i\) as being assigned to the \(i\)-th
leaf\index{color!of leaf in
tree}\index{leaves!tree!color}\index{tree!leaf!color} itself as well
as the edge impinging on it.  In fact, it will often be convenient
to think of a color assigned to an edge in a tree as also being
assigned to the vertex\index{color!of vertex in
tree}\index{vertex!tree!color}\index{tree!vertex!color} at the lower
end of the edge.  As a special case, a color assigned to the root
edge is not only thought of as assigned to \(\emptyset\) but also to
\(*\).  We call this color the root
color.\index{root!color}\index{color!root!of tree}

One inductively argues the following.

\begin{lemma}\mylabel{ColorComputLem} Given a finite, binary tree
\(T\) and color vector \(\mathbf c\) for \(T\), there is a unique
coloring of the edges of \(T\) from \(\Z_2\times \Z_2\) so that at
every vertex \(v\) of degree 3 in \(T\) the sum of the colors of the
edges impinging on \(v\) is zero.  \end{lemma}

We refer to the coloring of \(T\) given in Lemma
\ref{ColorComputLem} as the coloring determined\index{tree!coloring
from vector}\index{coloring!of tree from vector} by the color vector
\(\mathbf c\).

Note that the coloring in Lemma \ref{ColorComputLem} is not
guaranteed to be a proper, edge 3-coloring for several reasons.
First, we did not restrict the colors in \(\mathbf c\) to the non
identity elements of \(\Z_2\times \Z_2\).  Second, even if we made
such a restriction, there is nothing to prevent the identity from
showing up as the computed color of an edge not ending at a leaf.
Of course if this happens, then the colors leading to the computed
identity will also be identical.

If the identity is used in \(\mathbf c\), and the tree has more than
one leaf, then at least one vertex will not have three different
colors for the edges ending at that vertex.  A last comment is that
if no edge impinging on a given vertex \(v\) of degree 3 has color
0, then because of the arithmetic in \(\Z_2\times \Z_2\), the three
edges impinging on \(v\) have different colors.  We have argued the
following.

\begin{lemma}\mylabel{WhenVecValidLem} If \(T\) is a finite, binary
tree with at least two leaves and \(\mathbf c\) is a color vector
for \(T\), then the coloring of \(T\) determined by \(\mathbf c\)
will be a proper, edge 3-coloring of \(T\) if and only if the
identity in \(\Z_2\times \Z_2\) is not used in the coloring of \(T\)
determined by \(\mathbf c\).  \end{lemma}

We say that a color vector for a binary tree is {\itshape
valid}\index{valid!vector!for tree}\index{color!vector!valid for
tree}\index{tree!valid color vector}\index{vector!valid for tree} if
the coloring determined by the vector is a proper, edge 3-coloring.
There is no color vector valid for all trees.  In particular, we
have the following which is only a check of small number of cases.

\begin{lemma}\mylabel{NoFiveColLem} There are five finite, binary
trees with four leaves and there is no color vector that is valid
for all five.  \end{lemma}

There are color vectors that are valid for no trees.  We say that a
color vector is {\itshape acceptable}\index{acceptable!color
vector}\index{color!vector!acceptable}\index{vector!color!acceptable}
if there is a finite, binary tree for which it is valid.  In Section
\ref{AcceptColSec} we will prove that a color vector is acceptable
if and only if the vector is not a constant vector and does not sum
to zero.

A finite, binary tree \(T\) with \(n\) leaves gives a way to
parenthesize a sum of \(n\) variables.  A leaf corresponds to a
single (unparenthesized) variable.  If \(T=A\caret B\) where \(A\)
and \(B\) have \(j\) and \(k\) leaves respectively, then \(j+k=n\),
the tree \(A\) parenthesizes the sum of the first \(j\) variables to
give an expression \(E_A\), the tree \(B\) parenthesizes the sum of
the last \(k\) variables to give an expression \(E_B\) and we define
\(T\) to correspond to the expression \((E_A+E_B)\).  The first
sentence of the following is argued inductively and the second
follows from the associativity of the addition in \(\Z_2\times
\Z_2\).

\begin{lemma}\mylabel{RootColorLem} If \(\mathbf c\) is a color
vector for finite, binary tree \(T\), then the root color determined
by \(\mathbf c\) is simply the sum of the colors in \(\mathbf c\).
In particular this color depends only on \(\mathbf c\) and not on
the structure of \(T\).  \end{lemma}

We now turn to pairs of trees.

\subsection{Coloring binary tree pairs}\mylabel{TreePairColorSec}

Recall that given a pair \((D,R)\) of finite trees it is always
assumed that the trees have the same number of leaves.

If \((D,R)\) is a pair of finite, binary trees with \(n\) leaves
each, and \(\mathbf c\) is a color vector for both \(D\) and \(R\),
then we say that \(\mathbf c\) is {\itshape valid
for}\index{valid!vector!for pair}\index{color!vector!valid for
pair}\index{tree!finite pair!valid color vector}\index{vector!valid
for pair} or {\itshape is a coloring of} \((D,R)\) if it is valid
for both \(D\) and \(R\).  If there is a color vector valid for
\((D,R)\), we say that \((D,R)\) {\itshape has a coloring}.

We now can get a statement that is equivalent to the 4CT.  This
observation is just a variation on statements whose equivalence to
the 4CT follows from Whitney's theorem.

\begin{prop}\mylabel{TreePairEquivThm} The four color theorem is
equivalent to the statement that every pair of finite, binary trees
with the same number of leaves has a coloring.  \end{prop}

\begin{proof} If \((D,R)\) is the pair of trees associated to \(M\in
\mathfrak W\), then a color vector for the pair determines the same
root color for both \(D\) and \(R\) by Lemma \ref{RootColorLem}.
Thus the edge colorings of the two trees will match when the map is
reconstructed.  The other direction is immediate from the
construction in Section \ref{TreePairToMapSec}.  \end{proof}

Using arguments from Section \ref{TwoRedSec}, we can introduce the
word prime.

\begin{prop} The four color theorem is equivalent to the statement
that every prime pair of finite, binary trees with the same number
of leaves has a coloring.  \end{prop}

\subsection{Signs}

Let \(T\) be a finite, binary tree.  It will be convenient to view
\(T\) as drawn in the plane in a way that agrees with the ordering.
Assume that \(T\) is given a proper, edge 3-coloring from the
non-zero elements of \(\Z_2\times \Z_2\).  Taking into account the
cyclic ordering of the edges around a vertex \(v\) of valence three
and the fact that one edge (the edge to the parent) is
distinguished, there are six possible valid colorings of the edges
that meet at \(v\).  If the colors 1, 2 and 3 are assigned so that
the numbers increase by one as the intervals are visited in the
cyclic order, then \(v\) is said to have {\itshape positive} sign
(or color)\index{sign!at
vertex}\index{tree!vertex!sign}\index{vertex!of tree!sign}.
Otherwise it is said to have {\itshape negative} sign (or color).

In case our definition of positive and negative was not clear, we
show the positive arrangements
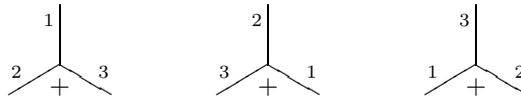
\[
\xy
(-6.8,-4); (0,0)**@{-};
(6.8,-4); (0,0)**@{-};
(0,8); (0,0)**@{-};
(-1.5,6)*{\scs1}; (-5.8,-1)*{\scs2}; (5.8,-1)*{\scs3}; (0,-3)*{+};
\endxy \qquad\qquad
\xy
(-6.8,-4); (0,0)**@{-};
(6.8,-4); (0,0)**@{-};
(0,8); (0,0)**@{-};
(-1.5,6)*{\scs2}; (-5.8,-1)*{\scs3}; (5.8,-1)*{\scs1}; (0,-3)*{+};
\endxy \qquad \qquad
\xy
(-6.8,-4); (0,0)**@{-};
(6.8,-4); (0,0)**@{-};
(0,8); (0,0)**@{-};
(-1.5,6)*{\scs3}; (-5.8,-1)*{\scs1}; (5.8,-1)*{\scs2}; (0,-3)*{+};
\endxy 
\]
and the negative arrangements.
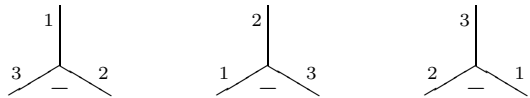
\[
\xy
(-6.8,-4); (0,0)**@{-};
(6.8,-4); (0,0)**@{-};
(0,8); (0,0)**@{-};
(-1.5,6)*{\scs1}; (-5.8,-1)*{\scs3}; (5.8,-1)*{\scs2}; (0,-3)*{-};
\endxy \qquad\qquad
\xy
(-6.8,-4); (0,0)**@{-};
(6.8,-4); (0,0)**@{-};
(0,8); (0,0)**@{-};
(-1.5,6)*{\scs2}; (-5.8,-1)*{\scs1}; (5.8,-1)*{\scs3}; (0,-3)*{-};
\endxy \qquad \qquad
\xy
(-6.8,-4); (0,0)**@{-};
(6.8,-4); (0,0)**@{-};
(0,8); (0,0)**@{-};
(-1.5,6)*{\scs3}; (-5.8,-1)*{\scs2}; (5.8,-1)*{\scs1}; (0,-3)*{-};
\endxy 
\]

Thinking of the plus and minus signs as colors gives us a vertex
2-coloring of the vertices of valence 3 in \(T\).  To keep the
number of references to colorings down, we refer to this kind of
coloring as a {\itshape sign assignment}\index{sign assignment!of
tree}\index{tree!sign assignment} on \(T\).  It is convenient to
have a symbol for a sign assignment, so saying that \(\sigma\) is a
sign assignment means \(\sigma\) is a function from the internal
vertices of \(T\) to the set of two symbols \(\{+,-\}\).  We will
use \(T^\sigma\)\index{\protect\(T^\sigma\protect\)} to denote a
tree with sign assignment \(\sigma\).

There are no restrictions on a sign assignment.  In particular two
vertices joined by an edge can be given the same sign, and assigning
the same sign to all vertices of valence 3 is an acceptable example
of an assignment.

We immediately have the following.

\begin{lemma}\mylabel{SignAnd1ColorLem} Given a finite, binary tree
\(T\), a sign assignment on \(T\), and an assignment of a non-zero
color from \(\Z_2\times \Z_2\) to a single edge of \(T\), then there
is a unique, proper, edge 3-coloring for \(T\) that agrees with the
given information.  \end{lemma}

Thus given a proper, edge 3-coloring of a \(T\) as above, we get a
sign assignment for \(T\), and given a sign assignment for such a
\(T\), we get three possible proper, edge 3-colorings that are
consistent with the sign assignment.  The three different colorings
are determined by starting with the three possible different colors
for one chosen edge.

\subsection{Small examples}

The trees shown below are the only finite, binary trees with three
leaves.

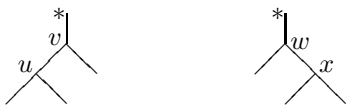
\[
\xy
(0,0); (-4,4)**@{-}; (-12,-4)**@{-}; 
(-8,0); (-4,-4)**@{-}; (-4,4); (-4,8)**@{-}; (-5,8)*{*};
(-5.5,4.5)*{v}; (-9.5,1)*{u};
\endxy
\qquad\qquad\qquad
\xy
(0,0); (4,4)**@{-}; (12,-4)**@{-}; 
(8,0); (4,-4)**@{-}; (4,4); (4,8)**@{-}; (3,8)*{*};
(6,4)*{w}; (9.5,1)*{x};
(0,-8)*{};
\endxy
\]

The following is the result of checking a small number of cases.

\begin{lemma}\mylabel{ColorX0Lem} The only color vectors that are
simultaneously valid for the pair of trees shown above have the form
\((a,b,a)\) where \(a\ne b\).  In all valid cases, the signs on
vertices \(u\) and \(v\) will be equal, and will be the negatives of
the signs on vertices \(w\) and \(x\).

Conversely, if a color vector is valid for the left tree shown above
and makes the signs of vertices \(u\) and \(v\) the same, then the
vector is of the form \((a,b,a)\) with \(a\ne b\) and the paragraph
above applies.

Lastly, the color vector \((a,b,a)\) with \(a\ne b\) produces the
color \(b\) at the root edge of either tree.  \end{lemma}

\subsection{Permuting colors}\mylabel{PermColorSec}

Two 3-colorings (of anything, a tree, an \(M\in \mathfrak M\)), that
differ but can be made the same by permuting the colors of one of
the colorings are not different to us in an interesting way.  If a
finite, binary tree \(T\) has a proper, edge 3-coloring and its
derived sign assignment, then permuting the
colors\index{permuting!colors}\index{color!permutation} will either
preserve the sign assignment (if the permutation is even) or negate
it (if the permutation is odd)\index{sign assignment!effect of
permutation}\index{color!permutation!effect on sign assignment}.
Conversely a sign assignment on \(T\) together with the negative of
the sign assignment will produce all six permutations of an edge
3-coloring of \(T\) by running through the three starting colors for
a particular edge for both the sign assignment and its negative.

We can try to ``normalize''
colorings\index{coloring!normal}\index{normal!coloring}
on a finite, binary tree \(T\) by insisting that the color of the
root edge be 1, and that the sign of the unique child of the root be
positive.  However, Lemma \ref{ColorX0Lem} shows that we might need
to consider colorings that violate this normalization if we are to
look for color vectors that are simultaneously valid for pairs of
trees.

Thus if the coloring of a finite, binary tree assigns the color 1 to
the root edge, we say it is {\itshape positive
normal}\index{normal!coloring!positive}\index{positive!normal
coloring} if the sign of the child of the root is positive and
{\itshape negative
normal}\index{normal!coloring!negative}\index{negative!normal
coloring} if the sign of the child of the root is negative.

Note that there is a one-to-one correspondence between sign
assignments of a finite, binary tree and proper, edge 3-colorings of
that tree that use the color 1 for the root edge.  The normalization
and this one-to-one correspondence will allow us to discuss sign
assignments as if they are edge 3-colorings and vice versa.

We apply this to counting colorings of pairs of finite, binary
trees.  Given a coloring of a pair \((D,R)\) of finite, binary
trees, there is a unique way to permute the colors so that the
coloring of \(D\) is positive normal.  Thus the number of colorings
of \((D,R)\) modulo the action of \(S_3\) on the colors is the
number of colorings in which the coloring of \(D\) is positive
normal.  When we count colorings of a pair of finite, binary trees,
we will always be counting modulo the action of \(S_3\) on the
colors.  Statements that colorings are unique will be made with this
convention.

\subsection{The dual view}\mylabel{DualColSec}

The material above applied to triangulated polygons is presented in
\Gravier.  In the setting of trees, the driving data is an edge
coloring, but for triangulated polygons it is easier to start with
vertex colorings.  Let \(T\) be a triangulation of a polygon.  Then
\((T,c)\) is a coloring of the vertices.  This corresponds to the
``color cycle'' of Section \ref{IntroAssocSec}.  

From \((T,c)\), we get \((T,d)\) where \(d\) is a coloring of the
internal edges of \(T\) where each edge is colored by the difference
of the colors of its endpoints.  Given an edge coloring \((T,d)\)
and a color of a single vertex, we recover the vertex coloring
\((T,c)\) from which \((T,d)\) is derived.

From \((T,d)\), we get \((T,\sigma)\) where \(\sigma\) is a sign
assignment or signature giving a plus or minus sign to each
triangular face depending on the cyclic order of the colors of the
edges around the face.  Given a signature \((T,\sigma)\) and a color
of a single edge, we recover the edge coloring \((T,d)\) from which
\((T,\sigma)\) is derived.  If instead of one edge color, we are
told the colors of two adjacent vertices, we recover not only
\((T,d)\) but also the vertex coloring \((T,c)\).

From \((T,\sigma)\), we get \((T,v)\) where \(v\) is a ``valuation''
on the edges taking values in \(\{0,1\}\) and an edge is assigned 0
if the signs of the two triangles that share the edge are the same,
and 1 otherwise.  We build back to \((T,\sigma)\), \((T,d)\) and
\((T,c)\) given a valuation \((T,v)\) and, respectively, the sign of
one triangle, the colors of two edges that share a vertex, or the
colors of three vertices of a triangle.

A remark from \Gravier{} is worth isolating as a separate statement.

\begin{lemma}\mylabel{DualRigidLem} If \(T\) is a triangulation of a
polygon and \((T,c)\) and \((T,v)\) are a vertex coloring and its
derived edge valuation, then \(c\) uses only three colors if and
only if all values of \(v\) equal 1.  \end{lemma}

\begin{proof} This follows inductively from the fact that it holds
for squares.  The figure below holds all the relevant information.
In it \(\alpha=x-y\), \(\beta=y-z\), \(\gamma=x-z\) are all in
\(\Z_2\times \Z_2\) and \(x\), \(y\), \(z\) and \(w\) are all
different.  \end{proof}

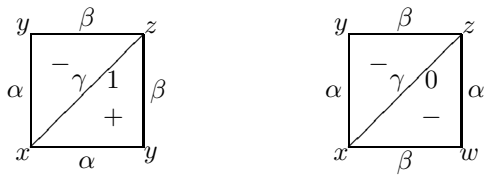
\[
\xy
(0,0); (15,15)**@{-}; (15,0)**@{-}; (0,0)**@{-}; (0,15)**@{-};
(15,15)**@{-}; (-1,-1)*{x};  (-1,16)*{y}; (16,16)*{z}; (16,-1)*{y};
(-2,7.5)*{\alpha}; (7.5,-2)*{\alpha}; (7.5,17)*{\beta}; (17,7.5)*{\beta};
(6.5,8.5)*{\gamma}; (11,9)*{1}; (11,4)*{+}; (4,11)*{-};
\endxy
\qquad\qquad\qquad
\xy
(0,0); (15,15)**@{-}; (15,0)**@{-}; (0,0)**@{-}; (0,15)**@{-};
(15,15)**@{-}; (-1,-1)*{x};  (-1,16)*{y}; (16,16)*{z}; (16,-1)*{w};
(-2,7.5)*{\alpha}; (7.5,-2)*{\beta}; (7.5,17)*{\beta}; (17,7.5)*{\alpha};
(6.5,8.5)*{\gamma}; (11,9)*{0}; (11,4)*{-}; (4,11)*{-};
\endxy
\]

\section{Colored rotations and colored paths}\mylabel{MainColRotSec}

From Proposition \ref{TreePairEquivThm}, we know that we want to color
pairs of trees.  It should be easier to color pairs of trees if the
structures of the two trees are closely related.  In our setting,
the most closely related trees pairs are those that are the
endpoints of an edge in an associahedron.  Thus we start with such
pairs and build from there.

\subsection{Colored rotations}\mylabel{ColRotSec}

We assume a color vector that is valid for both trees shown in
\tref{CRot}.  Note that these will be two trees that are connected
by an edge in some associahedron.  The left-right order on the
leaves of the two trees matches, each in left-right order, leaves of
\(D\) to the left of \(u\) in the two trees, leaves of \(A\) to
those of \(A'\), leaves of \(B\) to those of \(B'\), leaves of \(C\)
to those of \(C'\) and leaves of \(D\) to the right of \(u\) in the
two trees.  It follows that the three colors at \(u00\), \(u01\) and
\(u1\) in the left tree, equal in that order the colors at \(u0\),
\(u10\) and \(u11\) in the right tree.  From Lemma \ref{ColorX0Lem},
the signs at \(u0\) and \(u\), the pivot vertices of \(\rot u\), in
the left tree are equal.  Let \(\delta\) be this common sign.  Lemma
\ref{ColorX0Lem} also says that the signs at \(u\) and \(u1\), the
pivot vertices of \(\rot{\overline u}\), in the right tree are also
equal, and if \(\epsilon\) is this common sign, then
\(\epsilon=-\delta\).  This is illustrated below.
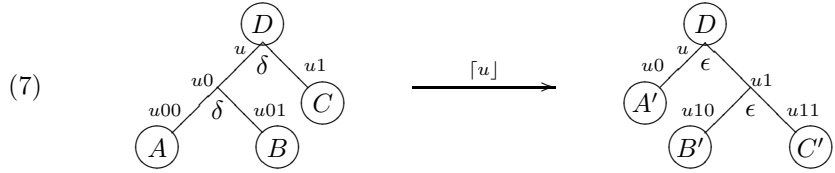
\mymargin{CRot}\begin{equation}\label{CRot}
\xy
(0,-6); (12,6)**@{-}; (18,0)**@{-};
(6,0); (12,-6)**@{-};
(-2,-8)*{A}*\cir<8pt>{}; (-1,-3)*{\scs u00};
(14,-8)*{B}*\cir<8pt>{}; (13,-3)*{\scs u01};
(20,-2)*{C}*\cir<8pt>{}; (19,3)*{\scs u1}; (4,1)*{\scs u0};
(12,8.5)*{D}*\cir<8pt>{}; (9,5)*{\scs u};
(12,3)*{\delta};
(6,-3)*{\delta};
\endxy
\qquad 
\xymatrix{
{}\ar[rr]^{\rot{u}}&&{}
}
\qquad
\xy
(0,-6); (-12,6)**@{-}; (-18,0)**@{-};
(-6,0); (-12,-6)**@{-};
(2,-8)*{C'}*\cir<8pt>{}; (1,-3)*{\scs u11};
(-14,-8)*{B'}*\cir<8pt>{}; (-13,-3)*{\scs u10};
(-20,-2)*{A'}*\cir<8pt>{}; (-19,3)*{\scs u0};  (-4.5,1)*{\scs u1};
(-12,8.5)*{D}*\cir<8pt>{}; (-15,5)*{\scs u};
(-12,3)*{\epsilon};
(-6,-3)*{\epsilon};
\endxy
\end{equation}

We call the rotation illustrated in \tref{CRot} a {\itshape signed
rotation}\index{rotation!signed}\index{signed!rotation}.  This is
the admissible transplantation of \Kryu{} and is
dual to the signed diagonal flip of \EliaOne.  The data for
such a rotation can be confined to the sign assignment of the two
trees and does not need to include either an edge coloring or a
color vector.  The requirement on the sign assignments is that the
signs of the relevant pivot vertices in the two trees are as shown
above with \(\epsilon=-\delta\), that the isomorphisms from \(A\),
\(B\) and \(C\) to, respectively, \(A'\), \(B'\) and \(C'\) preserve
signs on all non-root vertices, and that the signs in the common
subtree \(D\) in the two trees are identical.  It follows from the
remaining provisions of Lemma \ref{ColorX0Lem} that if there is a
signed rotation between two trees, then there is a color vector
valid for the two trees.

As before, we use \(\rot{u}\)\index{\protect\(\rot{u}\protect\)} to
denote the signed rotation illustrated above.  We think of
\(\rot{u}\) as a function whose value on the left tree in
\tref{CRot} is the right tree.

Extracting more from Lemma \ref{ColorX0Lem}, if there is a tree
\(T\) with sign assignment as illustrated by the left tree in
\tref{CRot} with signs on \(u\) and \(u0\) equal, then the rotation
\(\rot{u}\) can be applied to \(T\) to give the tree illustrated by
the right tree in \tref{CRot}.  In such case, we say that
\(\rot{u}\) is a {\itshape
valid}\index{valid!rotation}\index{rotation!valid} rotation for
\(T\).  Similarly, we can say that \(\rot{\overline u}\) is valid
for a tree as illustrated by the right tree in \tref{CRot}.

The following is contained in the discussion above.

\begin{lemma}\mylabel{ColoredEdgeLem}  Let \(T\) and \(S\) be
vertices in some associahedron \(A_d\) connected by an edge.  Let
\(\rot u\) be the rotation (unsigned) taking \(T\) to \(S\).  Let
\(\mathbf c\) be a color vector valid for \(T\) and give \(T\) the
sign assignment derived from \(\mathbf c\).  Then
the following are equivalent.
\begin{enumerate}
\item The rotation \(\rot u\) is valid for \(T\) as a signed
rotation.
\item The color vector \(\mathbf c\) is also valid for \(S\).
\end{enumerate}
\end{lemma}

If a finite, binary tree \(T\) with a sign assignment is given and
\(\rot{v}\) is a rotation valid for \(T\), then \(T\rot{v}\) will
denote the result of the rotation.

If \(w\) is a word in rotation symbols and if \(T\) is a finite,
binary tree with a sign assignment, then \(Tw\) will be defined
inductively as \((Tw')\rot{v}\) if \(w=w'\rot{v}\) , if \(Tw'\) is
defined, and if \(\rot{v}\) is valid for \(Tw'\).  In such case we
say that \(w\) is a {\itshape valid path of signed
rotations}\index{valid!path of rotations}\index{rotation!valid
path}\index{path!valid!of rotations} for \(T\).  Note that for \(w\)
to be be valid for \(T\), every prefix of \(w\) must be valid for
\(T\).  We use the word {\itshape path} since the movement from tree
to tree using the symbols that make up \(w\) travels along a path of
edges in an associahedron.

We have the following.

\begin{prop}\mylabel{RotChainProp} If \(T\) is a finite, binary tree
with a sign assignment, and \(w\) is a path of signed rotations that
is valid for \(T\), then there is a color vector that is valid for
the pair \((T,Tw)\).  \end{prop}

We can define the term {\itshape valid} to be applied to a path of
rotation symbols if it is valid for some finite, binary tree with a
sign assignment.  We will show in Section \ref{SignStructSec} that
the validity of a path can be detected directly from the path
without bringing in a tree with a sign assignment to test it on.

\subsection{A converse to Proposition
\ref{RotChainProp}}\mylabel{RotChainConvSec} 

A converse to Proposition \ref{RotChainProp} must be stated
carefully.  A path of signed rotations that is valid for a tree
\(T\) must start with a sign assignment for \(T\) that has at least
two adjacent vertices with the same sign.  A sign assignment for a
tree \(T\) that is proper in the sense that every pair of adjacent
vertices has opposite sign will be called {\itshape
rigid}\index{rigid!sign assignment}\index{sign assignment!rigid}
since it allows no signed rotations.  Rigidity will be discussed
further in Section \ref{RigidSec}.  A sign assignment that is not
rigid will be called {\itshape flexible}\index{flexible!sign
assignment}\index{sign assignment!flexible}.

The following is a rephrasing of the main result of \Gravier.

\begin{thm}\mylabel{GravPayanThm} If \(\mathbf c\) is a color vector
valid for the tree pair \((D,R)\) and the sign assignment derived
from \(\mathbf c\) on either \(D\) or \(R\) is flexible, then there
is a path \(w\) of signed rotations valid for \(D\) so that
\(Dw=R\).  \end{thm}

\subsection{The first signed path conjecture}

The following appears in \Kryu{} and
\EliaOne.  It will be restated in Section
\ref{SecondSPathConjSec} after we verify the claim made above that
the validity of a path can be detected directly from the path.

\begin{conj}\mylabel{SPathConjOne}\index{conjecture!signed
path!first} For every pair of finite, binary trees \((D, R)\) with
the same number of leaves, there is a sign assignment of \(D\) and a
word \(w\) of rotation symbols valid for \(D\) so that \(Dw=R\).
\end{conj}

It follows from Propositions \ref{RotChainProp} and
\ref{TreePairEquivThm} and Theorem \ref{GravPayanThm} that
Conjecture \ref{SPathConjOne} is equivalent to the 4CT.

\subsection{Rigidity basics}\mylabel{BasicRigidSec}

In a rigid sign assignment on a tree, the sign of the vertex
\(\emptyset\) determines the entire sign assignment.  Thus there are
two possible rigid sign assignments for a given tree and we refer to
a rigid sign assignment on a tree as its {\itshape positive rigid}
sign assignment if the vertex \(\emptyset\) is positive, and
{\itshape negative rigid} sign assignment otherwise.  A color
vector will be called positive rigid for a tree \(T\) if it leads
to the positive rigid sign assignment on \(T\), and negative rigid
for \(T\) if it leads to the negative rigid sign assignment on
\(T\).  A color vector is called flexible for \(T\) if it is
neither positive rigid, nor negative rigid for \(T\).

The following shows that the words ``for \(T\)'' are superfluous.

\begin{prop}[Trichotomy]\mylabel{DisjColorsProp} Let \(\mathbf c\)
be an acceptable color vector of length \(n\) and let \(A\) be the
set of all trees with \(n\) leaves for which \(\mathbf c\) is valid.
Then exactly one of the following holds.\index{Trichotomy
Theorem}\index{theorem!Trichotomy}
\begin{enumerate}
\item The vector \(\mathbf c\) is positive rigid for every \(T\) in
\(A\). 
\item The vector \(\mathbf c\) is negative rigid for every \(T\) in
\(A\). 
\item The vector \(\mathbf c\) is flexible for every \(T\) in \(A\).
\end{enumerate}
\end{prop}

\begin{proof}  This is easier to argue from the dual view as laid
out in Section \ref{DualColSec}.

A color vector gives an edge coloring to the tree and thus an edge
coloring to the dual triangulated polygon \(P\).  Note that a rigid
coloring (either positive or negative) corresponds to an edge
coloring on \(P\) that leads to a derived edge valuation that
assigns 1 to all edges.

Given a color \(c\) of a single vertex \(x\), this edge coloring of
\(P\) determines a vertex coloring of \(P\).  Changing the given
color on \(x\) can be viewed as adding an element \(s\) of
\(\Z_2\times \Z_2\) to \(c\) to give the color \(c+s\) on \(x\).
The coloring determined by the edge coloring and the color \(c+s\)
on \(x\) is obtained by adding \(s\) to all the colors obtained from
the color \(c\) on \(x\) and the edge coloring on \(P\).  In
particular, either all possible vertex colorings use three colors or
all possible vertex colorings use 4 colors.  By Lemma
\ref{DualRigidLem}, we see that either \(\mathbf c\) is rigid for
every \(T\) in \(A\) or \(\mathbf c\) is flexible for every \(T\) in
\(A\).

Now assume that \(\mathbf c\) is rigid for \(T\).  The sign of the
triangle containing the top edge of its dual triangulated polygon
\(P\) determines whether \(\mathbf c\) is positive rigid or
negative rigid.  By the arguments of the previous paragraph, we can
fix the color of one vertex on the top edge.  The edge coloring
determines the color of the other vertex of the top edge.  But this
determines the color of the third vertex of the triangle containing
the top edge since the coloring is rigid and only three colors are
used in the vertex coloring.  Thus the sign of that triangle does
not depend on the location of the third vertex and does not depend
on the particular triangulation of \(P\).
\end{proof}

\subsection{Colored associahedra and color
graphs}\mylabel{ColorGraphSec}

If \(A_d\) is the \(d\)-dimensional associahedron, then its vertices
are trees with \(d+2\) leaves.  A color vector \(\mathbf c\) with
\(d+2\) entries induces a coloring on all the vertices of \(A_d\),
although for some of the trees the induced coloring from \(\mathbf
c\) may have at least one zero on an edge and be invalid.  

We will call the pair \((A_d, \mathbf c)\) a colored
associahedron.\index{colored!associahedron}\index{associahedron!colored}
The subset of vertices of \(A_d\) consisting of trees for which
\(\mathbf c\) is valid will be called the colored
vertices\index{colored!vertices}\index{vertices!colored} of \((A_d,
\mathbf c)\).  The set of colored vertices will be empty if the
vector is not acceptable, so we restrict this discussion to
acceptable color vectors.

The subgraph of the 1-skeleton of \(A_d\) spanned by
the colored vertices will be called the {\itshape color
graph}\index{color graph!of color vector}\index{graph!color!of color
vector}\index{color!vector!color graph of}\index{vector!color!color
graph of} of \((A_d, \mathbf c)\).  Since the number of entries in
\(\mathbf c\) specifies the dimension of the associahedron that it
colors, it makes sense to refer to the color graph of \((A_d,\mathbf
c)\) as the color graph of \(\mathbf c\).  Note that from Lemma
\ref{ColoredEdgeLem}, the edges of the color graph of \(\mathbf c\)
correspond exactly to the rotations that are valid for trees validly
colored by \(\mathbf c\).

The following is an immediate consequence of Theorem
\ref{GravPayanThm} and Proposition \ref{DisjColorsProp}.

\begin{thm}\mylabel{GraphAlternateThm} Let \(\mathbf c\) be an
acceptable color vector of length \(d+2\).  Then its color graph in
\(A_d\) is either connected or has no edges.  \end{thm}

\section{Groups}\mylabel{GroupSec}

We introduce two groups.  One is widely known as Thompson's group
\(F\), and the other, while known and related to \(F\), is less well
known.  Their relevance is that one way to define their elements is
as equivalence classes of pairs of binary trees.

\subsection{Thompson's group
\protect\(F\protect\)}\mylabel{FDefsSec}

Thompson's group \(F\)\index{Thompson's group
\protect\(F\protect\)}\index{\protect\(F\protect\)}
is a finitely presented, infinite group with many interesting
properties.  It has many closely related faithful representations
and we will add to that list here.  While coming up with a new
representation of \(F\) is almost never interesting or deep, it is
often useful.  We will give what information we need to proceed and
will give references where necessary.

A standard reference for the group is \cite{CFP}.

Recall that given a pair \((D,R)\) of finite trees it is always
assumed that the trees have the same number of leaves.

We start with one of the most common representations.  The group
\(F\) can be defined as the group of self homeomorphisms of the unit
interval \([0,1]\) that are defined by pairs \((D,R)\) of finite,
binary trees.  We will show how to do this.

We associate to each word \(w\) in \(\{0,1\}^*\) an interval
\(i(w)\) in \([0,1]\).  We let \(i(\emptyset)=[0,1]\) and for any
\(w\) with \(i(w)\) defined, we let \(i(0w)\) be the image of
\(i(w)\) under the map \(x\mapsto x/2\) and \(i(1w)\) be the image
of \(i(w)\) under the map \(x\mapsto (x+1)/2\).  Note that \(i(0w)\)
and \(i(1w)\) will each be half as long as \(i(w)\) and their left
endpoints will be \(1/2\) apart in \([0,1]\).

Since \(i(0)\) is the left half of \(i(\emptyset)\) and \(i(1)\) is
the right half of \(i(\emptyset)\), it follows inductively that
\(i(w0)\) is the left half of \(i(w)\) and \(i(w1)\) is the right
half of \(i(w)\).  It will be convenient to let \(m(w)\) be the
midpoint of \(i(w)\).

We apply this by regarding words in \(\{0,1\}^*\) as vertices in
\(\mathcal T\).

Let \(T\) be a finite binary tree.  It follows inductively that if
\(L\) is the set of leaves of \(T\), then \(\{i(v)\mid v\in L\}\) is
a partition of \([0,1]\) into subintervals with disjoint interiors
so that if \(v\) is to the left of \(w\) in \(L\), then \(i(v)\) is
to the left of \(i(w)\) in \([0,1]\).  We denote this partition by
\(P(T)\).  Below is a small illustration.
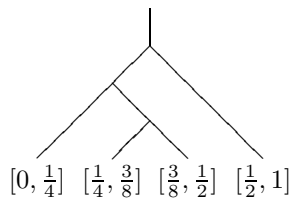
\[
\xy
(15,15); (15,20)**@{-};
(0,0); (15,15)**@{-}; (30,0)**@{-};
(10,10); (20,0)**@{-}; (15,5); (10,0)**@{-}; 
(0,-3)*{[0,\frac14]};
(10,-3)*{[\frac14,\frac38]};
(20,-3)*{[\frac38,\frac12]};
(30,-3)*{[\frac12,1]};
\endxy
\]

It follows that if \(v\) is any vertex in \(T\), then \(i(v)\) is
the union of all \(i(w)\) where \(w\) is a leaf of \(T_v\).  If this
is applied to \(v0\) and \(v1\), then we have that \(m(v)\) is the
right endpoint of the union of the \(i(w)\) where \(w\) is a leaf of
\(T_{w0}\) and the left endpoint of the union of the \(i(w)\) where
\(w\) is a leaf of \(T_{w1}\).

Let \((D,R)\) be a pair of finite, binary trees having \(n\) leaves
each and let \(I_1, \ldots, I_n\) be the intervals in \(P(D)\) in
left-to-right order in \([0,1]\), and let \(J_1, \ldots, J_n\) be
the intervals in \(P(R)\) in left-to-right order.  The homeomorhpism
\(h(D,R)\) defined by \((D,R)\) will take \(I_i\) affinely onto
\(J_i\), preserving orientation, for each \(i\).  The use of \(D\)
(for domain) and \(R\) (for range) is now explained.

Note that \(h(D,R)\) is piecewise linear, has only slopes that are
integral powers of 2, and that has discontinuities of slope confined
to the dyadic
rationals\index{dyadic!rational}\index{rational!dyadic}
\(\Z[\frac12]\), those rationals of the form \(m/2^n\) for integers
\(m\) and \(n\).  The group \(F\) is the set of all \(h(D,R)\) under
composition.  It is shown in Lemma 2.2 of \cite{CFP} that \(F\)
contains all piecewise linear self homeomorphisms of \([0,1]\) that
have only slopes that are integral powers of 2 and for which the
discontinuities of slope are confined to \(\Z[\frac12]\).

Several pairs of trees can give the same element of \(F\).  We can
declare two pairs to be \(h\)-equivalent if they do.  We multiply by
composing the homeomorphism and we write the composition from left
to right.

We define another relation.

Let \(A\) be a finite, binary tree with \(n\ge i\) leaves and let
\(u\) be the address of the \(i\)-th leaf in the left-right order.
We let \(A\caret^i\) denote the finite, binary tree consisting of
\(A\) and the two extra vertices \(u0\) and \(u1\).  

If \((A,B)\) and \((A',B')\) are two pairs of finite, binary trees
then we write \((A,B)\rightarrow (A',B')\) if there is an \(i\) so
that \(A'=A\caret^i\) and \(B'=B\caret^i\).  We let \(\sim\) be the
equivalence relation generated by \(\rightarrow\).  The equivalence
classes are the elements of \(F\).  It is shown in \cite{CFP}
(immediately after the proof of Lemma 2.2 of \cite{CFP})  that
\(\sim\) is identical to \(h\)-equivalence.  The group \(F\) is
commonly defined as the set of equivalence classes of pairs of
finite, binary trees under the relations just defined.

\subsection{Carets}

It is convenient to call a graph looking like \(\xy(0,-1);
(2,1)**@{-}; (4,-1)**@{-};\endxy\) a {\itshape caret}\index{caret}.
To make it a formal tree, a root and root edge must be added so that
it looks like \(\xy(0,-2); (2,0)**@{-}; (4,-2)**@{-}; (2,0);
(2,2)**@{-};\endxy\).  We will do so because then binary trees
become unions of carets.

In a binary tree containing vertices \(w\), \(u\), \(u0\) and \(u1\)
with \(w\) the parent of \(u\), the four element set \(\{w, u, u0,
u1\}\) forms a caret.  We will sometimes refer to \(u\) as the
{\itshape center vertex}\index{caret!center vertex}\index{center
vertex!caret}\index{vertex!center of caret} of the caret and say
that the caret is {\itshape centered} at \(u\).  In a finite binary
tree, the number of carets is the number of internal vertices.
Every finite, binary tree is now a union of a finite number of
carets if we agree that the trivial tree is a union of zero
carets.\index{tree!binary!as union of carets}\index{binary!tree!as
union of carets}\index{carets!tree union of}

We can give words to the construction \(A\caret^i\) and say that
\(A\caret^i\) is obtained from \(A\) by ``adding a caret to the
\(i\)-th leaf of \(A\).''  We can describe \((A,B)\rightarrow
(A',B')\) by saying that we have added a caret to the \(i\)-th
leaves of both \(A\) and \(B\).

\subsection{Reduced pairs}\mylabel{RedPairsSec}

A pair \((D,R)\) of finite, binary trees representing an element of
\(F\) is said to be {\itshape reduced}\index{reduced!tree
pair}\index{pair!of finite trees!reduced} if there is no pair
\((A,B)\) with \((A,B)\rightarrow (D,R)\).  Should such a pair
\((A,B)\) exist, then passing from \((D,R)\) to \((A,B)\) is called
a {\itshape reduction}\index{reduction!of tree
pair}\index{pair!of finite trees!reduction}.

Reductions are findable.  In a finite, binary tree, let us call a
caret whose two leaves are both leaves of the tree an {\itshape
exposed caret}\index{exposed caret}\index{caret!exposed}.  If
\((A,B)\rightarrow (D,R)\), then there will be an exposed caret in
\(D\) whose leaves occupy the same positions in the left-right order
of the leaves of \(D\) as the leaves of an exposed caret in \(R\).
Conversely, if such ``matching'' exposed carets exist in a pair of
trees \((D,R)\), then there is a pair \((A,B)\) such that
\((A,B)\rightarrow (D,R)\).

Each element of \(F\) is represented by a unique reduced pair (shown
following Lemma 2.2 in \cite{CFP}) and any pair will be reduced to a
reduced pair in the same equivalence class by any random sequence of
reductions taken as far as possible.

\subsection{The multiplication}

\index{Thompson's group \protect\(F\protect\)!multiplication}Since
we compose the homeomorphisms of \([0,1]\) that form the elements of
\(F\) from left to right, it is clear that if \((A,B)\) and
\((B,C)\) are two pairs of binary trees, then \((A,B)(B,C)=(A,C)\)
as elements of \(F\).  We need to deal with the more general
situation \((A,B)(C,D)\).  Using the equivalence relation \(\sim\)
derived in Section \ref{FDefsSec} from the relation \(\rightarrow\),
we replace \((A,B)(C,D)\) by \((A',B')(C',D')\) where
\((A,B)\sim(A',B')\), where \((C,D)\sim(C',D')\), and where
\(B'=C'\).  Now of course the product is \((A',D')\).

The required equality \(B'=C'\) can be achieved because the relation
\(\rightarrow\) allows us to add a caret at any leaf we please and
then, if desired, add more.  In particular, we can get \(B'=C'=B
\cup C\).  (See the remarks at the end of Section
\ref{MathcalTSec}.)

The multiplication is algorithmic (and the word problem is solvable)
in the sense that if \(f\), \(g\) and \(h\) are given, then it is
possible to tell if \(fg=h\) is true.  One computes a pair for
\(fg\) as just described, and then reduces the result and reduces
any pair for \(h\) using the comments in Section \ref{RedPairsSec}.
The equality \(fg=h\) holds if and only if the resulting reduced
pairs are identical.

An illustration is essential here.  Below we show the calculation of
the square of the smallest tree pair that represents the element
\(\rot{\emptyset}\).  We omit the root \(*\) and the root edge to
simplify the figures.
\mymargin{ExamplMult}\begin{equation}\label{ExamplMult}
\begin{split}
&\left(
\xy
(-3,-3); (3,3)**@{-}; (6,0)**@{-}; (0,0); (3,-3)**@{-};
\endxy
\ \ ,\ \ 
\xy
(3,-3); (-3,3)**@{-}; (-6,0)**@{-}; (0,0); (-3,-3)**@{-};
\endxy
\right)
\left(
\xy
(-3,-3); (3,3)**@{-}; (6,0)**@{-}; (0,0); (3,-3)**@{-};
\endxy
\ \ ,\ \ 
\xy
(3,-3); (-3,3)**@{-}; (-6,0)**@{-}; (0,0); (-3,-3)**@{-};
\endxy
\right) \\
= &
\left(
\xy
(-3,-4.5); (6,4.5)**@{-}; (9,1.5)**@{-};
(0,-1.5); (3,-4.5)**@{-};
(3,1.5); (6,-1.5)**@{-};
\endxy
\ \ ,\ \ 
\xy
(0,-1.5); (6,4.5)**@{-}; (12,-1.5)**@{-};
(3,1.5); (5,-1.5)**@{-}; (7,-1.5); (9,1.5)**@{-};
\endxy
\right)
\left(
\xy
(0,-1.5); (6,4.5)**@{-}; (12,-1.5)**@{-};
(3,1.5); (5,-1.5)**@{-}; (7,-1.5); (9,1.5)**@{-};
\endxy
\ \ ,\ \ 
\xy
(3,-4.5); (-6,4.5)**@{-}; (-9,1.5)**@{-};
(0,-1.5); (-3,-4.5)**@{-};
(-3,1.5); (-6,-1.5)**@{-};
\endxy
\right) \\
= &
\left(
\xy
(-3,-4.5); (6,4.5)**@{-}; (9,1.5)**@{-};
(0,-1.5); (3,-4.5)**@{-};
(3,1.5); (6,-1.5)**@{-};
\endxy
\ \ ,\ \ 
\xy
(3,-4.5); (-6,4.5)**@{-}; (-9,1.5)**@{-};
(0,-1.5); (-3,-4.5)**@{-};
(-3,1.5); (-6,-1.5)**@{-};
\endxy
\right)
\end{split}
\end{equation}
The first pair of trees on the first line is related by
\(\rightarrow\) to the first pair of trees on the second line.  The
second pairs of trees on the two lines are similarly related.  See
the end of Section \ref{FDefsSec}.

Regarding trees as unions of carets lets us give more detail to this
calculation.  We pause to discuss set operations on finite binary
trees before returning to the example.

\subsection{Unions, intersections and differences of
trees}\mylabel{UnionIntSec}

We elaborate on remarks made at the end of Section
\ref{MathcalTSec}.

If \(B\) and \(C\) are finite, binary trees in \(\mathcal T\), then
it is easier to work with expressions such as \(B-C\) and \(C-B\) if
we think of \(B\) and \(C\) as unions of carets.  The unions of
carets point of view makes no change when discussing \(B\cup C\) and
\(B\cap C\)\index{trees!binary!union, intersection,
difference}\index{binary trees!union, intersection, difference}, so
there is no disadvantage to this view.

We say that two carets are {\itshape adjacent} if they share an
edge.  If two carets are adjacent, there will be one shared edge and
it will be a root edge of one caret and a leaf edge of the other.
This allows us to talk about a component of a set of carets as the
union of a set of carets whose adjacency graph is maximally
connected.

If we define \(C-B\) to be the set of carets in \(C\) that are not
in \(B\), then \(C-B\) breaks into components with one component for
each leaf of \(B\cap C\) that is not a leaf of \(C\).  If \(v\) is
such a leaf, it is a leaf of \(B\) and the corresponding component
of \(C-B\) will be the subtree \(C_v\).  The subtree \(C_v\) will
share its root edge with the leaf edge of \(B\) (and of \(B\cap C\))
that impinges on \(v\).

Similar remarks apply to components of \(B-C\) and leaves of \(B\cap
C\) that are not leaves of \(B\).

We now have that \(B\cup C\) breaks into three parts, the tree
\(B\cap C\), a finite collection of trees making up \(C-B\) with one
tree for each leaf of \(B\cap C\) that is not a leaf of \(C\), and a
finite collection of trees making up \(B-C\) with one tree for each
leaf of \(B\cap C\) that is not a leaf of \(B\).  Any leaves of
\(B\cap C\) not yet accounted for are leaves of both \(B\) and
\(C\).

We can apply this to the calculation of the product \((A,B)(C,D)\).
This is equal to \((A',D')\) which we get from \((A', B \cup
C)(B\cup C, D')\) chosen so that \((A',B\cup C)\sim(A,B)\) and
\((B\cup C, D')\sim(C,D)\).  Now the components of \(A'-A\) form a
finite collection of trees that are in one-to-one correspondence
with the components of \(C-B\).  Let \(v\) be a leaf of \(B\) that
is not a leaf of \(C\) with component \(C_v\) attached to \(B\cap
C\) at \(v\) and assume \(v\) is the \(i\)-th leaf in the left-right
order of the leaves of \(B\).  Let \(v'\) be the \(i\)-th leaf in
the left-right order among the leaves of \(A\).  Then \(A'_{v'}\) is
the component of \(A'-A\) attached to \(A\) at \(v'\), is the
component of \(A'-A\) corresponding to \(C_v\), and is isomorphic to
\(C_v\).  Similar remarks apply to components of \(D'-D\).

In the calculation illustrated in \tref{ExamplMult}, let \(A, B, C,
D\) be the trees in the first line, reading left to right.  In the
move to the second line, the lone caret in \(C-B\) was attached to
the leftmost leaf of \(A\) and of \(B\), and the lone caret of
\(B-C\) was attached to the rightmost leaf of \(C\) and of \(D\).

\subsection{Another representation}\mylabel{AltRepSec}

Let \(\mathbf D=\Z[\frac12]\cap (0,1)\), the set of dyadic rationals
between 0 and 1.  Each \(x\in \mathbf D\) has a last ``1'' in its
base 2 representation.  The word \(w\) in \(\{0,1\}^*\) in the base
2 representation of \(x\) from the binary point up to but not
including the last 1 has \(m(w)=x\).  Thus \(m:\{0,1\}^*\rightarrow
\mathbf D\) is a bijection.  If we regard \(\{0,1\}^*\) as the
non-root nodes of \(\mathcal T\), if \(\mathcal T\) is given the
infix order, and if \(\mathbf D\) is given the usual order in
\([0,1]\), then \(m\) preserves order.  For example, the rationals
\(n/32\) for \(16\le n\le 21\) are the images, respectively, of
\(\emptyset\), \(1000\), \(100\), \(1001\), \(10\), and \(1010\)
under \(m\).

Since \(\mathbf D\) is dense in \([0,1]\) and an element \(h(D,R)\)
in \(F\) preserves \(\mathbf D\), the element \(h(D,R)\) is
determined by its action on \(\mathbf D\) and thus is also
determined by the corresponding action on the non-root vertices of
\(\mathcal T\).  We can extend the action to be the identity on the
root of \(\mathcal T\) so that we can stop saying ``non-root
vertex.''  Note that these actions are not tree homomorphisms.  We
denote the action given by a pair \((D,R)\) by \(t(D,R)\).

The ``other'' representation of \(F\) is as the set of permutations
\(t(D,R)\) on the vertices of \(\mathcal T\) as \((D,R)\) runs over
all pairs of finite, binary trees.  This will be our working version
of \(F\).  

Since this is the version that we work with and since it is
traditional to write elements of \(F\) as pairs of binary trees, we
will drop the \(t\) and write \((D,R)\) instead of \(t(D,R)\).  This
starts immediately.

\subsection{Actions on finite trees}\index{Thompson's group
\protect\(F\protect\)!action on trees}

We write the action of \((D,R)\) to the right of its arguments.

Noting that both \(D\) and \(R\) are rooted subtrees of \(\mathcal
T\), we see that the statement of the next lemma makes sense.  We
will also argue that it is true.

\begin{lemma}\mylabel{FactsRightLem}  Let  \((D,R)\)  be a  pair  of
finite,  binary  trees  with   the  same  number  of  leaves.   Then
\((D)(D,R)=R\).  Further  if the vertices  in both \(D\)  and \(R\)
are given  the infix  order, then the  restriction of  \((D,R)\) to
\(D\) preserves this order.  Lastly, \((D,R)\) takes leaves to
leaves and internal vertices to internal vertices.  \end{lemma}

\begin{proof} Since \(h(D,R)\) takes the \(j\)-th interval in
\(P(D)\) affinely onto the \(j\)-th interval in \(P(R)\), we have
that \((D,R)\) takes the \(j\) leaf of \(D\) to the \(j\)-th leaf of
\(R\).  Induction using \(\caret\) shows that, in the infix order of
any finite, binary tree, between any two consecutive leaves there is
a single internal vertex.  This immediately gives the last sentence
in the statement.  Further, let \(v_j\) and \(v_{j+1}\) be two
consecutive leaves of \(D\) in the infix order and let \(w\) be the
single internal vertex between them.  From the information about the
functions \(i\) and \(m\), we have that \(m(w)\) is the right
endpoint of \(i(v_j)\) and the left endpoint of \(i(v_{j+1})\).
Similar statements apply to the leaves of \(S\) that are the images
of \(v_j\) and \(v_{j+1}\), and so the image of \(w\) is the single
internal vertex of \(R\) between the images of \(v_j\) and
\(v_{j+1}\).  This completes all claims.  \end{proof}

\begin{cor}\mylabel{ActsLikeCatCor} For each \(d\ge1\), the vertices
of \(A_d\) form the objects of a category in which for each pair
\((D,R)\) of vertices, the only morphism is \((D,R)\).  \end{cor}

\begin{proof} Since each \((D,R)\) preserves a linear order, all
triangles to commute.  \end{proof}

\subsection{Rotation as action}\index{rotation!action of}

We apply the above to rotations.

\begin{cor}\mylabel{WhatRotDoesCor} If the rotation \(\rot{u}\)
takes \(T\) to \(S\) as shown in \tref{ARot}, then \((T,S)\) takes
\(u0\) and \(u\) in \(T\), respectively, to \(u\) and \(u1\) in
\(S\).  That is, \(\rot u\) takes the pivot vertices of \(\rot u\) to
the pivot vertices of \(\rot {\overline u}\) so as to preserve the
infix order.  \end{cor}

\begin{proof} The positions of \(u0\) and \(u\) in the infix order on
\(T\) are the positions, respectively, of \(u\) and \(u1\) in the
infix order on \(S\).  \end{proof}

Because of Corollary \ref{WhatRotDoesCor}, we identify the rotation
\(\rot{u}\) from \(T\) to \(S\) with the permutation \((T,S)\).  In
the next proposition, we make use of the notation \(T^\sigma\) to
indicate a tree with sign assignment \(\sigma\).

\begin{prop}\mylabel{RotOnColorProp} Let \(\rot{u}\) be the positive
signed rotation pictured in \tref{CRot}, let \(T^\sigma\) denote the
left tree with its sign assignment, and let \(S^\rho\) denote the
right tree with its sign assignment.  Then for all internal
vertices \(v\) of \(T\), the rotation \(\rot u\) preserves all signs
except those of the pivot vertices of \(\rot u\) which it negates.
Specifically
\[
\rho(v\rot u) = 
\begin{cases}
\sigma(v), &v\notin \{u,u0\}, \\
-\sigma(v), &v\in \{u,u0\}.
\end{cases}
\]
Similarly, the rotation \(\rot{\overline u}\) preserves the signs of
all internal vertices of \(S\) except for negating the signs of the
pivot vertices of \(\rot {\overline u}\).  \end{prop}

We define the action of \(\rot u\) on \(T^\sigma\) to have the
action of \(\rot u\) on the vertices of \(T\) as given by Corollary
\ref{WhatRotDoesCor} and the action of \(\rot u\) taking \(\sigma\)
to \(\rho\) as given by Proposition \ref{RotOnColorProp}.  We can
thus say \(T^\sigma\rot u=T^\rho\).

Note that if \(T\) is a tree with sign assignment \(\sigma\) and the
pivot vertices of \(\rot u\) are internal vertices of \(T\), then
\(T^\sigma\rot u\) can be defined using the specifications given in
Proposition \ref{RotOnColorProp} even if \(\sigma\) does not give
the same values to the pivot vertices of \(\rot u\).  That is, the
action can be defined even if \(\rot u\) is not valid for
\(T^\sigma\).

\subsection{Edge paths in the associahedra}

If a pair \((D,R)\) of finite, binary trees each of which has
\(d+2\) leaves represents an element of \(F\), then it also gives a
pair of vertices in the associahedron \(A_d\).  Since each edge of
\(A_d\) is a rotation, an edge path in the 1-skeleton of \(A_d\)
gives a finite string of rotations.  If \(\rot{u_1}\rot{u_2}\cdots
\rot{u_n}\) is such a path from \(D\) to \(R\), then it follows from
Lemma \ref{FactsRightLem} and the convention given after Corollary
\ref{WhatRotDoesCor} that the product \(\rot{u_1}\rot{u_2}\cdots
\rot{u_n}\) in \(F\) equals \((D,R)\) as an element of \(F\).

\subsection{Presentations}

From the statement that the 1-skeleton of each associahedron is
connected, we get that the rotations generate \(F\).  We can discuss
relations.

We need some notation.  Let \(u\) and \(v\) be elements of
\(\{0,1\}^*\).  If \(u\) is a prefix of \(v\), we write \(u|v\).  If
neither of \(u\) nor \(v\) is a prefix of the other, we write
\(u\perp v\).

We can now write specifically the action of the rotation \(\rot{u}\) on
\(\{0,1\}^*\).  We use \(w\) to represent an arbitrary element of
\(\{0,1\}^*\).
\[
v\rot{u} = 
\begin{cases}
v, & u\perp v, \\
v, & v|u, \\
u0w, & v=u00w, \\
u10w, & v=u01w, \\
u11w, & v=u1w, \\
u, & v=u0, \\
u1, & v=u.
\end{cases}
\]

We leave it to the reader to verify certain relations.  The
more straightforward of the relations are conjugacy relations.
Because of the first two lines in the definition of \(v\rot{u}\), the
``support'' of \(\rot{u}\) is all \(v\) where \(u|v\).  Thus a nice
conjugacy relation exists between \(\rot{u}\) and \(\rot{v}\) if the support
of \(\rot{v}\) is treated nicely by \(\rot{u}\).  This leads to the
following which hold for all \(u\) and \(v\) as specified and are
easy to verify.  We use \(a^b\) for \(b^{-1}ab\).
{\begin{enumerate}
\renewcommand{\theenumi}{F\arabic{enumi}}

\item \(\rot{v}^{\rot{u}} = \rot{v}\) if \(v\perp u\).

\item \(\rot{v}^{\rot{u}} = \rot{v\rot{u}}\) if \(u|v\) and \(v\) is
neither \(u0\) nor \(u\).

\end{enumerate}
}

There are relations between \(\rot{u}\) and \(\rot{u0}\).  The last set
should be compared to the copy of \(A_2\) shown in
\tref{ThePentagon} in Section \ref{AssocExmplSec}.
{\begin{enumerate}
\renewcommand{\theenumi}{F\arabic{enumi}}
\setcounter{enumi}{2}
\item \(\rot{u0}\rot{u}\rot{u1} = \rot{u}\rot{u}\).
\end{enumerate}
}

It is shown in \cite{dehornoy:geom-presentations} that \(F\) is
presented\index{Thompson's group
\protect\(F\protect\)!presentations} with the set of all \(\rot{u}\)
for \(u\in \{0,1\}^*\) as generating set and with all relations in
(F1), (F2) and (F3) as the relation set.  The generators in \(\{\rot
u\mid u\in \{0,1\}^*\}\) are called the {\itshape symmetric
generators} in \cite{dehornoy:geom-presentations} and we shall do
the same.


The argument in \cite{dehornoy:geom-presentations} assumes little
about \(F\).  A different proof that uses more information about
\(F\) is to argue that the relations above imply that \(F\) is
generated by the set \(\rot{1^*}\).  In fact, one shows that \(F\)
is generated by \(\rot{\emptyset}\) and \(\rot{1}\).  The relations
in (F2) contain all relations needed to present \(F\) using either
\(\rot{1^*}\) as generators or \(\rot{\emptyset}\) and \(\rot{1}\).
See \cite{CFP}.  (The rotation \(\rot{1^i}\) is the inverse of the
generator denoted \(X_i\) in Figure 8 of \cite{CFP}.)  This makes
our representation of \(F\) a homomorphic image of the group
presented in \cite{CFP}.  It is a fact that \(F\) has no proper
non-abelian quotients (Theorem 4.3 of \cite{CFP}).  That is, any
non-trivial normal subgroup of \(F\) contains the commutator
subgroup of \(F\).  Since our representation is non-abelian, the
homomorphism from the group in \cite{CFP} to our representation is
an isomorphism.

\subsection{Sizes of trees in tree pairs}

The notion of size is slippery in \(F\).  Since an element of \(F\)
is an equivalence class of pairs of trees, an obvious candidate is
some measure of tree pairs that is minimized over all pairs that
represent an element.  A typical measure is the number of carets
(which equals the number of internal vertices) used in each tree of
the tree pair.  This will be minimized by the unique reduced pair
representing an element.

This measure is only moderately well behaved.  In general it is
complicated to predict the size of the result of a multiplication
without actually doing the multiplication.  In particular it is
complicated to predict the size of the result of multiplying an
element by a rotation (one of the generators used above) without
actually doing the multiplication.

We deal with this by partly ignoring it.  We will not be concerned
with minimal representatives of elements of \(F\) and only concerned
with the size of the representative that we have in hand.  A single
pair \((D,R)\) is a pair of vertices on a specific associahedron
\(A_d\).  Edges in the associahedron \(A_d\) are rotations that
apply to trees of the size of \(D\) and \(R\).  Only finitely many
of the generators in \(\{\rot u\mid u\in \{0,1\}^*\}\) are available
among the edges of \(A_d\).  If a generator outside this finite set
is to be used, then a different (larger) representative will need to
be chosen (a move to a larger associahedron takes place) and this
will be duly noted.  These considerations will appear below in
Section \ref{RotOnColPairSec}.

\subsection{The group \protect\(E\protect\)}

We now describe the other group we consider.  

The relations in (F1) and (F2) involve four generators each and will
be referred to as the ``square
relations,''\index{square!relation}\index{relation!square} while the
relations in (F3) involve five generators and will be referred to
as the ``pentagonal
relations.''\index{pentagonal!relation}\index{relation!pentagonal}
We will see that the pentagonal relations behave very differently
from the square relations when the generators are interpreted as
signed rotations.

We define the group \(E\)\index{\protect\(E\protect\)}\index{group
\protect\(E\protect\)} as the group presented by the same generating
set as \(F\) and with the relations given by (F1) and (F2).  The
group \(E\) is finitely presented, but that fact is neither
surprising nor important here.  The group \(E\) is an extension of
\(F\).  We use the letter \(E\) since it is the first letter of
``extension,'' since it is next to \(F\) in the alphabet, and since
the other neighbor of \(F\) in the alphabet has already been used
for groups closely related to \(F\).

The group \(E\) has been mentioned in conversations by several
people familiar with \(F\) but there has never been any compelling
reason to consider it.  Part of the charm of Thompson's groups or
its close relatives is that they are either simple or close to it.
Groups that are farther from simple are often not considered.

We will bring the group \(E\) into the discussion when we
investigate valid paths of signed rotations in Section
\ref{SignStructSec}.

\section{Colors and the group operations}\mylabel{GpOpSec}

We extract some information about colorings from the multiplicative
structure of \(F\).  We ultimately discover that most of the
difficulty boils down to problems that live in the associahedra.

\subsection{Coloring representatives of a single element of
\protect\(F\protect\)}

Representatives of elements of \(F\) are pairs of finite, binary
trees.  We say that such a pair has a coloring if there is a color
vector valid for the pair.  From Proposition \ref{TreePairEquivThm},
the 4CT implies that every representative of every element of \(F\)
has a coloring.  For now, we do not assume the 4CT.

If \((A,B)\rightarrow (A',B')\), then the maps \(M\) and \(M'\)
that, respectively,
correspond to the pairs differ in one respect shown below.  The
figure only shows where they differ.
\[
\xy
(0,-10.5); (0,10.5)**@{-}; (2,0)*{a};
(0,-13.5)*{M};
\endxy
\qquad\qquad
\xy
(0,-10.5); (0,-4.5)**@{-}; (-3,-1.5)**@{-}; (-3,1.5)**@{-};
(0,4.5)**@{-}; (0,10.5)**@{-};
(0,-4.5); (3,-1.5)**@{-}; (3,1.5)**@{-};
(0,4.5)**@{-}; (-5,0)*{b}; (5,0)*{c};
(2,7.5)*{a}; (2,-7.5)*{a};
(0,-13.5)*{M'};
\endxy
\]
The letters show the only possible arrangements of a proper, edge
3-coloring in the portion shown.  It follows immediately that \(M\)
has a proper, edge 3-coloring if and only if \(M'\) has a proper,
edge 3-coloring, independent of the 4CT.  Thus \((A,B)\) has a
coloring if and only if \((A',B')\) has a coloring, independent of
the 4CT.

From this it follows that (independent of the 4CT) there is a well
defined notion of an element of \(F\) having a
coloring.\index{coloring!of element of
\protect\(F\protect\)}\index{Thompson's group
\protect\(F\protect\)!coloring of element} We have the next variant
of Proposition \ref{TreePairEquivThm}.

\begin{prop}\mylabel{FEquivThm} The four color theorem is equivalent
to the statement that every element of \(F\) has a coloring.
\end{prop}

\subsection{Extending the set of colored elements}

In Corollary 2.6 of \cite{CFP} it is shown that \(\rot{\emptyset}\)
and \(\rot{1}\) generate \(F\).  We know that the set of elements of
\(F\) with colorings is non-empty.  The identity certainly has a
coloring.  If a non-identity element is sought, then Lemma
\ref{ColorX0Lem} says that \(\rot{\emptyset}\) has a coloring.  It
is trivial to show that \(\rot1\) has a coloring.  Since every
element of \(F\) is a product of copies of \(\rot\emptyset\),
\(\rot1\) and their inverses, and a coloring of an element is also a
coloring of the inverse of that element, we have the following.

\begin{thm}\mylabel{FGenEquivThm} The four color theorem is
equivalent to the statement that the set of elements of \(F\) with
colorings is closed under right multiplication by \(\rot\emptyset\),
\(\rot{\overline \emptyset}\), \(\rot1\) and \(\rot{\overline1}\).
\end{thm}

We will see later that this theorem has limited effect.

We note that the set of elements of \(F\) with colorings is clearly
closed under inversion (independent of the 4CT).  The dihedral group
of order \(2(d+3)\) acts on \(A_d\) and thus on pairs of its
vertices (elements of \(F\)).  The set of elements of \(F\) with
colorings is also closed under this action (independent of the 4CT).
We have not explored how much this extends the set of colored
elements of \(F\), but there are reasons for not expecting much from
such extension.  See the remarks after \tref{NoIncChain} below.

\subsection{The compatibility lemma}

Consider a product \((A,B)(C,D)\) of elements of \(F\).  Assume that
\(\mathbf a\) is a color vector valid for \((A,B)\) and that
\(\mathbf c\) is a color vector valid for \((C,D)\).  The vector
\(\mathbf a\) determines a proper, edge 3-coloring of \(B\), and
\(\mathbf c\) determines a proper, edge 3-coloring of \(C\).  We say
that \(\mathbf a\) and \(\mathbf c\) are compatible on \(B\cap C\)
if the colorings they determine on \(B\) and \(C\), respectively,
agree on \(B\cap C\).

\begin{lemma}[Compatibility]\mylabel{CompatLem} If the vector
\(\mathbf a\) is valid for the pair \((A,B)\), the vector \(\mathbf
c\) is valid for \((C,D)\), and the vectors are compatible on
\(B\cap C\), then there is a coloring of
\((A,B)(C,D)\).\index{Compatibility Lemma}\index{lemma!Compatibility}
\end{lemma}

\begin{proof} We will refer to color vectors as indexed by leaves
rather than indexed by integers.  This suffices to specify the
vector.

We let the vector \(\mathbf d\) indexed over the leaves of \(B\cup
C\) have the color of \(\mathbf a\) on a leaf of \(B\cup C\) that is
a leaf of \(B\) but not a leaf of \(C\) and have the color of
\(\mathbf c\) on the remaining leaves.  Note that the second choice
agrees with the color from \(\mathbf a\) on the leaves of \(B\cup
C\) that are leaves of both \(B\) and \(C\) since such leaves are
also leaves of \(B\cap C\).

The discussion that follows is based on the discussion in Section
\ref{UnionIntSec}.  We view the product \((A,B)(C,D)\) as computed
by \((A',B\cup C)(B\cup C,D')\) as describe in that section.

For a component \(C_v\) of \(C-B\) that is attached to \(B\cap C\)
at leaf \(v\) of \(B\), the colors of \(\mathbf d\) on those leaves
are those of \(\mathbf c\).  Thus the root color of \(C_v\) from
\(\mathbf d\) is the root color of \(C_v\) from \(\mathbf c\).  But
the root edge of \(C_v\) is the leaf edge of \(B\cap C\) impinging
on \(v\).  By assumption, this edge gets the same color from
\(\mathbf a\) and \(\mathbf c\).  This common color is also the
color that \(\mathbf a\) assigns to the leaf edge in \(A\) for the
leaf \(v'\) that has the same left-right position in \(A\) as \(v\)
has in \(B\).  But \(A'\) has an isomorphic copy of \(C_v\) attached
to \(A\) at \(v'\) and the color derived from \(\mathbf d\) at the
leaf edge to \(v'\) will agree with the color assigned there by
\(\mathbf a\).  It follows that the edge coloring on \(A'\) derived
from \(\mathbf d\) will agree with the color on \(A\) derived from
\(\mathbf a\) and agree with the colors on components of \(A'-A\)
derived from \(\mathbf c\).  Thus the color vector \(\mathbf d\)
will be valid for \(A\).  Similar remarks apply to \(D'\).
\end{proof}

\subsection{Vines}\mylabel{VineSec}

The trees in this section are all finite, and binary.

A vine\index{vine} is a tree with exactly one exposed
caret.  Note that the adjacency graph of the carets in a vine will
have no branches.

If \(T\) is a tree and \(V\) is a vine, then \(T\cap V\) and \(V-T\)
are vines.  In particular, \(V-T\) is connected.

The right vine\index{vine!right}\index{right!vine} is a vine where
all the internal vertices are of the form \(1^*\).  The reader can
define the left vine\index{vine!left}\index{left!vine}, but we will
rarely refer to it.

If \(v\) is a vertex in \(\mathcal T\), then there is a unique vine
\(V_v\) in \(\mathcal T\) which has \(v\) as an internal vertex and
for which the children of \(v\) are the exposed leaves of \(V_v\).
Note that with this notation, we have that the rotation \(\rot u\)
is given by the pair \((V_{u0}, V_{u1})\).  Note also that \(V_{u0}
\cap V_{u1} = V_u\).

\subsection{Action of rotations on colored
pairs}\mylabel{RotOnColPairSec}

In the next discussion a rotation such as \(\rot u\) will be viewed
as both an element of \(F\) (and thus represented as an ordered pair
of trees) and an action on the vertices of \(\mathcal T\) and some
of its finite subtrees.\index{rotation!action of!on colored pairs}

Let a pair \((D,R)\) of finite, binary trees represent an element of
\(F\) and assume that \(\mathbf c\) is a valid color vector for the
pair.  We are interested in the influence of \(\mathbf c\) on
\((D,R)\rot u\) for some rotation \(\rot u = (V_{u0}, V_{u1})\).
The discussion of \(\rot{\overline u}\) is similar.

There are several cases depending on the relationship between
\(V_{u0}\) and \(R\).

If \(V_{u0}\) is a subtree of \(R\), then it is immediate that
\((D,R)\rot u = (D,R\rot u)\) in which both views of \(\rot u\) are
used.  From Section \ref{ColRotSec} we know that \(\mathbf c\) is
valid for the pair \((R,R\rot u)\) if and only if the signs of \(u\)
and \(u0\) agree in \(R\) under \(\mathbf c\).  Thus \(\mathbf c\)
is valid for \((D,R)\rot u = (D,R)(V_{u0}, V_{u1})\) if and only if
the signs of \(u\) and \(u0\) agree in \(R\) under \(\mathbf c\).

We turn to the case where \(V_{u0}\) is not a subtree of \(R\).

If \(V_{u0}\) is not a subtree of \(R\), then the behavior depends
on the number of carets of \(V_{u0}-R\).  We consider first the case
in which \(V_{u0}-R\) consists of only one caret.

If \(V_{u0}-R\) consists of one caret, then this caret is centered
at \(u0\), and \(u0\) is a leaf of \(R\).  If \(u0\) is the \(i\)-th
leaf of \(R\) in left-right order, then it is immediate that
\((D,R)\rot u = (D\caret^i, (R\caret^i)\rot u)\) and we become
interested in the pair \((R\caret^i, (R\caret^i)\rot u)\) shown
below in which \(\epsilon = -\delta\).
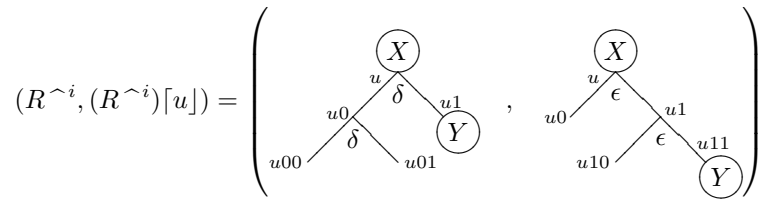
\[
(R\caret^i, (R\caret^i)\rot u) = 
\left(
\xy
(0,-8); (12,4)**@{-}; (18,-2)**@{-}; (6,-2); (12,-8)**@{-};
(12,6.8)*{X}*\cir<8pt>{};
(20,-4)*{Y}*\cir<8pt>{};
(-3,-8)*{\scs u00}; (4,-1.5)*{\scs u0};
(15,-8)*{\scs u01}; (9,3)*{\scs u};
(19,0)*{\scs u1};
(12,1)*{ \delta}; (6,-5)*{ \delta};
\endxy
\quad,\quad
\xy
(0,-8); (-12,4)**@{-}; (-18,-2)**@{-}; (-6,-2); (-12,-8)**@{-};
(-12,6.8)*{X}*\cir<8pt>{};
(2,-10)*{Y}*\cir<8pt>{};
(1,-5.5)*{\scs u11}; (-4,-1)*{\scs u1};
(-15,-8)*{\scs u10}; (-15,3)*{\scs u};
(-20,-2)*{\scs u0};
(-12,1)*{ \epsilon}; (-6,-5)*{ \epsilon};
\endxy
\right)
\]

Extending the color of \(\mathbf c\) to \(R\caret^i\) can be done in
either of two ways.  However, if the extension is to be valid for
\((R\caret^i, (R\caret^i)\rot u)\), then the signs at \(u\) and
\(u0\) must agree.  Since \(u\) is internal to \(R\), its sign is
fixed.  Thus only one of the two ways to extend \(\mathbf c\) to
\(R\caret^i\) is valid for the pair.  This extension will be valid
for \((D,R)\rot u\).  We thus have a process that gives a coloring
to \((D,R)\rot u\) from a coloring of \((D,R)\) under the assumption
that \(V_{u0}-R\) consists of one caret.

Still under the assumption that \(V_{u0}-R\) consists of one caret,
let us assume that  \((D,R)\rot u = (D\caret^i,
(R\caret^i)\rot u)\) has a coloring.  Since the caret in
\(D\caret^i\) attached at the \(i\)-th leaf of \(D\) is exposed, its
leaf edges must be assigned different colors.  It follows that the
leaf edges to \(u0\) and \(u10\) in \((R\caret^i)\rot u\) are
assigned different colors, and from that it follows that the signs
at \(u\) and \(u1\) in \((R\caret^i)\rot u\) are equal.  By
considerations similar to the paragraph above, we can build a
coloring valid for \((D,R)\).  

The processes of the previous two paragraphs are inverses to each
other and under the assumption that \(V_{u0}-R\) consists of one
caret, we have a one-to-one correspondence between the colorings
valid for \((D,R)\) and the colorings valid for \((D,R)\rot u\).
This recovers Proposition 11 of \Zeilberger.

We leave it to the reader to show that if \(V_{u0}-R\) consists of
more than one caret, then a coloring for \((D,R)\) can be converted,
non-uniquely, to a coloring of \((D,R)\rot u\).

We note that if \(V_{u0}-R\) is empty, then the number of carets
involved in \((D,R)\) is the same as the number of carets involved
in \((D,R)\rot u = (D,R\rot u)\).  If \(V_{u0}-R\) is not empty,
then the number of carets in \((D,R)\rot u\) is greater than the
number of carets in \((D,R)\).  Note that we are ignoring the fact
that different pairs representing the same element of \(F\) may use
different numbers of carets.  Here we are referring to pairs and not
elements of \(F\).

We say that multiplication \((D,R)\rot u\) is {\itshape
increasing}\index{increasing!multiplication}\index{multiplication!increasing}
if \(V_{u0}-R\) is not empty, and {\itshape minimally
increasing}\index{minimally
increasing!multiplication}\index{multiplication!minimally
increasing} if \(V_{u0}-R\) consists of exactly one caret.  We say a
chain \(\rot{u_1}\cdots \rot{u_n}\) of rotations is
increasing\index{increasing!chain of multiplications}\index{chain!of
increasing multiplications} if letting \(p_i=\rot{u_1}\cdots
\rot{u_i}\) makes \(p_i\rot{u_{i+1}}\) increasing for \(1\le i<n\).
We say the chain is minimally increasing\index{minimally
increasing!chain of multiplications}\index{chain!of minimally
increasing multiplications} if each multiplication is minimally
increasing.

In the above, similar analysis and terminology applies to
multiplication by \(\rot{\overline{u}}\) and also rotations (by
either \(\rot u\) or \(\rot{\overline u}\)) applied on the left.
(Earlier we said actions are written on the right, but \(\rot u\)
and \(\rot{\overline{u}}\) are also elements of \(F\) and can be
multiplied on the left of another element of \(F\).)

From Lemma \ref{ColorX0Lem}, we know that \(\rot \emptyset\) and
\(\rot{\overline\emptyset}\) each have a unique coloring.  All other
rotations have colorings which are not unique.  Recall from Section
\ref{PermColorSec} that we count colorings modulo permutations of
the colors.  From this and the discussions above, we get the
following.

\begin{thm}\mylabel{IncrChainThm} If an element of \(F\) is
represented as a minimally increasing chain starting with
\(\rot\emptyset\), then it has a unique coloring.  If an element of
\(F\) is represented as an increasing chain, then it has a coloring.
\end{thm}

The limitations of this result can be seen by noting that the result
of an increasing multiplication by a rotation must have an exposed
caret of one tree have its leaves match to leaves of adjacent carets
(carets that share an edge) in the other tree.  The following shows
that not every tree pair has this configuration.
\mymargin{NoIncChain}\begin{equation}\label{NoIncChain}
\left(\,
\xy
(0,1.5); (3,4.5)**@{-}; (6,1.5)**@{-}; (3,4.5); (3,6,5)**@{-};
(0,-4.5); (6,1.5)**@{-}; (12,-4.5)**@{-};
(3,-1.5); (5,-4.5)**@{-}; (7,-4.5); (9,-1.5)**@{-};
\endxy
\quad, \quad
\xy
(12,1.5); (9,4.5)**@{-}; (6,1.5)**@{-};  (9,4.5); (9,6,5)**@{-};
(0,-4.5); (6,1.5)**@{-}; (12,-4.5)**@{-};
(3,-1.5); (5,-4.5)**@{-}; (7,-4.5); (9,-1.5)**@{-};
\endxy\,
\right)
\end{equation}

The trees in \tref{NoIncChain} are those in \tref{SmallSymExmpl}.
It follows that the action of the dihedral group of order 12 by root
shifts and reflections can only take the element in
\tref{NoIncChain} to itself or its inverse.  In particular the
element in \tref{NoIncChain} is the image of no element under the
action other than itself or its inverse.

\subsection{Application to specific subsets of 
\protect\(F\protect\)}\mylabel{RotAppSec}

The set of elements of \(F\) of the form \((D,V)\) where \(V\) is a
right vine forms a monoid of \(F\) that is often called the positive
monoid\index{Thompson's group \protect\(F\protect\)!positive
monoid}\index{positive!monoid of \protect\(F\protect\)} of \(F\).
This is because any element \((D,R)\) of \(F\) is of the form
\((D,V)(V,R)\) and is thus of the form \(pn^{-1}\) where \(p\) and
\(n\) are from the positive monoid.  Elements of the positive
monoid are referred to as positive elements\index{positive!element
of \protect\(F\protect\)} of \(F\).

With prime as defined at the end of Section \ref{TreePairToMapSec},
we have that the prime, positive elements of \(F\) are of the form
\((D,V)\) where \(V\) is a right vine and the right subtree of \(D\)
is trivial.

\begin{thm}\mylabel{PrimPosUniqThm} The positive elements of \(F\)
have colorings and the prime positive elements of \(F\) have unique
colorings.  \end{thm}

\begin{proof} We leave it as an easy exercise to argue that positive
elements of \(F\) are represented as increasing chains and that the
prime positive elements are represented as minimally increasing
chains.  \end{proof}

Theorem \ref{PrimPosUniqThm} has a very straightforward proof using
the dual (triangulations) view.  See Section \ref{DualTriSec}.  We
leave this an exercise to the reader.

Computer calculations indicate that the fraction of the uniquely
colored elements of \(F\) represented by positive primes goes to
zero as the size of the trees in the pair increases.

We note that the cardinality portions of the conclusions of Theorems
4, 6 and Proposition 11 of \Zeilberger{} follow from Theorem
\ref{PrimPosUniqThm}.  The cardinality parts of the conclusions of
Theorems 5, 7, 8, 10 and 12 of \Zeilberger{} follow from Theorem
\ref{PrimPosUniqThm} and \tref{PrimeProdFormula} in Section
\ref{TwoRedSec}.

\subsection{Further considerations}

From Theorem \ref{IncrChainThm} we see that the difficulties in
attempting a proof of the 4CT by coloring pairs of trees are found
in deducing the existence of a coloring of a tree pair from a
coloring of a pair of the same size.  Proposition \ref{RotChainProp}
is our approach to this problem.  For a given tree pair, Proposition
\ref{RotChainProp} and its companion, Conjecture \ref{SPathConjOne},
live on the associahedron that has the trees of the pair as
vertices.

We thus concentrate on edge paths in an associahedron.  Rigid
colorings are beyond the edge path approach.  These were discussed
in Section \ref{BasicRigidSec} and will be discussed more in Section
\ref{RigidSec}.  Then Section \ref{ZeroSetSec} investigates edge
paths under the assumption that the colorings are not rigid.  The
study of edge paths is completed in Section \ref{SignStructSec}
which gives the argument, promised after the statement of
Proposition \ref{RotChainProp}, that the ``validity'' of an edge
path can be detected without bringing in a sign assignment of a
tree to test it on.

\section{Rigid colorings}\mylabel{RigidSec}

\subsection{Depth condition}

Let \(w\) be a vertex in \(\mathcal T\) other than the root \(*\)
regarded as a word in the alphabet \(\{0,1\}\).  We use \(|w|\) to
denote the length of \(w\) as a word.  We refer to \(|w|\) as the
level or depth of \(w\)\index{level!vertex in
tree}\index{tree!vertex!level}\index{vertex!of
tree!level}\index{depth!vertex in
tree}\index{tree!vertex!depth}\index{vertex!of tree!depth}.
If \((D,R)\) is a pair of binary trees of \(n\) leaves each, with
\((d_1,d_2, \ldots, d_n)\) and \((r_1,r_2, \ldots, r_n)\) the
leaves, respectively, of \(D\) and \(R\) in left-right order, then
we say that the pair satisfies the {\itshape parity
condition}\index{parity condition!on pair of trees}\index{pair!of
finite trees!parity condition} if for \(1\le i\le n\), the parity of
\(|d_i|\) and \(|r_i|\) agree.

\begin{prop}\mylabel{DepthCondProp} Let \((D,R)\) be a pair of
finite, binary trees.  Then \((D,R)\) has a rigid coloring if and
only if \((D,R)\) satisfies the parity condition.  \end{prop}

\begin{proof} From Lemma \ref{DualRigidLem}, we know that \((D,R)\)
has a rigid coloring if and only if the map created from \(D\cup R\)
as described in Section \ref{TreePairToMapSec} has a proper, face
3-coloring.  From Theorem 2-5 of \Saaty, we know that this happens
if and only if every face of the map has an even number of edges.
If we apply this to the two faces containing the root edges of \(D\)
and \(R\), we see that the parity condition holds for \(i=1\) and
\(i=n\).  The rest follow inductively from \(i=1\) by consideration
of the other faces.  \end{proof}

It is not possible to have a pair of binary trees with notation as
in the definition of the parity condition where for every \(i\) with
\(1\le i\le n\), the parity of \(|d_i|\) and \(|r_i|\) disagree.

\subsection{The positive rigid pattern}

Let \((D,R)\) be a pair of trees with a rigid coloring.  Then by
Proposition \ref{DisjColorsProp} they possess a common positive
rigid coloring and a common negative rigid coloring.  One is just
the negative of the other and it is not worth discussing both.  So
we concentrate on the positive rigid coloring.

It is clear that the positive rigid coloring \(\sigma\) of a finite,
binary tree \(T\) is given by \(\sigma(w) = (-1)^{|w|}\).  If we
apply this to the infnite, binary tree \(\mathcal T\), then it is
clear that the positive rigid coloring of any finite, binary tree
\(T\) is just the restriction to \(T\) of the positive rigid
coloring of \(\mathcal T\).  Below we show the edge colors of the
positive rigid coloring, normalized to have root color equal to 1,
and drawn to the depth where addresses have length 6.  The colors
are shown at the lower ends of the edges for clarity in the lower
row and to make the program generating the figure easier to write.
The picture gives a pleasant superposition of a period three pattern
on a period two structure.
\[
\xy
(0.00, 0.00); (0.95, 4.76)**@{-}; (1.90, 0.00)**@{-}; 
(0.00, -2.00)*{\scriptstyle{1}}; 
(1.90, -2.00)*{\scriptstyle{3}}; 
(3.81, 0.00); (4.76, 4.76)**@{-}; (5.71, 0.00)**@{-}; 
(3.81, -2.00)*{\scriptstyle{2}}; 
(5.71, -2.00)*{\scriptstyle{1}}; 
(7.62, 0.00); (8.57, 4.76)**@{-}; (9.52, 0.00)**@{-}; 
(7.62, -2.00)*{\scriptstyle{3}}; 
(9.52, -2.00)*{\scriptstyle{2}}; 
(11.43, 0.00); (12.38, 4.76)**@{-}; (13.33, 0.00)**@{-}; 
(11.43, -2.00)*{\scriptstyle{1}}; 
(13.33, -2.00)*{\scriptstyle{3}}; 
(15.24, 0.00); (16.19, 4.76)**@{-}; (17.14, 0.00)**@{-}; 
(15.24, -2.00)*{\scriptstyle{2}}; 
(17.14, -2.00)*{\scriptstyle{1}}; 
(19.05, 0.00); (20.00, 4.76)**@{-}; (20.95, 0.00)**@{-}; 
(19.05, -2.00)*{\scriptstyle{3}}; 
(20.95, -2.00)*{\scriptstyle{2}}; 
(22.86, 0.00); (23.81, 4.76)**@{-}; (24.76, 0.00)**@{-}; 
(22.86, -2.00)*{\scriptstyle{1}}; 
(24.76, -2.00)*{\scriptstyle{3}}; 
(26.67, 0.00); (27.62, 4.76)**@{-}; (28.57, 0.00)**@{-}; 
(26.67, -2.00)*{\scriptstyle{2}}; 
(28.57, -2.00)*{\scriptstyle{1}}; 
(30.48, 0.00); (31.43, 4.76)**@{-}; (32.38, 0.00)**@{-}; 
(30.48, -2.00)*{\scriptstyle{3}}; 
(32.38, -2.00)*{\scriptstyle{2}}; 
(34.29, 0.00); (35.24, 4.76)**@{-}; (36.19, 0.00)**@{-}; 
(34.29, -2.00)*{\scriptstyle{1}}; 
(36.19, -2.00)*{\scriptstyle{3}}; 
(38.10, 0.00); (39.05, 4.76)**@{-}; (40.00, 0.00)**@{-}; 
(38.10, -2.00)*{\scriptstyle{2}}; 
(40.00, -2.00)*{\scriptstyle{1}}; 
(41.90, 0.00); (42.86, 4.76)**@{-}; (43.81, 0.00)**@{-}; 
(41.90, -2.00)*{\scriptstyle{3}}; 
(43.81, -2.00)*{\scriptstyle{2}}; 
(45.71, 0.00); (46.67, 4.76)**@{-}; (47.62, 0.00)**@{-}; 
(45.71, -2.00)*{\scriptstyle{1}}; 
(47.62, -2.00)*{\scriptstyle{3}}; 
(49.52, 0.00); (50.48, 4.76)**@{-}; (51.43, 0.00)**@{-}; 
(49.52, -2.00)*{\scriptstyle{2}}; 
(51.43, -2.00)*{\scriptstyle{1}}; 
(53.33, 0.00); (54.29, 4.76)**@{-}; (55.24, 0.00)**@{-}; 
(53.33, -2.00)*{\scriptstyle{3}}; 
(55.24, -2.00)*{\scriptstyle{2}}; 
(57.14, 0.00); (58.10, 4.76)**@{-}; (59.05, 0.00)**@{-}; 
(57.14, -2.00)*{\scriptstyle{1}}; 
(59.05, -2.00)*{\scriptstyle{3}}; 
(60.95, 0.00); (61.90, 4.76)**@{-}; (62.86, 0.00)**@{-}; 
(60.95, -2.00)*{\scriptstyle{2}}; 
(62.86, -2.00)*{\scriptstyle{1}}; 
(64.76, 0.00); (65.71, 4.76)**@{-}; (66.67, 0.00)**@{-}; 
(64.76, -2.00)*{\scriptstyle{3}}; 
(66.67, -2.00)*{\scriptstyle{2}}; 
(68.57, 0.00); (69.52, 4.76)**@{-}; (70.48, 0.00)**@{-}; 
(68.57, -2.00)*{\scriptstyle{1}}; 
(70.48, -2.00)*{\scriptstyle{3}}; 
(72.38, 0.00); (73.33, 4.76)**@{-}; (74.29, 0.00)**@{-}; 
(72.38, -2.00)*{\scriptstyle{2}}; 
(74.29, -2.00)*{\scriptstyle{1}}; 
(76.19, 0.00); (77.14, 4.76)**@{-}; (78.10, 0.00)**@{-}; 
(76.19, -2.00)*{\scriptstyle{3}}; 
(78.10, -2.00)*{\scriptstyle{2}}; 
(80.00, 0.00); (80.95, 4.76)**@{-}; (81.90, 0.00)**@{-}; 
(80.00, -2.00)*{\scriptstyle{1}}; 
(81.90, -2.00)*{\scriptstyle{3}}; 
(83.81, 0.00); (84.76, 4.76)**@{-}; (85.71, 0.00)**@{-}; 
(83.81, -2.00)*{\scriptstyle{2}}; 
(85.71, -2.00)*{\scriptstyle{1}}; 
(87.62, 0.00); (88.57, 4.76)**@{-}; (89.52, 0.00)**@{-}; 
(87.62, -2.00)*{\scriptstyle{3}}; 
(89.52, -2.00)*{\scriptstyle{2}}; 
(91.43, 0.00); (92.38, 4.76)**@{-}; (93.33, 0.00)**@{-}; 
(91.43, -2.00)*{\scriptstyle{1}}; 
(93.33, -2.00)*{\scriptstyle{3}}; 
(95.24, 0.00); (96.19, 4.76)**@{-}; (97.14, 0.00)**@{-}; 
(95.24, -2.00)*{\scriptstyle{2}}; 
(97.14, -2.00)*{\scriptstyle{1}}; 
(99.05, 0.00); (100.00, 4.76)**@{-}; (100.95, 0.00)**@{-}; 
(99.05, -2.00)*{\scriptstyle{3}}; 
(100.95, -2.00)*{\scriptstyle{2}}; 
(102.86, 0.00); (103.81, 4.76)**@{-}; (104.76, 0.00)**@{-}; 
(102.86, -2.00)*{\scriptstyle{1}}; 
(104.76, -2.00)*{\scriptstyle{3}}; 
(106.67, 0.00); (107.62, 4.76)**@{-}; (108.57, 0.00)**@{-}; 
(106.67, -2.00)*{\scriptstyle{2}}; 
(108.57, -2.00)*{\scriptstyle{1}}; 
(110.48, 0.00); (111.43, 4.76)**@{-}; (112.38, 0.00)**@{-}; 
(110.48, -2.00)*{\scriptstyle{3}}; 
(112.38, -2.00)*{\scriptstyle{2}}; 
(114.29, 0.00); (115.24, 4.76)**@{-}; (116.19, 0.00)**@{-}; 
(114.29, -2.00)*{\scriptstyle{1}}; 
(116.19, -2.00)*{\scriptstyle{3}}; 
(118.10, 0.00); (119.05, 4.76)**@{-}; (120.00, 0.00)**@{-}; 
(118.10, -2.00)*{\scriptstyle{2}}; 
(120.00, -2.00)*{\scriptstyle{1}}; 
(0.95, 4.76); (2.86, 10.48)**@{-}; (4.76, 4.76)**@{-}; 
(0.45, 6.76)*{\scriptstyle{2}}; 
(5.26, 6.76)*{\scriptstyle{3}}; 
(8.57, 4.76); (10.48, 10.48)**@{-}; (12.38, 4.76)**@{-}; 
(8.07, 6.76)*{\scriptstyle{1}}; 
(12.88, 6.76)*{\scriptstyle{2}}; 
(16.19, 4.76); (18.10, 10.48)**@{-}; (20.00, 4.76)**@{-}; 
(15.69, 6.76)*{\scriptstyle{3}}; 
(20.50, 6.76)*{\scriptstyle{1}}; 
(23.81, 4.76); (25.71, 10.48)**@{-}; (27.62, 4.76)**@{-}; 
(23.31, 6.76)*{\scriptstyle{2}}; 
(28.12, 6.76)*{\scriptstyle{3}}; 
(31.43, 4.76); (33.33, 10.48)**@{-}; (35.24, 4.76)**@{-}; 
(30.93, 6.76)*{\scriptstyle{1}}; 
(35.74, 6.76)*{\scriptstyle{2}}; 
(39.05, 4.76); (40.95, 10.48)**@{-}; (42.86, 4.76)**@{-}; 
(38.55, 6.76)*{\scriptstyle{3}}; 
(43.36, 6.76)*{\scriptstyle{1}}; 
(46.67, 4.76); (48.57, 10.48)**@{-}; (50.48, 4.76)**@{-}; 
(46.17, 6.76)*{\scriptstyle{2}}; 
(50.98, 6.76)*{\scriptstyle{3}}; 
(54.29, 4.76); (56.19, 10.48)**@{-}; (58.10, 4.76)**@{-}; 
(53.79, 6.76)*{\scriptstyle{1}}; 
(58.60, 6.76)*{\scriptstyle{2}}; 
(61.90, 4.76); (63.81, 10.48)**@{-}; (65.71, 4.76)**@{-}; 
(61.40, 6.76)*{\scriptstyle{3}}; 
(66.21, 6.76)*{\scriptstyle{1}}; 
(69.52, 4.76); (71.43, 10.48)**@{-}; (73.33, 4.76)**@{-}; 
(69.02, 6.76)*{\scriptstyle{2}}; 
(73.83, 6.76)*{\scriptstyle{3}}; 
(77.14, 4.76); (79.05, 10.48)**@{-}; (80.95, 4.76)**@{-}; 
(76.64, 6.76)*{\scriptstyle{1}}; 
(81.45, 6.76)*{\scriptstyle{2}}; 
(84.76, 4.76); (86.67, 10.48)**@{-}; (88.57, 4.76)**@{-}; 
(84.26, 6.76)*{\scriptstyle{3}}; 
(89.07, 6.76)*{\scriptstyle{1}}; 
(92.38, 4.76); (94.29, 10.48)**@{-}; (96.19, 4.76)**@{-}; 
(91.88, 6.76)*{\scriptstyle{2}}; 
(96.69, 6.76)*{\scriptstyle{3}}; 
(100.00, 4.76); (101.90, 10.48)**@{-}; (103.81, 4.76)**@{-}; 
(99.50, 6.76)*{\scriptstyle{1}}; 
(104.31, 6.76)*{\scriptstyle{2}}; 
(107.62, 4.76); (109.52, 10.48)**@{-}; (111.43, 4.76)**@{-}; 
(107.12, 6.76)*{\scriptstyle{3}}; 
(111.93, 6.76)*{\scriptstyle{1}}; 
(115.24, 4.76); (117.14, 10.48)**@{-}; (119.05, 4.76)**@{-}; 
(114.74, 6.76)*{\scriptstyle{2}}; 
(119.55, 6.76)*{\scriptstyle{3}}; 
(2.86, 10.48); (6.67, 17.33)**@{-}; (10.48, 10.48)**@{-}; 
(2.36, 12.48)*{\scriptstyle{1}}; 
(10.98, 12.48)*{\scriptstyle{3}}; 
(18.10, 10.48); (21.90, 17.33)**@{-}; (25.71, 10.48)**@{-}; 
(17.60, 12.48)*{\scriptstyle{2}}; 
(26.21, 12.48)*{\scriptstyle{1}}; 
(33.33, 10.48); (37.14, 17.33)**@{-}; (40.95, 10.48)**@{-}; 
(32.83, 12.48)*{\scriptstyle{3}}; 
(41.45, 12.48)*{\scriptstyle{2}}; 
(48.57, 10.48); (52.38, 17.33)**@{-}; (56.19, 10.48)**@{-}; 
(48.07, 12.48)*{\scriptstyle{1}}; 
(56.69, 12.48)*{\scriptstyle{3}}; 
(63.81, 10.48); (67.62, 17.33)**@{-}; (71.43, 10.48)**@{-}; 
(63.31, 12.48)*{\scriptstyle{2}}; 
(71.93, 12.48)*{\scriptstyle{1}}; 
(79.05, 10.48); (82.86, 17.33)**@{-}; (86.67, 10.48)**@{-}; 
(78.55, 12.48)*{\scriptstyle{3}}; 
(87.17, 12.48)*{\scriptstyle{2}}; 
(94.29, 10.48); (98.10, 17.33)**@{-}; (101.90, 10.48)**@{-}; 
(93.79, 12.48)*{\scriptstyle{1}}; 
(102.40, 12.48)*{\scriptstyle{3}}; 
(109.52, 10.48); (113.33, 17.33)**@{-}; (117.14, 10.48)**@{-}; 
(109.02, 12.48)*{\scriptstyle{2}}; 
(117.64, 12.48)*{\scriptstyle{1}}; 
(6.67, 17.33); (14.29, 25.56)**@{-}; (21.90, 17.33)**@{-}; 
(6.17, 19.33)*{\scriptstyle{2}}; 
(22.40, 19.33)*{\scriptstyle{3}}; 
(37.14, 17.33); (44.76, 25.56)**@{-}; (52.38, 17.33)**@{-}; 
(36.64, 19.33)*{\scriptstyle{1}}; 
(52.88, 19.33)*{\scriptstyle{2}}; 
(67.62, 17.33); (75.24, 25.56)**@{-}; (82.86, 17.33)**@{-}; 
(67.12, 19.33)*{\scriptstyle{3}}; 
(83.36, 19.33)*{\scriptstyle{1}}; 
(98.10, 17.33); (105.71, 25.56)**@{-}; (113.33, 17.33)**@{-}; 
(97.60, 19.33)*{\scriptstyle{2}}; 
(113.83, 19.33)*{\scriptstyle{3}}; 
(14.29, 25.56); (29.52, 35.44)**@{-}; (44.76, 25.56)**@{-}; 
(13.79, 27.56)*{\scriptstyle{1}}; 
(45.26, 27.56)*{\scriptstyle{3}}; 
(75.24, 25.56); (90.48, 35.44)**@{-}; (105.71, 25.56)**@{-}; 
(74.74, 27.56)*{\scriptstyle{2}}; 
(106.21, 27.56)*{\scriptstyle{1}}; 
(29.52, 35.44); (60.00, 47.29)**@{-}; (90.48, 35.44)**@{-}; 
(29.02, 37.44)*{\scriptstyle{2}}; 
(90.98, 37.44)*{\scriptstyle{3}}; 
(60.00, 49.29)*{\scriptstyle{1}}; 
\endxy
\]

\subsection{The group \protect\(F_4\protect\)}

Here we identify the subgroup of \(F\) that has a rigid coloring.
We start by pointing out the following.

\begin{prop}\mylabel{RigidIsSubProp} The set of elements of \(F\)
that have a rigid coloring forms a subgroup of \(F\).  \end{prop}

\begin{proof} If \((A,B)\) and \((C,D)\) have rigid colorings, then
each pair satisfies the parity condition and we may assume that both
are positive rigid.  We can form the product from \((A',B\cup
C)(B\cup C,D')\) which is obtained by attaching components of
\(B-C\) to \(C\) and \(D\) to obtain \((B\cup C,D')\) and attaching
components of \(C-B\) to \(A\) and \(B\) to obtain \((A',B\cup C)\).
Each component of \(B-C\) has its root edge at a certain level which
is in turn the level of a leaf edge of a leaf of \(C\) which in turn
has the same parity as the level of the corresponding leaf of \(D\).
Thus the positive rigid sign pattern of \(\mathcal T\) will be
extended in the construction of \(D'\) from \(D\).  Similar remarks
apply to the construction of \(A'\) from \(A\).  Now the details of
the Compatibility lemma (Lemma \ref{CompatLem}) give that the
resulting color vectors on \(A'\) and \(D'\) will be identical.
\end{proof}

We will use \(F_4\)\index{\protect\(F_4\protect\)} to denote the
subgroup of \(F\) consisting of those elements with rigid colorings.
The reason for the notation comes from the next proposition.  The
notation is reasonably common.

\begin{prop}\mylabel{RigidIsF4Prop} The group \(F_4\) consists of
those elements of \(F\) which when viewed as self homeomorphisms of
\([0,1]\) use only slopes that are integral powers of 4.  \end{prop}

\begin{proof}  This is an immediate consequence of Proposition
\ref{DepthCondProp}.  \end{proof}

\subsection{Characterizing positive rigid color vectors}

We can now give a characterization of the positive rigid color
vectors that we will use in Section \ref{RigidCountSec} when we
count rigid color vectors.

\begin{prop}\mylabel{RigidCharLem} The positive rigid color vectors
(modulo the action of \(S_3\) on the colors) are the non-constant
vectors \(\mathbf c\) that sum to 1 and for which no prefix sums to
3.  \end{prop}

\begin{proof}  With the convention that 1 is the root edge color, we
have that the sum of any vector that we discuss is 1.  The assumption
that the top internal vertex is positive takes care of the rest
of the action by \(S_3\).

For a pair \((D,R)\) with a common positive, rigid coloring, we
consider the proper, face 3-coloring of the map \(D\cup R\)
compatible with the edge coloring.  Each edge has the non-zero color
that is not the color of either face impinging on the edge.  It
follows that all the edges that impinge on vertices in the boundary
of a face \(F\) but are not edges in that boundary are given the
color of \(F\).

The face that impinges on the root edges of \(D\cup R\) to the left
of the trees (as drawn in the plane) must have the color of the
right descending edge from the top internal vertex.  In the positive
rigid tree, this color is 3.  The remaining faces use colors 1, 2,
and 3.  All faces touch the leaves of the trees and the leaf
colors give the differences between the colors of two consecutive
faces.  If any prefix sums to 3, then with the left face being
colored 3, there will be a face colored zero and conversely.
\end{proof}

\section{Color graphs, zero sets, shadow patterns, and long
paths}\mylabel{ZeroSetSec}

\subsection{Zero sets}

We return to color graphs but from a complementary view.  The set of
vertices in a colored associahedron \((A_d, \mathbf c)\) for which
\(\mathbf c\) is not valid will be called the {\itshape zero
set}\index{zero set!of color vector}\index{color!vector!zero set
of}\index{vector!color!zero set of} of \((A_d, \mathbf c)\).  The
name is chosen since a vertex (tree) in the zero set will have an
edge colored zero by \(\mathbf c\).  As with the color graph it
makes sense to refer to the zero set of \((A_d, \mathbf c)\) as the
zero set of \(\mathbf c\).  Theorem \ref{GravPayanThm} implies the
following.

\begin{prop}\mylabel{NonSepProp} For each \(d>1\), no zero set of
any \((A_d, \mathbf c)\) with \(\mathbf c\) acceptable and flexible
separates the 1-skeleton of \(A_d\).  \end{prop}

We now discuss zero sets of acceptable color vectors.

Let \(\mathbf c\) be an acceptable color vector with \(d+2\)
entries, and let \(T\) be a vertex in \(A_d\).  If \(\mathbf c\) is
not valid for \(T\), then let \(e\) be an edge that evaluates to
zero when \(\mathbf c\) is applied to the leaves of \(T\).  If \(v\)
is the lower endpoint of \(e\), then \(T_v\) has edge \(e\) as its
root edge and the leaves of \(T_v\) form an interval in the totally
ordered set consisting of the leaves of \(T\).  That is, for some
interval \([i,j]\) in \([1,d+2]\), the leaves of \(T_v\) are the
\(i\)-th through the \(j\)-th in the left-right order.  The color
zero assigned to \(e\) is a result of the fact that the sum of the
colors in \(\mathbf c\) from the \(i\)-th through the \(j\)-th is
zero.

This leads to the next definition.  If \(T\) is a finite, binary
tree with \(d+2\) leaves, then for each internal vertex \(v\) of
\(T\), the leaves of \(T_v\) form the {\itshape
shadow}\index{shadow!of vertex in tree}\index{tree!vertex!shadow
of}\index{vertex!of tree!shadow}\index{tree!shadow!of vertex} of
\(v\), and the interval in \([1,d+2]\) that contains the numbers in
left-right order of the leaves of \(T_v\) is the {\itshape shadow
interval}\index{shadow!interval!of vertex in
tree}\index{tree!vertex!shadow interval of}\index{vertex!of
tree!shadow interval}\index{tree!shadow interval!of vertex} of
\(v\).  Note that different internal vertices of \(T\) yield
different shadow intervals.

The set of shadow intervals of all internal vertices of \(T\) except
for the vertex \(\emptyset\) forms the {\itshape shadow
pattern}\index{pattern!shadow!of tree}\index{shadow!pattern of
tree}\index{tree!shadow!pattern} of \(T\).  The shadow interval of
\(\emptyset\) includes all of \([1,d+2]\) and gives no information.
It is an exercise that \(T\) is determined by its shadow pattern.
We will see examples shortly.

Applying this to zero sets, we see that if \(\mathbf c\) is not
valid for \(T\), then there is a shadow interval \([m,n]\) for \(T\)
for which the sum of the \(c_i\) with \(i\in [m,n]\) is zero.  Thus
if \(Z_\mathbf c\)\index{\protect\(Z_\mathbf c\protect\)} is the set
of those intervals \(J\) for which the sum of the \(c_i\) with
\(i\in J\) is zero, then a tree \(T\) is in the zero set of
\(\mathbf c\) if and only if a shadow interval for \(T\) is in
\(Z_\mathbf c\).

Given any interval \(J\) of length \(k\) in \([1,d+2]\), the set of
trees with \(J\) in their shadow patterns is a codimension-1 face of
\(A_d\) of the form \(A_{k-2}\times A_{d-k+1}\).  From this it
follows that the zero set of a given color vector is a union of
codimension-1 faces.  Not every union of codimension-1 faces is a
zero set.  In particular not every collection of intervals is some
\(Z_\mathbf c\) for some color vector \(\mathbf c\).  If \([a,b]\)
is in \(Z_\mathbf c\), then \([a,b+1]\) cannot be since then the
color of leaf \(b+1\) would have to be zero.  Further, if \([a,b]\)
and \([b+1,d]\) are in \(Z_\mathbf c\), then so must \([a,d]\) be in
\(Z_\mathbf c\).  Lastly, we show a union of codimension-1 faces
that separates, in contrast to Proposition \ref{NonSepProp}.

\subsection{A separating example}

We will give a union of four codimension-1 faces of \(A_4\) that
separates the 1-skeleton of \(A_4\).

Since a codimension-1 face can be specified by an interval of leaf
numbers, we will give our union of codimension-1 faces as a set of
intervals in \([1,6]\).  The intervals are 
\[
[1,5], \quad [2,4],\quad [3,6],\quad [4,6].
\]

These cannot correspond to a zero set of an acceptable color vector
because of the presence of  intervals \([3,6]\) and \([4,6]\).

It will help to have the following picture of these intervals.
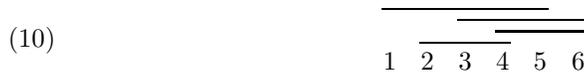
\mymargin{CodOneExmpl}\begin{equation}\label{CodOneExmpl}
\begin{split}
\xy
(0,0)*{1}; (5,0)*{2}; (10,0)*{3}; (15,0)*{4}; (20,0)*{5};
(25,0)*{6};
(4,2.5); (16,2.5)**@{-};
(14,4); (26,4)**@{-};
(9,5.5); (26,5.5)**@{-};
(-1,7); (21,7)**@{-};
\endxy
\end{split}
\end{equation}

We now show six trees that are not in the corresponding
codimension-1 faces.  That these are all the trees not in the faces
specified in \tref{CodOneExmpl} is not relevant.  The shadow
patterns of the trees are shown, and it is trivial to verify that
none of the six is in any of the four faces specified.
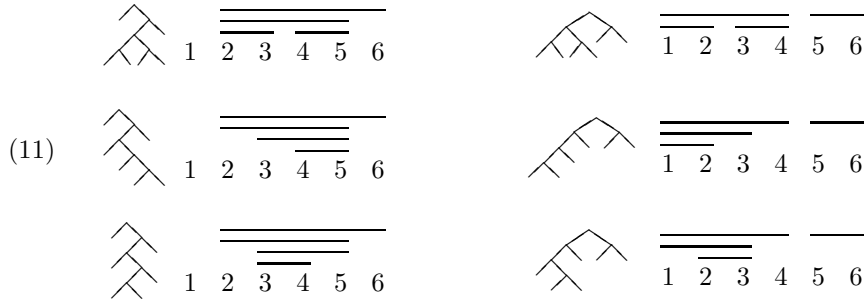
\mymargin{CodOneCompl}\begin{equation}\label{CodOneCompl}
\xy
(0,0)*{\xy
(0,0)*{
\xy
(0,-4); (2,-2)**@{-}; (3.5,-4)**@{-};
(2,-2); (4,0)**@{-}; (6,-2)**@{-}; 
(4.5,-4); (6,-2)**@{-}; (8,-4)**@{-}; 
(4,0); (6,2)**@{-}; (8,0)**@{-};
(6,2); (4,4)**@{-}; (2,2)**@{-};
\endxy};
(20,0)*{
\xy
(0,-2)*{1}; (5,-2)*{2}; (10,-2)*{3}; (15,-2)*{4}; (20,-2)*{5};
(25,-2)*{6};
(4,0.5); (11,0.5)**@{-};
(14,0.5); (21,0.5)**@{-};
(4,2); (21,2)**@{-};
(4,3.5); (26,3.5)**@{-};
\endxy};
(0,-15)*{
\xy
(2,-2); (4,0)**@{-}; (6,-2)**@{-}; 
(4,-4); (6,-2)**@{-}; (8,-4)**@{-}; 
(6,-6); (8,-4)**@{-}; (10,-6)**@{-};
(4,0); (6,2)**@{-}; (8,0)**@{-};
(6,2); (4,4)**@{-}; (2,2)**@{-};
\endxy};
(20,-15)*{
\xy
(0,-2)*{1}; (5,-2)*{2}; (10,-2)*{3}; (15,-2)*{4}; (20,-2)*{5};
(25,-2)*{6};
(9,2); (21,2)**@{-};
(14,0.5); (21,0.5)**@{-};
(4,3.5); (21,3.5)**@{-};
(4,5); (26,5)**@{-};
\endxy};
(0,-30)*{
\xy
(2,-2); (4,0)**@{-}; (6,-2)**@{-}; 
(4,-4); (6,-2)**@{-}; (8,-4)**@{-}; 
(2,-6); (4,-4)**@{-}; (6,-6)**@{-};
(4,0); (6,2)**@{-}; (8,0)**@{-};
(6,2); (4,4)**@{-}; (2,2)**@{-};
\endxy};
(20,-30)*{
\xy
(0,-2)*{1}; (5,-2)*{2}; (10,-2)*{3}; (15,-2)*{4}; (20,-2)*{5};
(25,-2)*{6};
(9,2); (21,2)**@{-};
(9,0.5); (16,0.5)**@{-};
(4,3.5); (21,3.5)**@{-};
(4,5); (26,5)**@{-};
\endxy};
\endxy};
(60,0)*{\xy
(0,0)*{
\xy
(0,-4); (2,-2)**@{-}; (3.5,-4)**@{-};
(2,-2); (4,0)**@{-}; (6,-2)**@{-}; 
(4.5,-4); (6,-2)**@{-}; (8,-4)**@{-}; 
(4,0); (7,2)**@{-}; (10,0)**@{-};
(8,-2); (10,0)**@{-}; (12,-2)**@{-};
\endxy};
(24,0)*{
\xy
(0,-2)*{1}; (5,-2)*{2}; (10,-2)*{3}; (15,-2)*{4}; (20,-2)*{5};
(25,-2)*{6};
(-1,0.5); (6,0.5)**@{-};
(9,0.5); (16,0.5)**@{-};
(19,2); (26,2)**@{-};
(-1,2); (16,2)**@{-};
\endxy};
(0,-15)*{
\xy
(-2,0); (0,2)**@{-}; (2,0)**@{-}; 
(-4,-2); (-2,0)**@{-}; (0,-2)**@{-}; 
(-6,-4); (-4,-2)**@{-}; (-2,-4)**@{-};
(4,0); (6,2)**@{-}; (8,0)**@{-};
(6,2); (3,4)**@{-}; (0,2)**@{-};
\endxy};
(24,-15)*{
\xy
(0,-2)*{1}; (5,-2)*{2}; (10,-2)*{3}; (15,-2)*{4}; (20,-2)*{5};
(25,-2)*{6};
(-1,0.5); (6,0.5)**@{-};
(-1,2); (11,2)**@{-};
(-1,3.5); (16,3.5)**@{-};
(19,3.5); (26,3.5)**@{-};
\endxy};
(0,-30)*{
\xy
(-2,0); (0,2)**@{-}; (2,0)**@{-}; 
(-4,-2); (-2,0)**@{-}; (0,-2)**@{-}; 
(-2,-4); (0,-2)**@{-}; (2,-4)**@{-};
(4,0); (6,2)**@{-}; (8,0)**@{-};
(6,2); (3,4)**@{-}; (0,2)**@{-};
\endxy};
(24,-30)*{
\xy
(0,-2)*{1}; (5,-2)*{2}; (10,-2)*{3}; (15,-2)*{4}; (20,-2)*{5};
(25,-2)*{6};
(4,0.5); (11,0.5)**@{-};
(-1,2); (11,2)**@{-};
(-1,3.5); (16,3.5)**@{-};
(19,3.5); (26,3.5)**@{-};
\endxy};
\endxy};
\endxy
\end{equation}
Recall that the shadow interval \([1,6]\) of the vertex
\(\emptyset\) is not part of any pattern and is not shown.

What must be done now is to verify that any rotation of any of the
six trees above either results in another one of the six or a tree
in the one of the faces specified by \tref{CodOneExmpl}.  To do
this, we look at the effect of rotation on the shadow intervals.
Rotations involving small parts of the tree have few possible
arrangements and we show three typical patterns below.
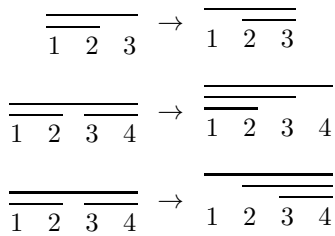
\[
\begin{split}
\xy
(0,-2)*{1}; (5,-2)*{2}; (10,-2)*{3};
(-1,0.5); (6,0.5)**@{-}; (-1,2); (11,2)**@{-};
(0,-6)*{\,};
\endxy
\,\,\,&\rightarrow\,\,\,
\xy
(0,-2)*{1}; (5,-2)*{2}; (10,-2)*{3};
(4,0.5); (11,0.5)**@{-}; (-1,2); (11,2)**@{-};
(0,-6)*{\,};
\endxy
\\
\xy
(0,-2)*{1}; (5,-2)*{2}; (10,-2)*{3}; (15,-2)*{4};
(-1,0.5); (6,0.5)**@{-}; (9,0.5); (16,0.5)**@{-}; 
(-1,2); (16,2)**@{-};
(0,-6)*{\,};
\endxy
\,\,\,&\rightarrow\,\,\,
\xy
(0,-2)*{1}; (5,-2)*{2}; (10,-2)*{3}; (15,-2)*{4};
(-1,0.5); (6,0.5)**@{-}; 
(-1,2); (11,2)**@{-}; 
(-1,3.5); (16,3.5)**@{-};
\endxy
\\
\xy
(0,-2)*{1}; (5,-2)*{2}; (10,-2)*{3}; (15,-2)*{4};
(-1,0.5); (6,0.5)**@{-}; (9,0.5); (16,0.5)**@{-}; 
(-1,2); (16,2)**@{-};
\endxy
\,\,\,&\rightarrow\,\,\,
\xy
(0,-2)*{1}; (5,-2)*{2}; (10,-2)*{3}; (15,-2)*{4};
(9,0.5); (16,0.5)**@{-}; 
(4,2); (16,2)**@{-}; 
(-1,3.5); (16,3.5)**@{-};
\endxy
\end{split}
\]
Longer intervals have to take into account the tree structure and
the effects can be worked out by the reader.

What is discovered by checking the four possible rotations of each
of the three trees in the left column of \tref{CodOneCompl} is that
a rotation either produces another tree in the left column or a tree
in a face specified by \tref{CodOneExmpl}.  Similarly a rotation of
a tree in the right column of \tref{CodOneCompl} either produces
another tree in the right column or a tree in a face specified by
\tref{CodOneExmpl}.  For example in any of the trees in the left
column of \tref{CodOneCompl}, a rotation using top vertex results in
the shadow interval \([1,5]\) which is in \tref{CodOneExmpl}.
Similarly in the right column one rotation using the top vertex in
each tree produces \([1,5]\) while the other rotation using the top
vertex produces \(([3,6]\) for the first tree and \([4,6]\) for the
second and third trees.  Other rotations are left to the reader.

It follows that the trees in the left column of \tref{CodOneCompl}
are in a different component of complement of the union of the faces
given by \tref{CodOneExmpl} from the trees in the right column of
\tref{CodOneCompl}, and thus the union of faces specified by
\tref{CodOneExmpl} separates the 1-skeleton of \(A_4\).

\subsection{Long paths}\mylabel{LongPathSec}

Consider the color vector (written without commas) \(\mathbf
c=1^m21^n\) with \(m\) and \(n\) at least zero.  This vector colors
vertices of \(A_d\) with \(d=m+n-1\).  We know that the diameter of
the 1-skeleton of \(A_d\) is no more than \(2d-4\) or \(2m+2n-5\).
Thus there is a path (that ignores signs) in the 1-skeleton of
length no more than \(2m+2n-5\) between any two vertices of \(A_d\).
However, we will show that there are vertices for which \(\mathbf
c\) is valid for which the shortest path between them that is sign
consistent (lies in the color graph of \(\mathbf c\)) has length
\(mn\).

We first note that any tree \(T\) for which \(\mathbf c\) is valid
must be a vine.  Any exposed caret in \(T\) must use the unique 2 in
\(\mathbf c\) as a color of a leaf edge, and this leaf edge can only
be in one caret.

All carets in \(T\) other than the exposed caret will have one
descending edge as an internal edge and the other descending edge
will be a leaf edge and, under \(\mathbf c\), this leaf edge will be
colored 1.  Thus the other (internal) descending edge of the caret
will be colored 2 or 3.  Note that even the exposed caret of \(T\)
has exactly one descending edge colored 1.  If a caret in \(T\) has
its left descending edge colored 1, we will label that caret \(l\)
and if its right descending edge is colored 1, we will label that
caret \(r\).  Since \(T\) is a vine, there is a well defined
top-to-bottom ordering of the carets, and we can read the labels
\(r\) or \(l\) in order from top to bottom.  This gives is a finite
word in the alphabet \(\{l,r\}\).  For example, in the trees below,
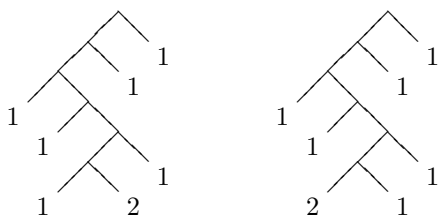
\[
\xy
(0,8); (4,12)**@{-}; (8,8)**@{-}; (10,6)*{\txst1};
(-4,4); (0,8)**@{-}; (4,4)**@{-}; (6,2)*{\txst1};
(-8,0); (-4,4)**@{-}; (0,0)**@{-}; (-10,-2)*{\txst1};
(-4,-4); (0,0)**@{-}; (4,-4)**@{-}; (-6,-6)*{\txst1};
(0,-8); (4,-4)**@{-}; (8,-8)**@{-};(10,-10)*{\txst1};
(-4,-12); (0,-8)**@{-}; (4,-12)**@{-}; (-6,-14)*{\txst1}; 
(6,-14)*{\txst2};
\endxy
\qquad\qquad
\xy
(0,8); (4,12)**@{-}; (8,8)**@{-}; (10,6)*{\txst1};
(-4,4); (0,8)**@{-}; (4,4)**@{-}; (6,2)*{\txst1};
(-8,0); (-4,4)**@{-}; (0,0)**@{-}; (-10,-2)*{\txst1};
(-4,-4); (0,0)**@{-}; (4,-4)**@{-}; (-6,-6)*{\txst1};
(0,-8); (4,-4)**@{-}; (8,-8)**@{-};(10,-10)*{\txst1};
(-4,-12); (0,-8)**@{-}; (4,-12)**@{-}; (-6,-14)*{\txst2}; 
(6,-14)*{\txst1};
\endxy
\]
the left tree gives the word \(rrllrl\) and the right tree gives the
word \(rrllrr\).  Note that the two trees are identical, but the
words and the color vectors are different.  The labeling convention
we use disagrees with that in Section 3 of \Zeilberger.

The following can easily be verified by the reader.

\begin{enumerate}

\item The colors of the internal edges alternate between 3 and 2
starting with 3 at the lowest internal edge.

\item Two adjacent carets have the same sign if and only if they
have opposite labels.

\item Any sign consistent rotation changes a single appearance of
\(rl\) to \(lr\) in the word for the tree (or the reverse) and
leaves the rest of the labels the same.

\item If a word \(w\) in \(\{l,r\}\) is given by a vine \(T\)
colored by \(\mathbf c=1^m21^n\), then the number of appearances of
\(l\) in \(w\) is \(m\) and the number of appearances of \(r\) in
\(w\) is \(n\).

\item The vines that can be colored by \(\mathbf c=1^m21^n\) are in
one-to-one correspondence with the words in \(\{l,r\}\) that use
\(m\) copies of \(l\) and \(n\) copies of \(r\).

\item If \(w_1\) and \(w_2\) are two words in \(\{l,r\}\) given by
vines \(V_1\) and \(V_2\), respectively, colored by \(\mathbf
c=1^n21^n\), then the shortest sign consistent path (path in the
color graph of \(\mathbf c\)) from \(V_1\) to \(V_2\) is the number
of \(lr\) to \(rl\) or reverse moves needed to take \(w_1\) to
\(w_2\).

\item If \(V_1\) gives the word \(w_1=l^mr^n\) and \(V_2\) gives the
word \(w_2=r^nl^m\) when colored by \(\mathbf c=1^m21^n\), then the
shortest sign consistent path (path in the color graph of \(\mathbf
c\)) from \(V_1\) to \(V_2\) as colored by \(\mathbf c\) is of
length \(mn\).  Further there is such a path of that length.  (The
vines \(V_1\) and \(V_2\) in are among those covered by Proposition
9 of \Zeilberger.)

\end{enumerate}

These combine to give the following.

\begin{prop}\mylabel{LongPathsLem} The diameter of the color graph
of \(\mathbf c=1^m21^n\) is \(mn\).  \end{prop}

Note that for \(m=n\), we get a coloring of \(A_{2n}\) and a color
graph of diameter \(n^2\) while the (uncolored) diameter of the
1-skeleton of \(A_{2n}\) is no more than \(4n-4\).

\section{Sign structures}\mylabel{SignStructSec}

In this section we build a signed graph \(\Sigma(w)\) for each edge
path \(w\) in an associahedron.  We will show that the path \(w\) is
``valid'' in that there is a finite, binary tree \(T\) and a sign
assignment for \(T\) so that \(w\) is valid for \(T\) in the sense
of Proposition \ref{RotChainProp} if and only if \(\Sigma(w)\) is
balanced.  This will verify the claim made after Proposition
\ref{RotChainProp}.

The vertex set of our graph \(\Sigma(w)\) will be the internal
vertices of the tree \(T\) above.  In \Carpentier, Carpentier
develops a similar but not identical criterion for the validity of a
path.  In particular, the vertex set of the two structures is
different with the vertices used in \Carpentier{} being the elements
of the path \(w\) and not the vertices of the tree \(T\).

\subsection{The signed graph}

Let \(w\) be an edge path in an associahedron that starts at a
vertex \(D\) and ends at a vertex \(R\).  We can think of \(w\) as a
word in the symmetric generators.  Let this word be \(w =
\rot{v_1}\rot{v_2}\cdots \rot{v_k}\).  Thinking of each
\(\rot{v_i}\) and \((D,R)\) as elements of \(F\), lets us write \[ w
= \rot{v_1}\rot{v_2}\cdots \rot{v_k} = (D,R).  \] However, there are
other paths that represent the same element of \(F\).  These paths
either have the same endpoints but get from one to the other by a
different route, or the paths read as the same string of symmetric
generators but connect different pairs of vertices, or both.  For
example, in the drawing \tref{TheAssoc} of \(A_3\) there are four
different paths that read as \(\rot u\rot{u1}\).  Our definition of
\(\Sigma(w)\) will depend on the word in the symmetric generators
and not on the particular start and end vertices.

The graph \(\Sigma(w)\) will be an undirected, signed graph with one
edge for each edge in the path \(w\).  The graph \(\Sigma(w)\) may
have parallel edges but will have no loops.  The vertex set will
start out as the set of vertices in \(\mathcal T\) so we will simply
use \(\mathcal T\) to denote the vertex set.  This is an infinite
set and is more than needed.  After proving a key lemma, we will cut
down the set of vertices.

We define \(\Sigma(w)\) inductively and if \(w=
\rot{v_1}\rot{v_2}\cdots \rot{v_k}\), then we will need \(p_i =
\rot{v_1}\rot{v_2}\cdots \rot{v_i}\) where \(1\le i\le k\) and we
will need \(p_0\) to be the identity in \(F\).  We start with
\(\Sigma_0 = \Sigma(p_0)\) as the graph on \(\mathcal T\) with no
edges.

If \(\Sigma_{i-1} = \Sigma(p_{i-1})\) is defined, then we form
\(\Sigma_i\) by adding an edge \(e_i\) to \(\Sigma_{i-1}\).  We need
to pick the endpoints of \(e_i\) and its sign.  Let \(x\) and \(y\)
be the pivot vertices of \(\rot {v_i}\).  The endpoints of \(e_i\)
will be \(s=(x)p_{i-1}^{-1}\) and \(t=(y)p_{i-1}^{-1}\).  Let
\(d_s\) and \(d_t\) be the degrees of \(s\) and \(t\), respectively,
in \(\Sigma_{i-1}\).  Then \(e_i\) will be given a positive sign if
and only if \(d_s+d_t\) is even.  The graph
\(\Sigma(w)\)\index{\protect\(\Sigma(w)\protect\)} that we seek is
\(\Sigma_k\).  Note that the starting vertex of \(w\) is not
relevant to the definition.

Rather than explain this definition, we will state and prove the
following.  A signed graph is {\itshape
balanced}\index{balanced!signed graph} if every closed walk
traverses an even number of negative edges.

\begin{thm}\mylabel{SignConsistencyThm} Let \(w\) be an edge path in
an associahedron that starts at a vertex \(D\).  Then 

\begin{enumerate} 

\item all of the endpoints of edges in \(\Sigma(w)\) are internal
vertices of \(D\), and 

\item  there is a sign assignment for \(D\) for which \(w\) is valid
if and only if \(\Sigma(w)\) is balanced.

\end{enumerate}

\end{thm}

\begin{proof} If the \(p_i\) are as defined above, then \(Dp_i\)
is the \(i\)-th vertex visited by \(w\) with \(D\) as the \(0\)-th.
Lemma \ref{FactsRightLem} makes \(p_i\) a bijection from the
internal vertices of \(D\) to the internal vertices of \(Dp_i\).
Item (1) follows immediately.

If there is a sign assignment for \(D\) for which \(w\) is valid,
then to show that \(\Sigma(w)\) is balanced it is sufficient to show
that each edge in \(\Sigma(w)\) is positive if and only if it
connects two vertices in \(D\) with the same sign.

The edge \(e_i\) in the construction above connects \(x\) and \(y\)
where \(xp_{i-1}\) and \(yp_{i-1}\) are the pivot vertices of
\(\rot{v_i}\).  But the sign of \(xp_{i-1}\) in the sign assignment
of \(Dp_{i-1}\) is the sign of \(x\) in \(D\) as modified by
\(p_{i-1} = \rot{v_1}\rot{v_2} \cdots \rot{v_{i-1}}\).  The
modification consists one negation for each \(xp_j\) with \(1\le
j<i\) for which \(xp_j\) is a pivot vertex of \(\rot{v_{j+1}}\).  But
this is exactly the degree of the vertex \(x\) in \(\Sigma_{i-1}\).
Similarly the sign of \(yp_{i-1}\) in \(Dp_{i-1}\) is the sign of
\(y\) negated a number of times which is the degree of the vertex
\(y\) in \(\Sigma_{i-1}\).  Thus the sign of \(e_i\) is positive if
and only if the signs \(xp_{i-1}\) and \(yp_{i-1}\) in \(Dp_{i-1}\)
both agree or both disagree, respectively, with the signs of \(x\)
and \(y\) in \(D\).  The validity of \(w\) says that the signs of
\(xp_{i-1}\) and \(yp_{i-1}\) must be equal.  Thus the sign of
\(e_i\) is positive if and only if the signs of \(x\) and \(y\) in
\(D\) are equal.  This proves the desired condition needed to show
that \(\Sigma(w)\) is balanced.

If it is known that \(\Sigma(w)\) is balanced, then there is a sign
assignment of the internal vertices of \(D\) with the property that
each edge in \(\Sigma(w)\) is positive if and only if it connects
two vertices in \(D\) with the same sign.  Considerations almost
identical to the previous paragraph show that \(w\) is valid for
this sign assignment on \(D\).  \end{proof}

The details in the above proof give us the following.  If
\(\Sigma(w)\) is balanced (continuing the notation of this section),
then a sign assignment of the internal vertices of \(D\) is
{\itshape compatible}\index{sign assignment!compatible with sign
structure}\index{sign structure!compatible with sign assignment}
with \(\Sigma(w)\) if an edge in \(\Sigma(w)\) is positive if and
only if it connects edges with the same sign.  We say that a
coloring\index{coloring!compatible with sign structure}\index{sign
structure!compatible with coloring} of the pair \((D,Dw)\) is
compatible with \(\Sigma(w)\) if the associated sign assignment on
the internal vertices of \(D\) is compatible with \(\Sigma(w)\).

\begin{thm}\mylabel{SignConsThmII} Let \(w\) be an edge path in an
associahedron that starts at a vertex \(D\).  Then the following
hold.

\begin{enumerate}

\item The path \(w\) is valid with a sign assignment for \(D\) if
and only if the sign assignment is compatible with \(\Sigma(w)\).

\item The number of sign assignments for which \(w\) is valid is
\(2^p\) where \(p\) is the number of components of \(\Sigma(w)\)
regarded as a graph on the internal vertices of \(D\).

\item The number of valid colorings of the pair \((D,Dw)\) that are
compatible with \(\Sigma(w)\) modulo permutations of the colors is
\(2^{p-1}\) with \(p\) as in (2).

\end{enumerate}

\end{thm}

\begin{proof} We have (1) from the details of the proof above, (2)
is immediate, and (3) from the fact that there is one representative
in each permutation class that has root color 1 and the sign of the
child of the root of \(D\) positive.  \end{proof}

Note that a single vertex can be a component of \(\Sigma(w)\).

For an edge path \(w\) in an associahedron, we call \(\Sigma(w)\)
the {\itshape sign structure}\index{sign structure!of
path}\index{path!sign structure of} of \(w\).

We say that an edge path \(w\) in an associahedron is {\itshape sign
consistent}\index{sign consistent path}\index{path!sign consistent}
if \(\Sigma(w)\) is balanced.

We have one observation that can be made immediately about sign
consistent paths.

\begin{lemma}\mylabel{SubConsistLem} A subpath of a sign consistent
path is sign consistent.  \end{lemma}

\begin{proof} If a sign consistent path \(w=pvs\) starts at tree
\(D\) with a given sign assignment \(\sigma\), then the subpath
\(v\) starts at \((D^\sigma)p\) and must be consistent with the sign
assignment there.  \end{proof}

\subsection{The second signed path
conjecture}\mylabel{SecondSPathConjSec}

We can apply Theorem \ref{SignConsistencyThm} to give a statement
equivalent to Conjecture \ref{SPathConjOne}.

\begin{conj}\mylabel{SPathConjTwo}\index{conjecture!signed
path!second} For every pair of finite, binary trees \((D, R)\) with
the same number of leaves, there is a sign consistent path from
\(D\) to \(R\).
\end{conj}

\subsection{The vertex set of a sign structure}

Giving the set of vertices of \(\mathcal T\) to \(\Sigma(w)\) is
clearly excessive, and also unrevealing.  Item (1) of Theorem
\ref{SignConsistencyThm} hints that if \(w\) is a path starting at a
tree \(D\), then the internal vertices of \(D\) might make a good
vertex set for \(\Sigma(w)\).  However, we have pointed out that
different paths can read as the same string of rotation symbols, and
the edges of \(\Sigma(w)\) really only depend on the string of
symbols.  Another obvious choice would be the set of endpoints of
the edges of \(\Sigma(w)\).  For reasons that we hope will be made
clear, we will adopt a compromise.

Let \(w\) be a path in some \(A_d\), and let \(S\) be the set of
endpoints of \(\Sigma(w)\).  There is a smallest tree \(T\)
containing all of \(S\) among the internal vertices of \(T\).  We
let the internal vertices of \(T\) be the set of vertices of
\(\Sigma(w)\).\index{vertex set!of sign structure}\index{sign
structure!vertex set}

The reason for our choice of vertex set is the next result.  

\begin{thm}\mylabel{PrimeConnThm} Let \(w\) be an edge path in an
associahedron, and let \(T\) be the tree whose internal vertices are
the vertices of \(\Sigma(w)\).  If the pair \((T,Tw)\) is prime,
then \(\Sigma(w)\) is connected.  \end{thm}

\begin{proof} The proof will be inductive, and we will have to prove
more than the statement of the theorem.

Let \(w\) have length \(n\), and let \(p_i\) be the prefix of length
\(i\) of \(w\), \(0\le i<n\).  Obviously, \(p_0\) is the empty word.
Recall that we can think of \(w\) and all the \(p_i\) as functions,
where the special case of \(i=0\) has \(p_0\) as the identity.

We need a preliminary discussion before giving the statements that
we will prove.  If \(S\) is a tree and \(S'\) and \(S''\) are
subtrees of a tree \(S\) having disjoint non-empty sets of internal
vertices and there is an edge from an internal vertex of \(S'\) to
an internal vertex of \(S''\), then either a leaf of \(S'\) is an
internal vertex of \(S''\) or the reverse.  The cases are mutually
exclusive.  In both cases, we say that \(S'\) and \(S''\) are
adjacent\index{adjacent!subtrees}\index{subtree!adjacent to a
subtree}.  In the case that a leaf of \(S'\) is an internal vertex
of \(S''\), we say that \(S'\) is above\index{above!for
subtrees}\index{subtree!above another subtree} \(S''\).  Otherwise
we say that \(S''\) is above \(S'\).

Consider the following statements where \(1\le i \le n\).

\(\alpha_i\): If \(C\) is a component of \(\Sigma(p_i)\) and
\(V(C)\) is its vertex set, then for each \(j\) with \(0\le j\le
i\), there is a subtree \(T_{C(j)}\) of \(Tp_j\) whose internal
vertices are exactly the vertices \((V(C))p_j\).

\(\beta_i\): Let \(C\) and \(D\) be components of \(\Sigma(p_i)\).
If \(T_{C(0)}\) and \(T_{D(0)}\) are adjacent with \(T_{C(0)}\) above
\(T_{D(0)}\), then for each \(j\) with \(0\le j\le i\), the subtrees
\(T_{C(j)}\) and \(T_{D(j)}\) are adjacent with \(T_{C(j)}\) above
\(T_{D(j)}\).  If \(T_{C(0)}\) and \(T_{D(0)}\) are not adjacent, then for
each \(j\) with \(0\le j\le i\), the subtrees \(T_{C(j)}\) and
\(T_{D(j)}\) are not adjacent.

We have that \(\alpha_1\) is trivial, and we have that \(\beta_1\)
follows from the nature of a single rotation.

We assume \(\alpha_k\) and \(\beta_k\) hold for some \(k<n\).  Let
\(\rot u=w_{k+1}\).  There are two cases to consider.

Case I: Both pivot vertices of \(\rot u\) are internal vertices of a
single \(T_{C(k)}\).  In this case, the edge added to
\(\Sigma(p_k)\) to create \(\Sigma(p_{k+1})\) has both of endpoints
in the vertices of a single component \(C\) of \(\Sigma(p_k)\).
Thus the components of \(\Sigma(p_k)\) and \(\Sigma(p_{k+1})\) have
identical vertex sets and \(\alpha_{k+1}\) holds for \(1\le j\le
k\).  To get \(\alpha_{k+1}\) for \(j=k+1\), we note that the vertex
set \((V(C))p_{k+1} = ((V(C))p_k)\rot u\) is the image under \(\rot
u\) of the set of internal vertices of the subtree \(T_{C(k)}\).
From the illustration in \tref{ARot}, we see that this forms the set
of internal vertices of a subtree of \(Tp_k\).  All other subtrees
corresponding to components of the sign structure are carried to
\(Tp_{k+1}\) isomorphically.  This proves \(\alpha_{k+1}\).

The statement \(\beta_{k+1}\) clearly holds for \(1\le j\le k\).
The \(\beta\) family of statements says that the relation ``adjacent
to'' and the relation ``over'' among the relevant subtrees are
preserved by the prefixes of \(w\).  The illustration in \tref{ARot}
also shows that these relations are preserved by a single rotation.
Thus \(\beta_{k+1}\) also holds for \(j=k+1\).

Case II: The pivot vertices of \(\rot u\) are internal vertices of
two different \(T_{C(k)}\) and \(T_{D(k)}\).  Since the pivot
vertices are endpoints of an edge, we have that \(T_{C(k)}\) and
\(T_{D(k)}\) are adjacent, and we can assume that \(T_{C(k)}\) is
over \(T_{D(k)}\).  Since \(\beta_k\) says that the relation
``adjacent to'' and the relation ``over'' are preserved by the
\(p_j\) with \(1\le j\le k\), we know that \(T_C\) and \(T_D\) are
adjacent with \(T_C\) over \(T_D\).  Similar statements apply to the
\(T_{C(j)}\) and \(T_{D(j)}\) for \(1\le j< k\), while for \(j=k\)
it is our hypothesis.  It is now seen that for \(0\le j\le k\), the
union of \(T_{C(j)}\) with \(T_{D(j)}\) is a subtree \(T_{A(j)}\) of
\(Tp_j\) whose set of internal vertices is the vertex set of a
component of \((\Sigma(p_{k+1}))p_j\).  Thus if we replace
\(T_{C(0)}\) and \(T_{D(0)}\) by \(T_{A(0)}\) in \(T\) and keep all
the other subtrees the same, then we now have the desired one-to-one
correspondence between subtree and component of \(\Sigma(p_{k+1})\).
This is carried in the correct way to each \(Tp_j\) by prefixes
\(p_j\) for \(1\le j\le k\), and the argument that this works for
\(j=k+1\) follows because we are now in the situation of Case I and
the argument for \(j=k+1\) for both \(\alpha\) and \(\beta\) applies
here.

We now assume that \(\Sigma(w)=\Sigma(p_n)\) is not connected.
Let \(C\) be a component of \(\Sigma(w)\) so that \(T_C\) is minimal
with respect to the ``over'' relation.  This will make every leaf of
\(T_C\) a leaf of \(T\).  Thus every leaf of \(T_{C(n)}\) is a leaf
of \(Tw\).  Since rotations, and thus chains of rotations, preserve
prefix order and thus the left-right order of the leaves, the leaves
of the subtrees \(T_C\) and \(T_{C(n)}\) define exactly the same
intervals in the leaf numberings of \(T\) and \(Tw\).  Since
\(\Sigma(w)\) is not connected, this interval is not the entire
interval of leaf numbers.  Thus \((T,Tw)\) is not prime.
\end{proof}

The following corollary combines primality and consistency with
Theorem \ref{PrimeConnThm}.

\begin{cor}\mylabel{OneColorPerSignCor} If \((T,Tw)\) is a prime
pair with \(w\) a consistent edge path starting at \(T\), then there
is a unique normal coloring for \(T\) compatible with \(\Sigma(w)\)
with the sign of \(\emptyset\) positive.  \end{cor}

\subsection{Relations among the paths}\mylabel{RelPathsSec}

This section contains a set of related observations about the
effects on sign structures of changing the paths.  We give them to
explain the wording of some questions that we raise about colorings
that arise from sign structure considerations.

\subsubsection{Moving paths across faces}

If \(w\) and \(w'\) are two edge paths in some \(A_d\) with the same
endpoints, then as words in the symmetric generators, they represent
the same element of \(F\).  It follows from the presentation
discussed in Section \ref{AltRepSec} that as words, we can alter
\(w\) using a sequence of ``square'' or ``pentagonal'' relations so
that the end result is the word \(w'\).  We will refer to the
corresponding alterations on paths as ``square'' or ``pentagonal''
moves.\index{square!move}\index{pentagonal!move}
\index{move!square}\index{move!pentagonal}  From the nature
of the square and pentagonal moves, it is clear that these moves can
be realized as moves across 2-dimensional faces in \(A_d\).  We look
at the two kinds of alterations.

\subsubsection{Moving paths across pentagons}

Let \(D\) and \(R\) be two non-adjacent vertices in \(A_2\).  See
\tref{ThePentagon}.  There are two simple paths between them---one
of length three and one of length two.  The path of length three is
always inconsistent and the path of length two is always consistent.
This is easy enough to check by hand.  For example, \(\rot{u0} \rot
u \rot {u1}\) is inconsistent, but \(\rot {u0} \rot u\) is
consistent.  Because of the action of the dihedral group of order
ten on \(A_2\), it is only necessary to check one of each length.
Thus a pentagonal move is capable of changing a path from
consistent to inconsistent or vice versa.

\subsubsection{Adding or removing canceling pairs}

It is also clear that a consistent path can be ``ruined'' by
inserting consecutive canceling pairs.  The consistent path
\(\rot{u0}\rot u\) represents the same element of \(F\) as the
inconsistent path \(\rot {u0}\rot u \rot {u1} \rot{\overline{u1}}\).

\subsubsection{Moving paths across
squares}\mylabel{SquareMoveSubSec}

There are two cases to consider.

In the simpler case, a two-edge subpath along part of a square is
replaced by the other two-edge subpath along the other part of the
square.  If the various square relations are investigated, it is
found that the first two-edge subpath introduces two edges in
\(\Sigma\) of the path that have disjoint endpoints and that the
other two-edge subpath introduces the same two edges in \(\Sigma\)
in the other order.  This is in spite of the fact that the four
edges around the square correspond to three different symmetric
generators.  When two edges with disjoint end points are introduced
in succession into \(\Sigma\), reversing the order does not change
any of the sign assignments in \(\Sigma\).  Thus a two-edge to
two-edge move across a square preserves sign consistency.

In the more complicated case, a three-edge subpath is replaced by a
one-edge subpath.  The three edge path introduces three edges into
\(\Sigma\) in which two are parallel and one has endpoints disjoint
from the parallel edges.  The two parallel edges acquire the same
signs as each other.  The shorter path leaves out the two parallel
edges.  Removing a pair of parallel edges does not alter the
parities of any vertex.  Thus it is seen that a three-edge to
one-edge move across a square does not destroy sign consistency, but
it might convert a sign inconsistent path into a sign consistent
path by removing a damaging term.  For example \(w_1=\rot {u0}\rot
u\rot{u1} \rot{u111} \rot{\overline{u1}}\) is sign inconsistent.  But
this represents the same element of \(F\) and of \(E\) as \(w_2=\rot
{u0}\rot u\rot{u11}\) which is sign consistent.

Since the passage from \(w_2\) to \(w_1\) takes a shorter consistent
path to a longer inconsistent path, we see again that increasing the
length of a path can cause problems with consistency.

\subsubsection{Shortest paths are not always the best}

We give an example here to show that making paths shorter is not
always the best.  The path
\[
\rot{u} \rot{u1}^3 \rot{u}^{-1}
\]
is easily shown to be sign inconsistent.  However this represents
the same element of \(F\) as 
\[
\left(
\rot{u0}^{-1}\rot{u}
\right)^3
\]
which is longer and sign consistent.  Computer search verifies that
no word with less than six symbols is sign consistent and equivalent
to the words above.  For the interest of the reader we mention that
there are 5 different paths of length 6 equivalent to the words
above with 5 different sign structures among them.

\subsubsection{Parallel edges in the sign structure}

The discussion in Section \ref{SquareMoveSubSec} shows that parallel
edges can come from canceling pairs and conjugations.  However,
they do not have to arise that way.

The word \(\rot u\rot u\rot{\overline{u1}}\) has a sign structure
with two parallel edges that have opposite signs and is thus clearly
not sign consistent.  It is also one of the forbidden three-edge
paths in \(A_2\).  See \tref{ThePentagon}.

The word \(\rot u\rot u \rot {u1} \rot{\overline{u11}}\), has
parallel edges with the same sign and is sign consistent.  It cannot
be made shorter.  With \(u=\emptyset\), the reduced tree pair for
this element of \(F\) is 
\mymargin{ParalleExmpl}\begin{equation}\label{ParalleExmpl}
\left(
{\xy
(4,2);(6,4)**@{-};(8,2)**@{-};
(2,0);(4,2)**@{-};(6,0)**@{-};
(4,-2);(2,0)**@{-};(0,-2)**@{-};
(2,-4);(4,-2)**@{-};(6,-4)**@{-};
\endxy}
\quad , \quad
{\xy
(0,2);(2,4)**@{-};(4,2)**@{-};
(2,0);(4,2)**@{-};(6,0)**@{-};
(4,-2);(6,0)**@{-};(8,-2)**@{-};
(2,-4);(4,-2)**@{-};(6,-4)**@{-};
\endxy}
\right).
\end{equation}
These can be located in the drawing \tref{TheAssoc} of \(A_3\) where
it is seen that the shortest edge path between them is of length 4.
It is an exercise (which we leave to the reader) to show that if two
vertices lie in a face of an associahedron, then the shortest edge
path between them lies in that face.  The exercise consists of
looking at the set of subtrees that define that face and noting that
all rotations outside that set of subtrees will add excessively to
the length of the path.  It follows, that no shorter word in the
symmetric generators represents the element of \(F\) pictured in
\tref{ParalleExmpl}.

It should be noted that the element in \tref{ParalleExmpl} is one of
the smallest elements that has a rigid coloring.  It is easy to
check that positive normal coloring consistent with the word \(\rot
u\rot u \rot {u1} \rot{\overline{u11}}\) comes from the color vector
\((2,2,3,1,3)\).  It is just as easy to check that the positive
rigid coloring for the pair comes from the color vector
\((2,2,1,3,3)\).

\subsubsection{Paths with
identical sign structures}\mylabel{DiffESameSignsSec}

Consider the following sequence of ten rotations.

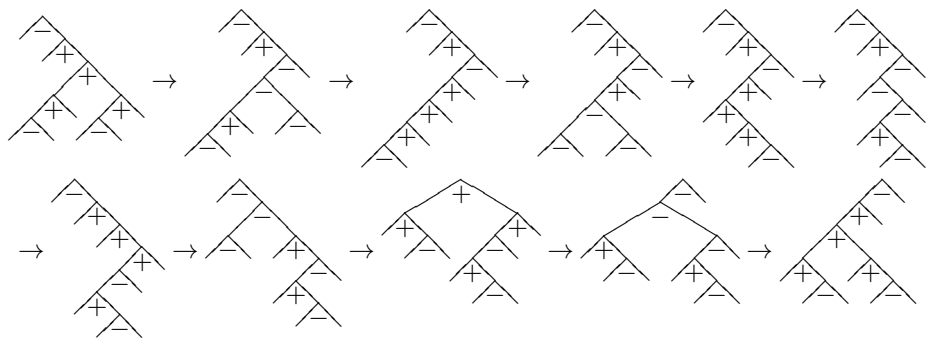
\newcommand{\NP}{\xy(-3,-3); (0,0)**@{-}; (3,-3)**@{-};
(0,-2)*{+};\endxy} 

\newcommand{\NM}{\xy(-3,-3); (0,0)**@{-}; (3,-3)**@{-};
(0,-2)*{-};\endxy} 

\newcommand{\WP}{\xy(-4.5,-4.5); (0,0)**@{-}; (4.5,-4.5)**@{-};
(0,-2)*{+};\endxy} 

\newcommand{\WM}{\xy(-4.5,-4.5); (0,0)**@{-}; (4.5,-4.5)**@{-};
(0,-2)*{-};\endxy} 

\newcommand{\SWP}{\xy(-7.5,-4.5); (0,0)**@{-}; (7.5,-4.5)**@{-};
(0,-2)*{+};\endxy} 

\newcommand{\SWM}{\xy(-7.5,-4.5); (0,0)**@{-}; (7.5,-4.5)**@{-};
(0,-2)*{-};\endxy} 

\begin{gather*}
\xy 
(0,8)*{\NM}; (3,5)*{\NP}; (6,1.3)*{\WP}; (1.5,-2.5)*{\NP}; (-1.5,-5.5)*{\NM};
(10.5,-2.5)*{\NP}; (7.5,-5.5)*{\NM};
\endxy
\rightarrow
\xy 
(0,8)*{\NM}; (3,5)*{\NP}; (6,2)*{\NM}; (3,-1.7)*{\WM}; (-1.5,-5.5)*{\NP};
(-4.5,-8.5)*{\NM}; (7.5,-5.5)*{\NM};
\endxy
\rightarrow
\xy 
(0,8)*{\NM}; (3,5)*{\NP}; (6,2)*{\NM}; (3,-1)*{\NP}; (0,-4)*{\NP};
(-3,-7)*{\NP}; (-6,-10)*{\NM};
\endxy
\rightarrow
\xy 
(0,8)*{\NM}; (3,5)*{\NP}; (6,2)*{\NM}; (3,-1)*{\NP}; (0,-4.7)*{\WM};
(-4.5,-8.5)*{\NM}; (4.5,-8.5)*{\NM};
\endxy
\rightarrow
\xy 
(0,8)*{\NM}; (3,5)*{\NP}; (6,2)*{\NM}; (3,-1)*{\NP}; (0,-4)*{\NP};
(3,-7)*{\NP}; (6,-10)*{\NM};
\endxy 
\rightarrow
\xy 
(0,8)*{\NM}; (3,5)*{\NP}; (6,2)*{\NM}; (3,-1)*{\NM}; (6,-4)*{\NM};
(3,-7)*{\NP}; (6,-10)*{\NM};
\endxy \\
\rightarrow
\xy 
(0,8)*{\NM}; (3,5)*{\NP}; (6,2)*{\NP}; (9,-1)*{\NP}; (6,-4)*{\NM};
(3,-7)*{\NP}; (6,-10)*{\NM};
\endxy
\rightarrow
\xy 
(0,8)*{\NM}; (3,4.3)*{\WM}; (-1.5, 0.5)*{\NM}; (7.5,0.5)*{\NP};
(10.5,-2.5)*{\NM}; (7.5,-5.5)*{\NP}; 
(10.5,-8.5)*{\NM};
\endxy
\rightarrow
\xy 
(0,7.3)*{\SWP}; (-7.5,3.5)*{\NP}; (-4.5, 0.5)*{\NM}; (7.5,3.5)*{\NP};
(4.5,0.5)*{\NM}; (1.5,-2.5)*{\NP}; 
(4.5,-5.5)*{\NM};
\endxy
\rightarrow
\xy 
(3,8)*{\NM}; (0,4.3)*{\SWM}; (-7.5,0.5)*{\NP}; (-4.5, -2.5)*{\NM};
(7.5,0.5)*{\NM}; 
(4.5,-2.5)*{\NP}; (7.5,-5.5)*{\NM}; 
\endxy
\rightarrow
\xy 
(0,8)*{\NM}; (-3,5)*{\NP}; (-6,1.3)*{\WP}; (-1.5,-2.5)*{\NP}; (1.5,-5.5)*{\NM};
(-10.5,-2.5)*{\NP}; (-7.5,-5.5)*{\NM};
\endxy
\end{gather*}

Let us refer to the first (upper left) tree in the sequence by \(D\)
and the last (lower right) tree by \(R\).  The reader can check that
the color vector \(1132133\) is consistent with the signs shown.
The reader can also check that the pair \((D,R)\) is prime.

Except for \(D\) and \(R\), all of the trees show a triple of
adjacent vertices with equal signs.  Thus these trees have exactly
two locations at which rotations can occur, one given by the arrow
leaving the tree and one given by the arrow arriving at the tree.
Thus (in the forward direction), the squence is completely
determined by the first arrow, and (in the reverse direction) by the
last arrow.

The trees \(D\) and \(R\) are unlike the other nine trees in that
there is a cluster of four adjacent vertices with equal signs.  Thus
there are three locations in each of \(D\) and \(R\) at which
rotations can occur and the figure above shows exactly one of them.
The reader can check that if either of the other two locations is
used to give a first rotation from \(D\) (in the forward direction)
or from \(R\) (in the reverse direction), then the same phenomenon
is observed.  One obtains a path of ten rotations from \(D\) to
\(R\) with no choices in how to go from the second rotation on.  One
can also check that the three paths have no vertices in common
except for \(D\) and \(R\).

It follows that the color graph of the vector \(1132133\) consists
exactly of three paths of ten edges each from \(D\) to \(R\) that
are disjoint except at their endpoints.  In particular, the color
graph contains no squares.

It is reasonable to guess (from the absence of squares in the color
graph) that the three paths give three different elements of \(E\)
that correspond to the same element \((D,R)\) of \(F\).  Further the
three paths give the same sign structure since \((D,R)\) is prime
and the colorings (all the same) determine and are determined by the
sign structure.  While the guess that the paths give different
elements of \(E\) is reasonable, we do not have a proof of this.

The above example lives in the 1-skeleton of \(A_6\).  The smallest
examples live in \(A_5\), but they are not as clean, having more
squares to complicate the graphs.

\subsection{On the relevance of the group \protect\(E\protect\)}

We have partial, set valued function from paths to colorings.  The
function is partial since there are paths (unsuccessful) that lead
to no colorings, and set valued since some paths (with non-connected
sign structures) that lead to many colorings (Theorem
\ref{SignConsThmII}).  It would be interesting to know if this
function is well defined on \(E\).  

A given path represents an element of \(E\) as well as an element of
\(F\).  Two paths with the same endpoints represent the same element
of \(F\) and they are equivalent modulo ``moves across squares and
pentagons.''  The two paths represent the same element of \(E\) if
they are {\itshape square equivalent}\index{square!equivalence of
paths}\index{path!square equivalence} in that one can be carried to
the other (keeping the endpoints fixed) by only moving across
squares.  Since moves across squares seem to be less drastic than
moves across pentagons, there is hope that some sort of well
definedness is obtainable.

The observations in Section \ref{SquareMoveSubSec} show that square
equivalent paths do not have to all be successful if one if them is.
So the well definedness question has to be worded to ask if two
successful paths are square equivalent, then are their sign
structures the same.

If there is a well defined function from \(E\) to sets of colorings,
then the observations in Section \ref{DiffESameSignsSec} hints
strongly that the function will not be one-to-one.

\subsection{On a converse to Theorem \ref{PrimeConnThm}}

Consider the following sequence of nine rotations.

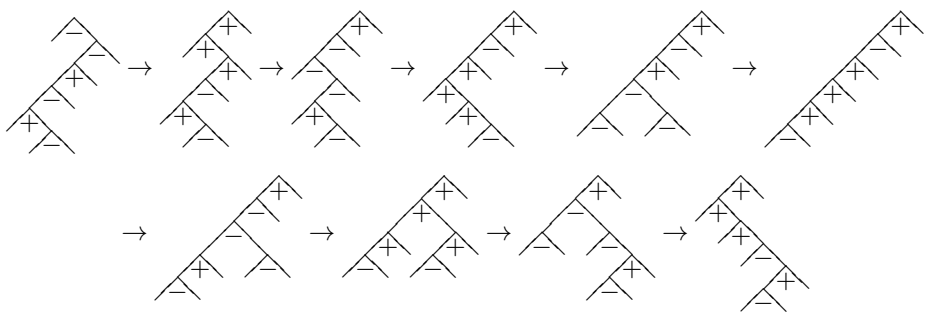
\begin{gather*}
\xy
(0,6)*{\NM}; (3,3)*{\NM};
(0,0)*{\NP}; (-3,-3)*{\NM};
(-6,-6)*{\NP}; (-3,-9)*{\NM};
\endxy
\rightarrow
\xy
(0,6)*{\NP}; (-3,3)*{\NP};
(0,0)*{\NP}; (-3,-3)*{\NM};
(-6,-6)*{\NP}; (-3,-9)*{\NM};
\endxy
\rightarrow
\xy
(0,6)*{\NP}; (-3,3)*{\NM};
(-6,0)*{\NM}; (-3,-3)*{\NM};
(-6,-6)*{\NP}; (-3,-9)*{\NM};
\endxy
\rightarrow
\xy
(0,6)*{\NP}; (-3,3)*{\NM};
(-6,0)*{\NP}; (-9,-3)*{\NP};
(-6,-6)*{\NP}; (-3,-9)*{\NM};
\endxy
\rightarrow
\xy
(0,6)*{\NP}; (-3,3)*{\NM};
(-6,0)*{\NP}; (-9,-3.7)*{\WM};
(-13.5,-7.5)*{\NM}; (-4.5,-7.5)*{\NM};
\endxy
\rightarrow
\xy
(0,6)*{\NP}; (-3,3)*{\NM};
(-6,0)*{\NP}; (-9,-3)*{\NP};
(-12,-6)*{\NP}; (-15,-9)*{\NM};
\endxy
\\
\rightarrow
\xy
(3,9)*{};
(0,6)*{\NP}; (-3,3)*{\NM};
(-6,-0.7)*{\WM}; (-10.5,-4.5)*{\NP};
(-13.5,-7.5)*{\NM}; (-1.5,-4.5)*{\NM};
\endxy
\rightarrow
\xy
(3,9)*{};
(0,6)*{\NP}; (-3,2.3)*{\WP};
(1.5,-1.5)*{\NP}; (-7.5,-1.5)*{\NP};
(-10.5,-4.5)*{\NM}; (-1.5,-4.5)*{\NM};
\endxy
\rightarrow
\xy
(3,9)*{};
(0,6)*{\NP}; (-3,2.3)*{\WM};
(4.5,-4.5)*{\NP}; (1.5,-1.5)*{\NM};
(-7.5,-1.5)*{\NM}; (1.5,-7.5)*{\NM};
\endxy
\rightarrow
\xy
(3,9)*{};
(0,6)*{\NP}; (-3,3)*{\NP};
(0,0)*{\NP}; (3,-3)*{\NM};
(6,-6)*{\NP}; (3,-9)*{\NM};
\endxy
\end{gather*}

The reader can verify that the color vector 1332111 produces the
signs shown.  The reader will also note the the pair consisting of
the first and last trees in the sequence is not prime and that the
``prime factors'' are a pair of trees with four carets each together
with the pair of 2-caret trees used in \(\rot \emptyset\).  The four
caret pair is rigidly colored (it is among the smallest non-identity
examples and one of the standard generators of \(F_4\)), and the two
caret pair is flexibly colored.

The sign structure for the path shown above is connected.  It is
also (computer verified and easily verified by an exhaustive search
by hand) the shortest path between the trees at the two ends.  Among
the features of the sign structure are two pairs of parallel edges,
two independent cycles (treating the parallel edges as one) and only
one vertex not part of a cycle.  The example is a good illustration
of the ``luck'' demanded by the Signed Path Conjectures.  It also
shows that there is no converse to Theorem \ref{PrimeConnThm}.

The flexible coloring of the top two carets seems to ``corrupt'' the
rigidity of the bottom four.  The process that does this is rather
complicated.  To see the extent of the ability of a small amount of
flexibility to corrupt rigidity, the reader can supply the pleasant
inductive step for the following.  The induction is on the number of
carets.  A 4-tree\index{4-tree} is a finite, binary tree in which
every leaf is of even level.  The reason for the terminology is
obvious once the reader tries to draw one.  It is an exercise that
every element of \(F_4\) can be represented by a pair of 4-trees.
Another exercise is that a positive, rigid color vector for a 4-tree
always has the form \((132)^j1\).

\begin{lemma} Let \(T\) be a 4-tree and let \(A\) be the trivial
tree.  In \(T\caret A\), give \(\emptyset\) the positive sign and
let the signs of the rest of the vertices agree with the positive,
rigid coloring on \(T\) (so that the vertex 0 has positive sign in
\(T\caret A\)).  Then there is a sign consistent path from \(T\caret
A\) to the right vine in which all internal vertices are signed
positively.  \end{lemma}

\subsection{Non-prime maps}

The relation between the prime factors of a map and the components
of a sign structure are not clear.  From the above example, it is
seen that there is no one-to-one correspondence.  From the
discussion leading to \tref{PrimeProdFormula} in Sectlion
\ref{TwoRedSec}, we can derive colorings of a non-prime map from the
colorings of the prime factors and compute the number of colorings
of the non-prime map from the number of colorings of the prime
factors.  However, we cannot always compute the components of a sign
structure by simple knowledge of the prime factors of a tree pair.

\begin{question} What is the relationship between the prime factors
of a map and the components of the sign structure?  \end{question}

\section{Acceptable color vectors}\mylabel{AcceptColSec}

This section and the sections that follow gather observations that
either fill earlier promises, or give extra facts, some of which are
related to each other, and that are somewhat off the main narrative.
Section \ref{EnumSec} counts many of the objects that have been
encountered.

After Lemma \ref{NoFiveColLem}, we announced that a color vector
is acceptable (valid for some tree) if and only if the vector is not
a constant and does not sum to zero.  Here we prove that claim.

A color vector that is constant produces a zero wherever an exposed
caret exists, and a vector that sums to zero produces a zero at the
root edge.  Thus the conditions are necessary.  We give the result
after a sequence of lemmas.  

We will use \(\mathbf v\) to denote a color vector that is not a
constant and does not sum to zero.  We will represent the contents
of a vector as strings and combine vectors by concatenation.
Letters \(x\), \(y\) and \(z\) will represent unspecified colors
that are assumed to be different if the letters are different.

\begin{lemma} A constant vector \(x^i\) sums to zero if \(i\) is
even and \(x\) if \(i\) is odd.  \end{lemma}

\begin{lemma}\mylabel{VineClrLem} The vectors \(x^ny\) and \(xy^n\)
are always acceptable.  \end{lemma}

\begin{proof} The first is valid for a right vine (see Section
\ref{VineSec}), and the second is valid for a left vine.
\end{proof}

\begin{lemma}\mylabel{ValValIsValLem} If \(\mathbf v = \mathbf{ps}\)
and both \(\mathbf{p}\) and \(\mathbf{s}\) are acceptable, then
\(\mathbf v\) is acceptable.  \end{lemma}

\begin{proof} Since the sum of \(\mathbf v\) is not zero, the sums
of \(\mathbf p\) and \(\mathbf s\) are not equal.  If \(\mathbf p\)
is valid for \(A\) and \(\mathbf s\) is valid for \(B\), then
\(\mathbf{ps}\) is valid for \(A\caret B\) (see Section
\ref{TreeIndSec}).  \end{proof}

\begin{lemma}\mylabel{ConstAndValidLem} If \(\mathbf v=x^i\mathbf
s\) or \(\mathbf v=\mathbf s x^i\), \(i\ge1\) and \(\mathbf s\) is
acceptable with sum different from \(x\), then \(\mathbf v\) is
acceptable.  \end{lemma}

\begin{proof} If \(\mathbf s\) is valid for \(A\), then \(
x^i\mathbf s\) is valid for the tree formed by identifying the root
edge of \(A\) with the rightmost leaf edge of the right vine \(V_i\)
(right vine with \(i\) internal vertices).  A similar construction
with a left vine handles \(\mathbf s x^i\).  \end{proof}

The next lemma is not used, but is too cute to omit.

\begin{lemma} If \(\mathbf v=x^iy^j\), \(1\ge1\), \(j\ge1\), then
\(\mathbf v\) is acceptable.  \end{lemma}

\begin{proof} The sum of \(\mathbf v\) is not zero, so at least one
of \(i\) or \(j\) is odd.  Assume \(i\) is odd.  From Lemma
\ref{VineClrLem}, we can assume \(j>1\).  The sum of \(x^iy\) is
\(z\), and by Lemma \ref{ConstAndValidLem}, \(\mathbf v =
(x^iy)y^{j-1}\) is acceptable.  If \(j\) is odd, we assume \(i>1\)
and a similar proof works.  \end{proof}

Recall that \([xyz]\) denotes an unspecified choice of one of \(x\),
\(y\) or \(z\).  The expression \([xyz]^n\) denotes a string of
\(n\) such choices rather than one choice repeated \(n\) times.

\begin{prop}\mylabel{ValidCharProp} A color vector of length at
least 2 is acceptable if and only if it is non-constant and has a
non-zero sum.\index{acceptable!color vector!characterization}
\end{prop}

\begin{proof} We only need to argue one direction and we assume
\(\mathbf v\) is not constant and does not sum to zero.  The claim
is true if \(\mathbf v\) has length 2.

We know \(\mathbf v=x^iy[xyz]^{n-i}\) with \(i\ge1\).  If the sum of
\(\mathbf v\) is \(y\) or \(z\), then \(\mathbf v=x\mathbf s\) where
the sum of \(\mathbf s\) cannot be zero and cannot be \(x\).  Now
\(\mathbf v\) is acceptable by Lemma \ref{ConstAndValidLem}.

If the sum of \(\mathbf v\) is \(x\) and the last color in \(\mathbf
v\) is not \(x\), then \(\mathbf v=\mathbf p[yz]\) and where the sum
of \(\mathbf p\) is \(x\) minus the last color in \(\mathbf
v\).  Again \(\mathbf v\) is acceptable by Lemma
\ref{ConstAndValidLem}.

If the sum of \(\mathbf v\) is \(x\) and the last color in \(\mathbf
v\) is \(x\), then \(\mathbf v=x(x^{i-1}y[xyz]^{n-i-1})x\), and the
sum of \(\mathbf m=(x^{i-1}y[xyz]^{n-i-1})\) must be \(x\).  Now
\(\mathbf m=\mathbf {ps}\) where \(\mathbf p=x^{i-1}y\) and
\(\mathbf s\) is the rest of \(\mathbf m\).  The sum of \(\mathbf
p\) is \(y\) or \(z\), so the sum of \(\mathbf s\) is \(z\) or
\(y\).  Now \(\mathbf v=(x\mathbf p)(\mathbf sx)\) has been
decomposed into two pieces with lengths at least two whose sums are
each different from \(x\).  Since each includes \(x\), neither is
constant.  By induction on length, each of \(x\mathbf p\) and
\(\mathbf sx\) are acceptable and \(\mathbf v\) is acceptable by
Lemma \ref{ValValIsValLem}.  \end{proof}

Proposition \ref{ValidCharProp} supplies a converse to Proposition 2
of \Zeilberger.  To state the converse, we write that if \(w\) is a
word in the non-identity elements of \(\Z_2\times \Z_2\) and
\((x,y,z)\) represent these three elements in some order, then
\(|w|_x\) is the number of occurrences of \(x\) in \(w\), and so
forth.

\begin{lemma}  If \(w\) is a non-constant finite word in \((x,y,z)\)
and
\[
|w|_x \equiv |w|_y \not\equiv |w|_z\equiv |w| \mod 2,
\]
then \(w\) is valid for some tree \(T\).  \end{lemma}

\begin{proof} All that is needed is that the sum not be zero.  The
brief proof that the sum is \(z\) is given in \Zeilberger{} from its
point of view.  From our view the argument is as follows.  If
\(|w|_x\) and \(|w|_y\) are both even, then the letters \(x\) and
\(y\) contribute zero to the sum and \(|w|_z\) being odd makes the
letter \(z\) contribute \(z\).  If \(|w|_x\) and \(|w|_y\) are both
odd, then letters \(x\) and \(y\) contribute \(x+y=z\) and the
letter \(z\) contributes zero.  \end{proof}

\section{Patterns}\mylabel{PatternSec}

\subsection{Patterns and multiplication}

This section generalizes some of the observations in the proof of
Proposition \ref{RigidIsSubProp} which found a group arising from a
certain sign assignment on \(\mathcal T\).  However, we have not
explored the generalization that we are about to present and it is
not clear that it opens up much in the way of examples.  There is
one example that is rather trivial to establish and that will be
given later in this section.  Even though it is trivial to establish,
it is not trivial in structure.  However, the structure has not yet
been explored.

A {\itshape pattern}\index{pattern} \(P\) is a sign assignment on
\(\mathcal T\).  Any finite, binary tree inherits a sign assignment
from \(P\), and the convention that the root edge be colored 1 gives
each finite binary tree an edge coloring derived from \(P\).  If the
vector of the colors of the leaf edges of two trees turn out to be
identical under the coloring derived from \(P\), then the common
color vector is valid for the pair.  In such cases, we will say that
the pair is \(P\)-compatible\index{P-compatible!tree
pair}\index{pair!of finite trees!P-compatible}.

An example of a pattern is the positive rigid
pattern\index{positive!rigid!pattern}
\index{rigid!pattern!positive}\index{pattern!positive rigid}that we
denote \(P^r\)\index{\protect\(P^r\protect\)}.  It was defined as
\(P^r(v) = (-1)^{|v|}\).  The negative rigid 
pattern\index{negative!rigid!pattern}
\index{rigid!pattern!negative}\index{pattern!negative rigid}\(-P^r\) was
also defined.  Proposition \ref{RigidIsSubProp} shows that the
\(P^r\)-compatible pairs form a subgroup of \(F\).  We discuss
conditions under which an arbitrary pattern \(P\) leads to a
subgroup.

To help with the discussion, we look at
subpatterns\index{subpattern}.  If \(P\) is a pattern and \(u\) is a
vertex in \(\mathcal T\), then
\(P_u\)\index{\protect\(P_u\protect\)} will be defined by \(P_u(v) =
P(uv)\).  If we recall that \(\mathcal T_u\) is defined by
\(\mathcal T_u = \{uv\mid v\in \mathcal T\}\), then we see that
\(P_u\) is the composition of the obvious isomorphism \(v\mapsto
uv\) from \(\mathcal T\) to \(\mathcal T_u\) with the restriction of
\(P\) to \(\mathcal T_u\).

If we now have a pattern \(P\) and a pair of \(P\)-compatible pairs
\((A,B)\) and \((C,D)\), then \((A,B)(C,D)\) has a coloring since
for trivial reasons the coloring of \((A,B)\) and \((C,D)\) induce
the same coloring on \(B\cap C\) and the Compatibility Lemma (Lemma
\ref{CompatLem}) applies.  However, the coloring on the product
might not be derived from \(P\).

We say that the pair of pairs \(((A,B),(C,D))\) is
\(P\)-compatible\index{P-compatible!pair of pairs}
if both of the following hold.  

\begin{enumerate}

\item  For every leaf \(v\) of 
\(B\cap C\) that is a leaf of \(C\) and not a leaf of \(B\) and leaf
\(w\) of \(D\) that occupies the same position in the left-right
order of the leaves of \(D\) as \(v\) does in the leaves of \(C\),
then  \(P_v=P_w\).

\item  For every leaf \(v\) of 
\(B\cap C\) that is a leaf of \(B\) and not a leaf of \(C\) and leaf
\(w\) of \(A\) that occupies the same position in the left-right
order of the leaves of \(A\) as \(v\) does in the leaves of \(B\),
then  \(P_v=P_w\).

\end{enumerate}

Our overuse of the term \(P\)-compatible leads to following which
has to be read carefully.

\begin{lemma}\mylabel{PCompatLem} If \(P\) is a pattern and
\(((A,B),(C,D))\) is a \(P\)-compatible pair of \(P\)-compatible
pairs of binary trees (sic), then the product \((A,B)(C,D)\) is
\(P\)-compatible.  \end{lemma}

\begin{proof} The proof is identical to the proof of Proposition
\ref{RigidIsSubProp}.  \end{proof}

\subsection{The positive pattern and group}\mylabel{PosPatSec}

We apply Lemma \ref{PCompatLem} to a trivial setting.  We define the
pattern \(P^+\)\index{\protect\(P^+\protect\)} by having \(P^+(v)\)
positive for all \(v\).  For every vertex \(u\), we have
\(P^+_u=P^+\).

\begin{lemma} The set of \(P^+\)-compatible pairs of trees forms a
group under multiplication.  \end{lemma}

\begin{proof}  There is no way for the hypotheses of Lemma
\ref{PCompatLem} to fail.  \end{proof}

Here is what the all positive pattern\index{pattern!all
positive}\index{positive!pattern} \(P^+\) looks like to level 6.
\[
\xy
(0.00, 0.00); (0.99, 4.96)**@{-}; (1.98, 0.00)**@{-}; 
(0.00, -2.00)*{\scriptstyle{1}}; 
(1.98, -2.00)*{\scriptstyle{2}}; 
(3.97, 0.00); (4.96, 4.96)**@{-}; (5.95, 0.00)**@{-}; 
(3.97, -2.00)*{\scriptstyle{2}}; 
(5.95, -2.00)*{\scriptstyle{3}}; 
(7.94, 0.00); (8.93, 4.96)**@{-}; (9.92, 0.00)**@{-}; 
(7.94, -2.00)*{\scriptstyle{2}}; 
(9.92, -2.00)*{\scriptstyle{3}}; 
(11.90, 0.00); (12.90, 4.96)**@{-}; (13.89, 0.00)**@{-}; 
(11.90, -2.00)*{\scriptstyle{3}}; 
(13.89, -2.00)*{\scriptstyle{1}}; 
(15.87, 0.00); (16.87, 4.96)**@{-}; (17.86, 0.00)**@{-}; 
(15.87, -2.00)*{\scriptstyle{2}}; 
(17.86, -2.00)*{\scriptstyle{3}}; 
(19.84, 0.00); (20.83, 4.96)**@{-}; (21.83, 0.00)**@{-}; 
(19.84, -2.00)*{\scriptstyle{3}}; 
(21.83, -2.00)*{\scriptstyle{1}}; 
(23.81, 0.00); (24.80, 4.96)**@{-}; (25.79, 0.00)**@{-}; 
(23.81, -2.00)*{\scriptstyle{3}}; 
(25.79, -2.00)*{\scriptstyle{1}}; 
(27.78, 0.00); (28.77, 4.96)**@{-}; (29.76, 0.00)**@{-}; 
(27.78, -2.00)*{\scriptstyle{1}}; 
(29.76, -2.00)*{\scriptstyle{2}}; 
(31.75, 0.00); (32.74, 4.96)**@{-}; (33.73, 0.00)**@{-}; 
(31.75, -2.00)*{\scriptstyle{2}}; 
(33.73, -2.00)*{\scriptstyle{3}}; 
(35.71, 0.00); (36.71, 4.96)**@{-}; (37.70, 0.00)**@{-}; 
(35.71, -2.00)*{\scriptstyle{3}}; 
(37.70, -2.00)*{\scriptstyle{1}}; 
(39.68, 0.00); (40.67, 4.96)**@{-}; (41.67, 0.00)**@{-}; 
(39.68, -2.00)*{\scriptstyle{3}}; 
(41.67, -2.00)*{\scriptstyle{1}}; 
(43.65, 0.00); (44.64, 4.96)**@{-}; (45.63, 0.00)**@{-}; 
(43.65, -2.00)*{\scriptstyle{1}}; 
(45.63, -2.00)*{\scriptstyle{2}}; 
(47.62, 0.00); (48.61, 4.96)**@{-}; (49.60, 0.00)**@{-}; 
(47.62, -2.00)*{\scriptstyle{3}}; 
(49.60, -2.00)*{\scriptstyle{1}}; 
(51.59, 0.00); (52.58, 4.96)**@{-}; (53.57, 0.00)**@{-}; 
(51.59, -2.00)*{\scriptstyle{1}}; 
(53.57, -2.00)*{\scriptstyle{2}}; 
(55.56, 0.00); (56.55, 4.96)**@{-}; (57.54, 0.00)**@{-}; 
(55.56, -2.00)*{\scriptstyle{1}}; 
(57.54, -2.00)*{\scriptstyle{2}}; 
(59.52, 0.00); (60.52, 4.96)**@{-}; (61.51, 0.00)**@{-}; 
(59.52, -2.00)*{\scriptstyle{2}}; 
(61.51, -2.00)*{\scriptstyle{3}}; 
(63.49, 0.00); (64.48, 4.96)**@{-}; (65.48, 0.00)**@{-}; 
(63.49, -2.00)*{\scriptstyle{2}}; 
(65.48, -2.00)*{\scriptstyle{3}}; 
(67.46, 0.00); (68.45, 4.96)**@{-}; (69.44, 0.00)**@{-}; 
(67.46, -2.00)*{\scriptstyle{3}}; 
(69.44, -2.00)*{\scriptstyle{1}}; 
(71.43, 0.00); (72.42, 4.96)**@{-}; (73.41, 0.00)**@{-}; 
(71.43, -2.00)*{\scriptstyle{3}}; 
(73.41, -2.00)*{\scriptstyle{1}}; 
(75.40, 0.00); (76.39, 4.96)**@{-}; (77.38, 0.00)**@{-}; 
(75.40, -2.00)*{\scriptstyle{1}}; 
(77.38, -2.00)*{\scriptstyle{2}}; 
(79.37, 0.00); (80.36, 4.96)**@{-}; (81.35, 0.00)**@{-}; 
(79.37, -2.00)*{\scriptstyle{3}}; 
(81.35, -2.00)*{\scriptstyle{1}}; 
(83.33, 0.00); (84.33, 4.96)**@{-}; (85.32, 0.00)**@{-}; 
(83.33, -2.00)*{\scriptstyle{1}}; 
(85.32, -2.00)*{\scriptstyle{2}}; 
(87.30, 0.00); (88.29, 4.96)**@{-}; (89.29, 0.00)**@{-}; 
(87.30, -2.00)*{\scriptstyle{1}}; 
(89.29, -2.00)*{\scriptstyle{2}}; 
(91.27, 0.00); (92.26, 4.96)**@{-}; (93.25, 0.00)**@{-}; 
(91.27, -2.00)*{\scriptstyle{2}}; 
(93.25, -2.00)*{\scriptstyle{3}}; 
(95.24, 0.00); (96.23, 4.96)**@{-}; (97.22, 0.00)**@{-}; 
(95.24, -2.00)*{\scriptstyle{3}}; 
(97.22, -2.00)*{\scriptstyle{1}}; 
(99.21, 0.00); (100.20, 4.96)**@{-}; (101.19, 0.00)**@{-}; 
(99.21, -2.00)*{\scriptstyle{1}}; 
(101.19, -2.00)*{\scriptstyle{2}}; 
(103.17, 0.00); (104.17, 4.96)**@{-}; (105.16, 0.00)**@{-}; 
(103.17, -2.00)*{\scriptstyle{1}}; 
(105.16, -2.00)*{\scriptstyle{2}}; 
(107.14, 0.00); (108.13, 4.96)**@{-}; (109.13, 0.00)**@{-}; 
(107.14, -2.00)*{\scriptstyle{2}}; 
(109.13, -2.00)*{\scriptstyle{3}}; 
(111.11, 0.00); (112.10, 4.96)**@{-}; (113.10, 0.00)**@{-}; 
(111.11, -2.00)*{\scriptstyle{1}}; 
(113.10, -2.00)*{\scriptstyle{2}}; 
(115.08, 0.00); (116.07, 4.96)**@{-}; (117.06, 0.00)**@{-}; 
(115.08, -2.00)*{\scriptstyle{2}}; 
(117.06, -2.00)*{\scriptstyle{3}}; 
(119.05, 0.00); (120.04, 4.96)**@{-}; (121.03, 0.00)**@{-}; 
(119.05, -2.00)*{\scriptstyle{2}}; 
(121.03, -2.00)*{\scriptstyle{3}}; 
(123.02, 0.00); (124.01, 4.96)**@{-}; (125.00, 0.00)**@{-}; 
(123.02, -2.00)*{\scriptstyle{3}}; 
(125.00, -2.00)*{\scriptstyle{1}}; 
(0.99, 4.96); (2.98, 10.91)**@{-}; (4.96, 4.96)**@{-}; 
(0.49, 6.96)*{\scriptstyle{3}}; 
(5.46, 6.96)*{\scriptstyle{1}}; 
(8.93, 4.96); (10.91, 10.91)**@{-}; (12.90, 4.96)**@{-}; 
(8.43, 6.96)*{\scriptstyle{1}}; 
(13.40, 6.96)*{\scriptstyle{2}}; 
(16.87, 4.96); (18.85, 10.91)**@{-}; (20.83, 4.96)**@{-}; 
(16.37, 6.96)*{\scriptstyle{1}}; 
(21.33, 6.96)*{\scriptstyle{2}}; 
(24.80, 4.96); (26.79, 10.91)**@{-}; (28.77, 4.96)**@{-}; 
(24.30, 6.96)*{\scriptstyle{2}}; 
(29.27, 6.96)*{\scriptstyle{3}}; 
(32.74, 4.96); (34.72, 10.91)**@{-}; (36.71, 4.96)**@{-}; 
(32.24, 6.96)*{\scriptstyle{1}}; 
(37.21, 6.96)*{\scriptstyle{2}}; 
(40.67, 4.96); (42.66, 10.91)**@{-}; (44.64, 4.96)**@{-}; 
(40.17, 6.96)*{\scriptstyle{2}}; 
(45.14, 6.96)*{\scriptstyle{3}}; 
(48.61, 4.96); (50.60, 10.91)**@{-}; (52.58, 4.96)**@{-}; 
(48.11, 6.96)*{\scriptstyle{2}}; 
(53.08, 6.96)*{\scriptstyle{3}}; 
(56.55, 4.96); (58.53, 10.91)**@{-}; (60.52, 4.96)**@{-}; 
(56.05, 6.96)*{\scriptstyle{3}}; 
(61.02, 6.96)*{\scriptstyle{1}}; 
(64.48, 4.96); (66.47, 10.91)**@{-}; (68.45, 4.96)**@{-}; 
(63.98, 6.96)*{\scriptstyle{1}}; 
(68.95, 6.96)*{\scriptstyle{2}}; 
(72.42, 4.96); (74.40, 10.91)**@{-}; (76.39, 4.96)**@{-}; 
(71.92, 6.96)*{\scriptstyle{2}}; 
(76.89, 6.96)*{\scriptstyle{3}}; 
(80.36, 4.96); (82.34, 10.91)**@{-}; (84.33, 4.96)**@{-}; 
(79.86, 6.96)*{\scriptstyle{2}}; 
(84.83, 6.96)*{\scriptstyle{3}}; 
(88.29, 4.96); (90.28, 10.91)**@{-}; (92.26, 4.96)**@{-}; 
(87.79, 6.96)*{\scriptstyle{3}}; 
(92.76, 6.96)*{\scriptstyle{1}}; 
(96.23, 4.96); (98.21, 10.91)**@{-}; (100.20, 4.96)**@{-}; 
(95.73, 6.96)*{\scriptstyle{2}}; 
(100.70, 6.96)*{\scriptstyle{3}}; 
(104.17, 4.96); (106.15, 10.91)**@{-}; (108.13, 4.96)**@{-}; 
(103.67, 6.96)*{\scriptstyle{3}}; 
(108.63, 6.96)*{\scriptstyle{1}}; 
(112.10, 4.96); (114.09, 10.91)**@{-}; (116.07, 4.96)**@{-}; 
(111.60, 6.96)*{\scriptstyle{3}}; 
(116.57, 6.96)*{\scriptstyle{1}}; 
(120.04, 4.96); (122.02, 10.91)**@{-}; (124.01, 4.96)**@{-}; 
(119.54, 6.96)*{\scriptstyle{1}}; 
(124.51, 6.96)*{\scriptstyle{2}}; 
(2.98, 10.91); (6.94, 18.06)**@{-}; (10.91, 10.91)**@{-}; 
(2.48, 12.91)*{\scriptstyle{2}}; 
(11.41, 12.91)*{\scriptstyle{3}}; 
(18.85, 10.91); (22.82, 18.06)**@{-}; (26.79, 10.91)**@{-}; 
(18.35, 12.91)*{\scriptstyle{3}}; 
(27.29, 12.91)*{\scriptstyle{1}}; 
(34.72, 10.91); (38.69, 18.06)**@{-}; (42.66, 10.91)**@{-}; 
(34.22, 12.91)*{\scriptstyle{3}}; 
(43.16, 12.91)*{\scriptstyle{1}}; 
(50.60, 10.91); (54.56, 18.06)**@{-}; (58.53, 10.91)**@{-}; 
(50.10, 12.91)*{\scriptstyle{1}}; 
(59.03, 12.91)*{\scriptstyle{2}}; 
(66.47, 10.91); (70.44, 18.06)**@{-}; (74.40, 10.91)**@{-}; 
(65.97, 12.91)*{\scriptstyle{3}}; 
(74.90, 12.91)*{\scriptstyle{1}}; 
(82.34, 10.91); (86.31, 18.06)**@{-}; (90.28, 10.91)**@{-}; 
(81.84, 12.91)*{\scriptstyle{1}}; 
(90.78, 12.91)*{\scriptstyle{2}}; 
(98.21, 10.91); (102.18, 18.06)**@{-}; (106.15, 10.91)**@{-}; 
(97.71, 12.91)*{\scriptstyle{1}}; 
(106.65, 12.91)*{\scriptstyle{2}}; 
(114.09, 10.91); (118.06, 18.06)**@{-}; (122.02, 10.91)**@{-}; 
(113.59, 12.91)*{\scriptstyle{2}}; 
(122.52, 12.91)*{\scriptstyle{3}}; 
(6.94, 18.06); (14.88, 26.63)**@{-}; (22.82, 18.06)**@{-}; 
(6.44, 20.06)*{\scriptstyle{1}}; 
(23.32, 20.06)*{\scriptstyle{2}}; 
(38.69, 18.06); (46.63, 26.63)**@{-}; (54.56, 18.06)**@{-}; 
(38.19, 20.06)*{\scriptstyle{2}}; 
(55.06, 20.06)*{\scriptstyle{3}}; 
(70.44, 18.06); (78.37, 26.63)**@{-}; (86.31, 18.06)**@{-}; 
(69.94, 20.06)*{\scriptstyle{2}}; 
(86.81, 20.06)*{\scriptstyle{3}}; 
(102.18, 18.06); (110.12, 26.63)**@{-}; (118.06, 18.06)**@{-}; 
(101.68, 20.06)*{\scriptstyle{3}}; 
(118.56, 20.06)*{\scriptstyle{1}}; 
(14.88, 26.63); (30.75, 36.91)**@{-}; (46.63, 26.63)**@{-}; 
(14.38, 28.63)*{\scriptstyle{3}}; 
(47.13, 28.63)*{\scriptstyle{1}}; 
(78.37, 26.63); (94.25, 36.91)**@{-}; (110.12, 26.63)**@{-}; 
(77.87, 28.63)*{\scriptstyle{1}}; 
(110.62, 28.63)*{\scriptstyle{2}}; 
(30.75, 36.91); (62.50, 49.26)**@{-}; (94.25, 36.91)**@{-}; 
(30.25, 38.91)*{\scriptstyle{2}}; 
(94.75, 38.91)*{\scriptstyle{3}}; 
(62.50, 51.26)*{\scriptstyle{1}}; 
\endxy
\]

Let us use \(F^+\)\index{\protect\(F^+\protect\)} to denote the
subgroup of \(F\) consisting of the \(P^+\)-compatible pairs.  While
it is clear that the identity element is in \(F^+\), it takes work
to find non-trivial elements.  Below is what seems to be the
simplest (measured by the size of the trees involved).
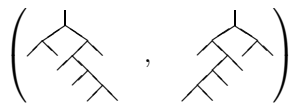
\[
\left(
\xy
(0,0); (2,2)**@{-}; (4,0)**@{-};
(2,2); (5,4)**@{-}; (8,2)**@{-}; (5,4); (5,6)**@{-};
(6,0); (8,2)**@{-}; (10,0)**@{-};
(4,-2); (6,0)**@{-}; (8,-2)**@{-};
(6,-4); (8,-2)**@{-}; (10,-4)**@{-};
(8,-6); (10,-4)**@{-}; (12,-6)**@{-};
\endxy
\quad, \quad
\xy
(0,0); (-2,2)**@{-}; (-4,0)**@{-};
(-2,2); (-5,4)**@{-}; (-8,2)**@{-}; (-5,4); (-5,6)**@{-};
(-6,0); (-8,2)**@{-}; (-10,0)**@{-};
(-4,-2); (-6,0)**@{-}; (-8,-2)**@{-};
(-6,-4); (-8,-2)**@{-}; (-10,-4)**@{-};
(-8,-6); (-10,-4)**@{-}; (-12,-6)**@{-};
\endxy
\right)
\]
It has a nice symmetry that other examples do not.  Powers of this
element follow a nice pattern.  A less symmetric example is shown
below.
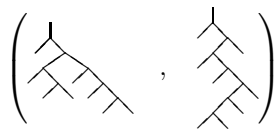
\[
\left(
\xy
(0,-2); (2,0)**@{-}; (4,-2)**@{-};
(2,0); (5,2)**@{-}; (8,0)**@{-};
(5,2); (3,4)**@{-}; (1,2)**@{-}; (3,4)**@{-}; (3,6)**@{-};
(6,-2); (8,0)**@{-}; (10,-2)**@{-};
(2,-4); (4,-2)**@{-}; (6,-4)**@{-}; 
(8,-4); (10,-2)**@{-}; (12,-4)**@{-};
(10,-6); (12,-4)**@{-}; (14,-6)**@{-};
\endxy
\quad, \quad
\xy
(0,0); (2,2)**@{-}; (4,0)**@{-};
(2,2); (4,4)**@{-}; (6,2)**@{-};
(4,4); (2,6)**@{-}; (0,4)**@{-}; (2,6); (2,8)**@{-};
(2,-2); (4,0)**@{-}; (6,-2)**@{-};
(4,-4); (6,-2)**@{-}; (8,-4)**@{-};
(2,-6); (4,-4)**@{-}; (6,-6)**@{-};
(0,-8); (2,-6)**@{-}; (4,-8)**@{-};
\endxy
\right)
\]

The group \(F^+\) has a property shared by \(F\). If \(T\) is a
finite binary tree, if \((A,B)\) is a pair in \(F^+\), and \(v\) is
a leaf of \(T\), then we can form a new element of \(F^+\).  We do
this by attaching \(A\) to \(T\) at \(v\) to form \(A'\) and \(B\)
to \(T\) at \(v\) to form \(B'\).  This means that \(A'\) is
obtained by identifying the root edge of \(A\) with the leaf edge of
\(T\) that impinges on \(v\).  A more technical definition is
\(A'=T\cup vA\) where the vertices of \(vA\) are those in \(\{vu\mid
u\in A\}\).  The new element of \(F^+\) is \((A',B')\).

Note that the element of \(F^+\) that results depends only on
\((A,B)\) and \(v\) and not on the particular \(T\) that has \(v\)
as a leaf.  Following \cite{mck+thomp}, we refer to \((A',B')\) as
the {\itshape deferment}\index{deferment!of element of
\protect\(F\protect\)}\index{Thompson's group
\protect\(F\protect\)!deferment of element} of \((A,B)\) to \(v\).
We thus have that, like \(F\), the group \(F^+\) is closed under
deferment.

\subsection{Neighborhoods of positive colorings}\mylabel{PosNbhdSec}

Each vertex in \(A_d\) has \(d\) edges impinging on it.  For a given
color vector \(\mathbf c\), there might be fewer than \(d\) valid
edges impinging on a given vertex.  If the sign assignment on a
given vertex is all positive (or all negative), then all edges
leaving from that vertex are valid.  It might be guessed that
starting with a coloring of a vertex that gives an all positive sign
assignment would give a large color graph.  To give words to the
concept, let us use the word {\itshape positive
neighborhood}\index{positive!neighborhood!of
vertex}\index{vertex!positive
neighborhood}\index{neighborhood!positive!of vertex} of a vertex
\(T\) of \(A_d\) to refer to the color graph of that vector
\(\mathbf c\) (normalized to give 1 as the root edge color) that
gives the all positive sign assignment on \(T\).

In \(A_2\), the positive neighborhoods are all the paths of length
2.  Since every pair of vertices in \(A_2\) is in a path of length
2, we have that there is a coloring for every pair of vertices in
\(A_2\).

If for every \(d\), it is true that each pair of vertices in \(A_d\)
lies in the positive neighborhood of some third vertex, then the 4CT
would follow.  In fact this fails for \(d=3\).  

The two trees shown below
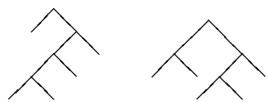
\[
{\xy
(3,9);(6,12)**@{-};(9,9)**@{-}; 
(6,6);(9,9)**@{-};(12,6)**@{-};
(3,3);(6,6)**@{-};(9,3)**@{-};
(0,0);(3,3)**@{-};(6,0)**@{-};
\endxy}
\qquad
{\xy
(6,6);(10.5,10.5)**@{-};(15,6)**@{-};
(3,3);(6,6)**@{-};(9,3)**@{-};
(9,0);(12,3)**@{-};(15,0)**@{-};
(12,3);(15,6)**@{-};(18,3)**@{-};
\endxy}
\]
have only one normalized coloring and it is given by the color
vector \(21211\).  With the top vertex of the left tree assigned
\(+\), there are only 8 possible sign assignments to check.  These
produce 8 color vectors for the left tree, of which the only one
valid for the right tree is \(21211\).  The full color graph for the
vector \(21211\) is shown below where the signs of the vertices are
shown rather than the colors.
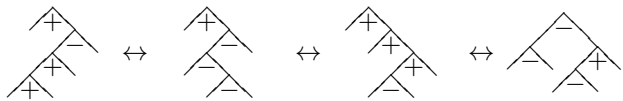
\[
\xy
(0,0)*{\xy
(3,9);(6,12)**@{-};(9,9)**@{-}; (6,10)*{+};
(6,6);(9,9)**@{-};(12,6)**@{-}; (9,7)*{-};
(3,3);(6,6)**@{-};(9,3)**@{-};  (6,4)*{+};
(0,0);(3,3)**@{-};(6,0)**@{-};  (3,1)*{+};
\endxy};
(11,0)*{\leftrightarrow};
(22,0)*{\xy
(3,9);(6,12)**@{-};(9,9)**@{-}; (6,10)*{+};
(6,6);(9,9)**@{-};(12,6)**@{-}; (9,7)*{-};
(3,3);(6,6)**@{-};(9,3)**@{-};  (6,4)*{-};
(6,0);(9,3)**@{-};(12,0)**@{-};  (9,1)*{-};
\endxy};
(34,0)*{\leftrightarrow};
(45,0)*{\xy
(3,9);(6,12)**@{-};(9,9)**@{-}; (6,10)*{+};
(6,6);(9,9)**@{-};(12,6)**@{-}; (9,7)*{+};
(9,3);(12,6)**@{-};(15,3)**@{-};  (12,4)*{+};
(6,0);(9,3)**@{-};(12,0)**@{-};  (9,1)*{-};
\endxy};
(57,0)*{\leftrightarrow};
(68,0)*{\xy
(6,6);(10.5,10.5)**@{-};(15,6)**@{-}; (10.5,8.5)*{-};
(3,3);(6,6)**@{-};(9,3)**@{-}; (6,4)*{-};
(12,3);(15,6)**@{-};(18,3)**@{-}; (15,4)*{+};
(9,0);(12,3)**@{-};(15,0)**@{-}; (12,1)*{-};
\endxy};
\endxy
\]

None of the sign assignments is all positive (or all negative).

Note that the fact that there are only four vertices in the color
graph of \(21211\) can be argued by zero sets.  The zero set is the
union of two pentagons.  In \tref{TheAssoc} they are the disjoint
pentagons on the lower right and upper left.  The four vertices not
in the zero set are the vertices shown above.

\section{Higher genera and Thompson's group
\protect\(V\protect\)}\mylabel{HGV}

This section relates to the discussion in the last four paragraphs of
Section 2 of \EliaOne.

\subsection{Tree pairs on the torus}\mylabel{TreeTorusSec}

The following is Figure 2-16 on Page 35 of \Saaty{} with the addition
of labels on the vertices.
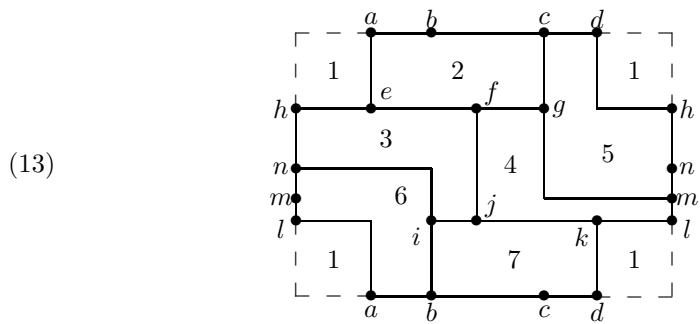
\mymargin{SevenColEx}\begin{equation}\label{SevenColEx}
\begin{split}
\xy
(0,0); (10,0)**@{--}; (40,0)**@{-}; (50,0)**@{--}; 
(50,10)**@{--}; (50,25)**@{-};
(50,35)**@{--}; (40,35)**@{--};  (10,35)**@{-};
(0,35)**@{--}; (0,25)**@{--}; (0,10)**@{-}; (0,0)**@{--};
(10,0); (10,10)**@{-}; (0,10)**@{-};
(18,0); (18,17)**@{-}; (0,17)**@{-};
(18,10); (50,10)**@{-}; (40,10); (40,0)**@{-};
(24,10); (24,25)**@{-}; (0,25); (33,25)**@{-};
(33,35); (33,13)**@{-}; (50,13)**@{-};
(10,25); (10,35)**@{-}; 
(40,35); (40,25)**@{-}; (50,25)**@{-};
(5,5)*{1}; (45,5)*{1}; (45,30)*{1}; (5,30)*{1};
(29,5)*{7}; (14, 13.5)*{6}; (21.5, 30)*{2};
(12,21)*{3}; (28.5, 17.5)*{4}; (41.5, 19)*{5};
(10,-2)*{a}; (10,37)*{a}; (18,37)*{b}; (33,37)*{c};
(10,0)*{\bullet}; (10,35)*{\bullet}; (18,35)*{\bullet};
(33,35)*{\bullet}; (40,35)*{\bullet}; (40,37)*{d};
(0,25)*{\bullet}; (-2,25)*{h}; (10,25)*{\bullet}; (12,27)*{e};
(24,25)*{\bullet}; (26,27)*{f}; (33,25)*{\bullet}; (35,25)*{g};
(50,25)*{\bullet}; (52,25)*{h}; (0,17)*{\bullet}; (-2,17)*{n};
(0,13)*{\bullet}; (-2,13)*{m}; (50,17)*{\bullet}; (52,17)*{n};
(50,13)*{\bullet}; (52,13)*{m}; (0,10)*{\bullet}; (-2,9)*{l};
(50,10)*{\bullet}; (52,9)*{l}; (18,10)*{\bullet}; (16,8)*{i};
(24,10)*{\bullet}; (26,12)*{j}; (40,10)*{\bullet}; (38,8)*{k};
(18,0)*{\bullet}; (18,-2)*{b}; (33,0)*{\bullet}; (33,-2)*{c};
(40,0)*{\bullet}; (40,-2)*{d};
\endxy
\end{split}
\end{equation}

The figure shows a map on a torus\index{torus!map needing seven
colors}.  The figure is to be interpreted as having the top edge of
the rectangle identified with the bottom edge and the left edge of
the rectangle identified with the right edge.  Some parts of the
top, bottom and sides are drawn as dashed lines since they are not
part of the graph of the map and are interior to face 1.  The
point of the figure is that the graph cuts the torus into seven
faces and each pair of faces shares an edge.  Thus there is no
coloring of the faces with less than seven colors.

We can cut the graph into two pieces.  If the cuts are made at the
numbered points indicated by \(\times\) below, then the two pieces
are trees of eight leaves each.
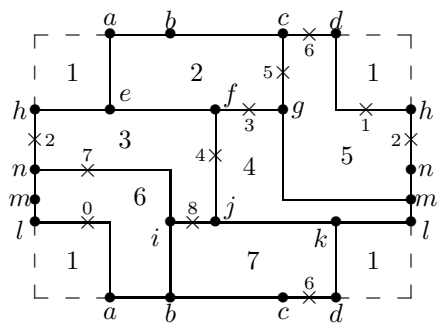
\[
\xy
(0,0); (10,0)**@{--}; (40,0)**@{-}; (50,0)**@{--}; 
(50,10)**@{--}; (50,25)**@{-};
(50,35)**@{--}; (40,35)**@{--};  (10,35)**@{-};
(0,35)**@{--}; (0,25)**@{--}; (0,10)**@{-}; (0,0)**@{--};
(10,0); (10,10)**@{-}; (0,10)**@{-};
(18,0); (18,17)**@{-}; (0,17)**@{-};
(18,10); (50,10)**@{-}; (40,10); (40,0)**@{-};
(24,10); (24,25)**@{-}; (0,25); (33,25)**@{-};
(33,35); (33,13)**@{-}; (50,13)**@{-};
(10,25); (10,35)**@{-}; 
(40,35); (40,25)**@{-}; (50,25)**@{-};
(5,5)*{1}; (45,5)*{1}; (45,30)*{1}; (5,30)*{1};
(29,5)*{7}; (14, 13.5)*{6}; (21.5, 30)*{2};
(12,21)*{3}; (28.5, 17.5)*{4}; (41.5, 19)*{5};
(10,-2)*{a}; (10,37)*{a}; (18,37)*{b}; (33,37)*{c};
(10,0)*{\bullet}; (10,35)*{\bullet}; (18,35)*{\bullet};
(33,35)*{\bullet}; (40,35)*{\bullet}; (40,37)*{d};
(0,25)*{\bullet}; (-2,25)*{h}; (10,25)*{\bullet}; (12,27)*{e};
(24,25)*{\bullet}; (26,27)*{f}; (33,25)*{\bullet}; (35,25)*{g};
(50,25)*{\bullet}; (52,25)*{h}; (0,17)*{\bullet}; (-2,17)*{n};
(0,13)*{\bullet}; (-2,13)*{m}; (50,17)*{\bullet}; (52,17)*{n};
(50,13)*{\bullet}; (52,13)*{m}; (0,10)*{\bullet}; (-2,9)*{l};
(50,10)*{\bullet}; (52,9)*{l}; (18,10)*{\bullet}; (16,8)*{i};
(24,10)*{\bullet}; (26,12)*{j}; (40,10)*{\bullet}; (38,8)*{k};
(18,0)*{\bullet}; (18,-2)*{b}; (33,0)*{\bullet}; (33,-2)*{c};
(40,0)*{\bullet}; (40,-2)*{d}; 
(36.5,35)*{\times}; (36.5,33)*{\scriptstyle{6}};
(36.5,0)*{\times};  (36.5,2)*{\scriptstyle{6}};
(0,21)*{\times}; (2,21)*{\scriptstyle{2}};
(50,21)*{\times}; (48,21)*{\scriptstyle{2}}; 
(7,17)*{\times}; (7,19)*{\scriptstyle{7}};
(7,10)*{\times}; (7,12)*{\scriptstyle{0}};
(21,10)*{\times}; (21,12)*{\scriptstyle{8}};
(24,19)*{\times}; (22,19)*{\scriptstyle{4}};
(28.5, 25)*{\times}; (28.5,23)*{\scriptstyle{3}};
(33,30)*{\times}; (31,30)*{\scriptstyle{5}};
(44,25)*{\times}; (44,23)*{\scriptstyle{1}};
\endxy
\]

The two trees are pictured below where the numbers on the external
vertices (leaves and root) correspond to the numbers next to the
points labeled \(\times\) above.
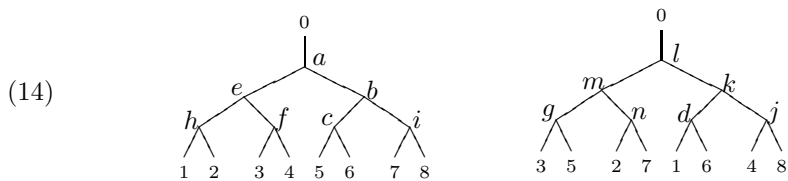
\mymargin{TwoTreesOnTorus}\begin{equation}\label{TwoTreesOnTorus}
\begin{split}
\xy
(0,0); (0,-4)**@{-}; (-8,-8)**@{-}; (-14,-12)**@{-};
(0,-4); (8,-8)**@{-}; (14,-12)**@{-};
(-8,-8); (-4,-12)**@{-}; (8,-8); (4,-12)**@{-};
(-16,-16); (-14,-12)**@{-}; (-12,-16)**@{-};
(16,-16); (14,-12)**@{-}; (12,-16)**@{-};
(-6,-16); (-4,-12)**@{-}; (-2,-16)**@{-};
(6,-16); (4,-12)**@{-}; (2,-16)**@{-};
(0,2)*{\scs 0}; (2,-3)*{a};
(-9,-7)*{e}; (9,-7)*{b}; (-15,-11)*{h};
(-3,-11)*{f}; (3,-11)*{c}; (15,-11)*{i};
(-16,-18)*{\scs 1}; (-12,-18)*{\scs 2};
(-6,-18)*{\scs 3}; (-2,-18)*{\scs 4}; 
(2,-18)*{\scs 5}; (6,-18)*{\scs 6};
(12,-18)*{\scs 7}; (16,-18)*{\scs 8};
\endxy
\qquad\qquad
\xy
(0,0); (0,-4)**@{-}; (-8,-8)**@{-}; (-14,-12)**@{-};
(0,-4); (8,-8)**@{-}; (14,-12)**@{-};
(-8,-8); (-4,-12)**@{-}; (8,-8); (4,-12)**@{-};
(-16,-16); (-14,-12)**@{-}; (-12,-16)**@{-};
(16,-16); (14,-12)**@{-}; (12,-16)**@{-};
(-6,-16); (-4,-12)**@{-}; (-2,-16)**@{-};
(6,-16); (4,-12)**@{-}; (2,-16)**@{-};
(0,2)*{\scs 0}; (2,-3)*{l};
(-9,-7)*{m}; (9,-7)*{k}; (-15,-11)*{g};
(-3,-11)*{n}; (3,-11)*{d}; (15,-11)*{j};
(-16,-18)*{\scs 3}; (-12,-18)*{\scs 5};
(-6,-18)*{\scs 2}; (-2,-18)*{\scs 7}; 
(2,-18)*{\scs 1}; (6,-18)*{\scs 6};
(12,-18)*{\scs 4}; (16,-18)*{\scs 8};
\endxy
\end{split}
\end{equation}

The numbering and the embeddings of the trees in the plane have been
chosen so that the leaves on the left tree are numbered from 1
through 8 in order in the left-right ordering of the leaves.  On the
right tree this is not possible because the graph in
\tref{SevenColEx} is not planar.

It is simple to check that the color vector \(\mathbf c =
(1,3,1,2,2,3,1,3)\) is valid for the left tree.  If the colors are
assigned on the right tree by using the numbering of the leaves as
shown (and not the left-right ordering) so that leaf \(i\) gets
color \(\mathbf c_i\), then it is seen that the coloring is valid
for the right tree as well.  It follows that the graph in
\tref{SevenColEx} has a proper, edge 3-coloring in spite of the fact
that as a map it does not have a proper, face 4-coloring.

\subsection{Thompson's group \protect\(V\protect\)}\mylabel{VSec}

The tree pair shown in \tref{TwoTreesOnTorus} represents an element
of another of Thompson's groups known as \(V\)\index{Thompson's
group \protect\(V\protect\)}\index{\protect\(V\protect\)}.  A
typical representative is a triple \((D, \sigma, R)\) where \(D\)
and \(R\) are a pair of finite, binary trees with the same number of
leaves and \(\sigma\) is a bijection from the leaves of \(D\) to the
leaves of \(R\).  See \cite{CFP} for descriptions of the equivalence
relation on the triples used to define the elements of \(V\) and the
multiplication on the triples.  These are easy enough to guess
correctly.

We give an example to show that not every element of \(V\) has a
valid edge 3-coloring.  That is, there is a triple \((D,\sigma,
R)\) where \(D\) and \(R\) have \(n\) leaves each so that no color
vector \(\mathbf c = (\mathbf c_1, \ldots, \mathbf c_n)\) of length
\(n\) exists that is valid for both \(D\) and \(R\) in the following
sense.  If \(v_1\), \dots, \(v_n\) are the leaves of \(D\) in
left-right order then for each \(i\) with \(1\le i\le n\) the color
\(\mathbf c_i\) is applied to leaf \(v_i\) of \(D\) and leaf
\(\sigma(v_i)\) of \(R\).

The smallest examples (by computer search) have 6 leaves for each
tree, and our example is of this size.  Examples are not rare.  Of
13,800 triples involving trees with 6 leaves, there are 3,584
triples with no valid edge 3-coloring.

Below are two trees where the leaves of each tree are numbered to
indicate the bijection between the two sets of leaves.
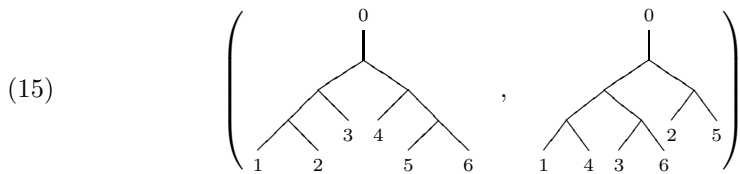
\mymargin{NoColorV}\begin{equation}\label{NoColorV}
\begin{split}
\left(\,\,
\xy
(-18,-8); (-10,0)**@{-}; (-4,4)**@{-}; (2,0)**@{-}; (10,-8)**@{-};
(-4,4); (-4,8)**@{-};
(6,-4); (2,-8)**@{-}; (-14,-4); (-10,-8)**@{-}; 
(-10,0); (-6,-4)**@{-}; (2,0); (-2,-4)**@{-};
(-18,-10)*{\scriptstyle{1}};
(-10,-10)*{\scriptstyle{2}};
(-6,-6)*{\scriptstyle{3}};
(-2,-6)*{\scriptstyle{4}};
(2,-10)*{\scriptstyle{5}};
(10,-10)*{\scriptstyle{6}};
(-4,10)*{\scriptstyle{0}};
\endxy
\quad,\quad
\xy
(0,-4); (3,0)**@{-}; (6,-4)**@{-}; 
(3,0); (-3,4)**@{-}; (-9,0)**@{-}; (-3,4); (-3,8)**@{-};
(-14,-4); (-9,0)**@{-}; (-4,-4)**@{-};
(-17,-8); (-14,-4)**@{-}; (-11,-8)**@{-};
(-7,-8); (-4,-4)**@{-}; (-1,-8)**@{-};
(-3,10)*{\scriptstyle{0}};
(-17,-10)*{\scriptstyle{1}};
(-11,-10)*{\scriptstyle{4}};
(-7,-10)*{\scriptstyle{3}};
(-1,-10)*{\scriptstyle{6}};
(0,-6)*{\scriptstyle{2}};
(6,-6)*{\scriptstyle{5}};
\endxy
\right)
\end{split}
\end{equation}

To discuss the possible colorings of the two trees we redraw the
trees by labeling the leaves and vertices with letters.  We use
letters \(a\) through \(d\) to represent colors not yet assigned.
We use \(x\), \(y\) and \(z\) for assigned colors where the colors
\(x\), \(y\) and \(z\) represent different colors.  We can assign
the two colors on leaves 1 and 2 arbitrarily (as long as they are
different) and we assign them colors \(x\) and \(y\), respectively.
We get the following picture.
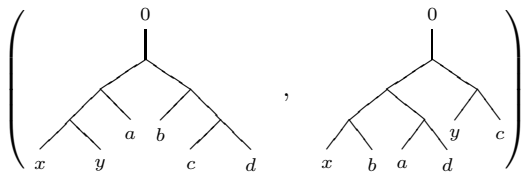
\[
\begin{split}
\left(\,\,
\xy
(-18,-8); (-10,0)**@{-}; (-4,4)**@{-}; (2,0)**@{-}; (10,-8)**@{-};
(-4,4); (-4,8)**@{-};
(6,-4); (2,-8)**@{-}; (-14,-4); (-10,-8)**@{-}; 
(-10,0); (-6,-4)**@{-}; (2,0); (-2,-4)**@{-};
(-18,-10)*{\scriptstyle{x}};
(-10,-10)*{\scriptstyle{y}};
(-6,-6)*{\scriptstyle{a}};
(-2,-6)*{\scriptstyle{b}};
(2,-10)*{\scriptstyle{c}};
(10,-10)*{\scriptstyle{d}};
(-4,10)*{\scriptstyle{0}};
\endxy
\quad,\quad
\xy
(0,-4); (3,0)**@{-}; (6,-4)**@{-}; 
(3,0); (-3,4)**@{-}; (-9,0)**@{-}; (-3,4); (-3,8)**@{-};
(-14,-4); (-9,0)**@{-}; (-4,-4)**@{-};
(-17,-8); (-14,-4)**@{-}; (-11,-8)**@{-};
(-7,-8); (-4,-4)**@{-}; (-1,-8)**@{-};
(-3,10)*{\scriptstyle{0}};
(-17,-10)*{\scriptstyle{x}};
(-11,-10)*{\scriptstyle{b}};
(-7,-10)*{\scriptstyle{a}};
(-1,-10)*{\scriptstyle{d}};
(0,-6)*{\scriptstyle{y}};
(6,-6)*{\scriptstyle{c}};
\endxy
\right)
\end{split}
\]

We have \(c\ne d\), and the color \(b\) cannot equal \(c+d\), so
either \(b=c\) or \(b=d\).  We will consider the two cases
separately, but first we note facts that apply to either case.  We
have \(a\ne x+y=z\), \(c\ne d\), \(b\ne x\), \(a\ne d\), and \(c\ne
y\).

{\itshape Case 1:} \(b=c\).  Now \(c\ne y\) and \(c=b\ne x\) has
\(b=c=z\).  With \(b=c\), we get \(b+c+d=d\) which cannot equal
\(x+y+a=z+a\).  But \(a\ne d\) so \(d=z\).  But this makes \(c=d\)
which cannot happen.

{\itshape Case 2:} \(b=d\).  Now \(x+b\ne a+d = a+b\) so \(a\ne x\).
We also have \(a\ne x+y=z\), so \(a=y\).  But now \(b=d\ne a=y\) and
\(d=b\ne x\) so \(b=d=z\).  Now \(b+c+d=c\) which cannot equal
\(x+y+a=z+a = z+y=x\).  But \(c\ne y\) so \(c=z\).  But now
\(c=d=z\) which cannot happen.

Thus \tref{NoColorV} gives an example of an element of \(V\) with no
valid coloring.  We now show that this is a well known fact.

\subsection{Tree pairs on the projective plane}\mylabel{RP2Sec}

Below on the left is the standard cubic ``map'' on
\(RP^2\)\index{projective plane!map needing six colors} which has
six pentagonal faces each of which shares an edge with all the
others.  Each point on the outer decagon is to be identified with
its antipode.  The graph is the Petersen graph\index{Petersen graph}
(Figure 4-2 on Page 102 of \Saaty).
\mymargin{ProjPlane}
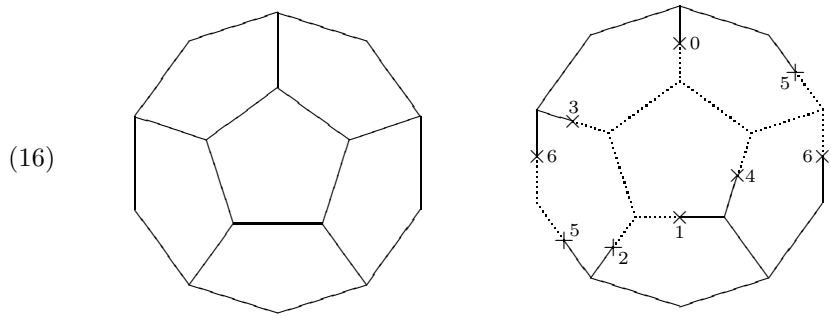
\begin{equation}\label{ProjPlane}
\begin{split}
\xy
(0,-20); (11.8,-16.2)**@{-}; (19,-6.2)**@{-};
(19,6.2)**@{-};  (11.8,16.2)**@{-}; (0,20)**@{-};
(-11.8,16.2)**@{-}; (-19,6.2)**@{-}; 
(-19,-6.2)**@{-}; (-11.8,-16.2)**@{-}; (0,-20)**@{-};;
(5.9,-8.1); (9.5,3.1)**@{-}; (0,10)**@{-};
(-9.5,3.1)**@{-}; (-5.9,-8.1)**@{-}; (5.9,-8.1)**@{-};
(5.9,-8.1); (11.8,-16.2)**@{-};
(9.5,3.1); (19,6.2)**@{-};
(0,10); (0,20)**@{-};
(-9.5,3.1); (-19,6.2)**@{-}; 
(-5.9,-8.1); (-11.8,-16.2)**@{-};
\endxy
\qquad
\qquad
\xy
(0,-20); (11.8,-16.2)**@{-}; (19,-6.2)**@{-};
(19,0)**@{-}; (19,6.2)**@{.};  (15.4,11.2)**@{.};
(11.8,16.2)**@{-}; (0,20)**@{-};
(-11.8,16.2)**@{-}; (-19,6.2)**@{-}; (-19,0)**@{-};
(-19,-6.2)**@{.};   (-15.4,-11.2)**@{.};
(-11.8,-16.2)**@{-}; (0,-20)**@{-};;
(5.9,-8.1); (7.7,-2.5)**@{-}; (9.5,3.1)**@{.}; (0,10)**@{.};
(-9.5,3.1)**@{.}; (-5.9,-8.1)**@{.}; (0,-8.1)**@{.};
(5.9,-8.1)**@{-}; 
(5.9,-8.1); (11.8,-16.2)**@{-};
(9.5,3.1); (19,6.2)**@{.};
(0,10); (0,15)**@{.}; (0,20)**@{-};
(-9.5,3.1); (-14.25,4.65)**@{.}; (-19,6.2)**@{-}; 
(-5.9,-8.1); (-8.85,-12.15)**@{.}; (-11.8,-16.2)**@{-};
(0,15)*{\times}; (2,15)*{\scriptstyle{0}};
(-14.25,4.65)*{\times}; (-14.1,6.5)*{\scriptstyle{3}};
(-8.85,-12.15)*{+}; (-7.5,-13.5)*{\scriptstyle{2}};
(-19,0)*{\times}; (-17,0)*{\scriptstyle{6}};
(19,0)*{\times}; (17,0)*{\scriptstyle{6}};
(15.4,11.2)*{+}; (14,9.8)*{\scriptstyle{5}};
(-15.4,-11.2)*{+}; (-14,-9.8)*{\scriptstyle{5}};
(0,-8.1)*{\times}; (0,-10.1)*{\scriptstyle{1}};
(7.7,-2.5)*{\times}; (9.4,-2.7)*{\scriptstyle{4}};
\endxy
\end{split}
\end{equation}

On the right, the map has been cut in 7 places, numbered 0 through
6, to break the map into two trees that meet in their leaves.  It is
easy to check that the two trees and the numbering of the leaves is
an embedding of the pair of trees in \tref{NoColorV} so that the
left tree in \tref{NoColorV} maps to the tree in \tref{ProjPlane}
drawn in dotted lines and the right tree in \tref{NoColorV} maps to
the tree in \tref{ProjPlane} drawn in solid lines.  From the facts
shown about \tref{NoColorV}, the map on the left in \tref{ProjPlane}
has no proper, edge 3-coloring.  This fact is standard.  See the
discussion on Page 102 of \Saaty.

The map in \tref{ProjPlane} is not ``Whitney'' in several senses.
While the map does break into a pair of trees, the break is not
induced by a Hamiltonian circuit in the dual graph.  The dual graph
\(K_6\) has plenty of Hamiltonian circuits, but they are all of
length 6 and cannot break the graph in \tref{ProjPlane} into a pair
of trees since it would only make 6 cuts.  This would result in a
loop in one of the two pieces.  The dual graph also does not satisfy
the hypotheses of Whitney's theorem.  The graph \(K_6\) has 20
triangles, but in the embedding into \(RP^2\), only 10 of them bound
faces.  Similar comments apply to the map \tref{SevenColEx} on the
torus.

\section{Enumeration}\mylabel{EnumSec}

In this section we record various results and questions about
counts.  Trees are counted by the Catalan numbers and we will
mention those first.  While the Catalan numbers have exponential
growth, there is no exact formula for them specifically in terms of
powers.  The other counts we come across have such formulas.  While
it is not necessary to do so, it is often pleasant to arrive at
these other formulas recursively.  Thus we will also discuss well
known generalizations of the Fibonacci numbers that we will
encounter.

\subsection{Trees}\mylabel{TreeCountSec}

As mentioned in Section \ref{AssocSec}, the number of trees with
\(n\) internal vertices is 
\[
C(n) = \frac{1}{n+1}\binom{2n}{n} = \frac{(2n)!}{n!(n+1)!}
\]
the \(n\)-th Catalan number\index{sequence!Catalan
numbers}\index{Catalan numbers}.  We have
\[
\begin{split}
\frac{C(n+1)}{C(n)}
&=
\frac{(2n+2)!}{(n+1)!(n+2)!}
\frac{n!(n+1)!}{(2n)!} \\
&=
\frac{(2n+2)(2n+1)}{(n+2)(n+1)} \\
&= 4-\frac{6}{n+2}.
\end{split}
\]
Thus the Catalan numbers grow about as fast as \(4^n\) after a
somewhat slower start.  Thus we can take it that the Catalan numbers
will overtake any function growing no faster than \(k^n\) for \(k<4\).

\subsection{Colorings of a single tree}

If \(T\) is a tree with \(n\) internal vertices, then it has \(2^n\)
sign assignments.  We count colorings modulo the action of \(S_3\)
on the colors, so we assume that the root color is 1 and that the
top internal vertex is positive.  Thus we get that the number of
colorings of \(T\) modulo the action of \(S_3\) on the colors is
\(2^{n-1}\).  This is equivalent to the formula for \(\delta_n\) at
the top of Page 212 in \cite{MR1550643} where the action of \(S_3\)
is not taken into account.

Note that \(2^{n-1}\) is smaller than the total number vectors of
length \(n+1\) with values in \(\{1,2,3\}\) ignoring validity.  This
latter number is \(3^{n+1}\) if the action of \(S_3\) is ignored, or
approximately \(\frac12 (3^n)\) modulo the action of \(S_3\).  We say
``approximately'' since not every vector uses all three colors.

Note also that these numbers are smaller than \(4^n\).

\subsection{Recursive sequences}\mylabel{RecurSeqSec}

Given a quintuple of numbers \((p,q,k,a,b)\) we can define a
sequence recursively by \(t(0)=a\), \(t(1)=b\), and
\(t(n+1)=pt(n)+qt(n-1)+k\) for \(n\ge 2\).  When the quintuple is
\((1,1,0,0,1)\), we have the Fibonacci
sequence\index{sequence!Fibonacci numbers}\index{Fibonacci numbers}.
For any quintuple with \(k=0\), it can be shown that the sequence is
a linear combination of two special sequences, known as Lucas
sequences, with quintuples \((p,q,0,0,1)\) and \((p,q,0,2,p)\).  The
numbers \((1,2,0,0,1)\) define the Jacobsthal numbers \(J(n)\) which
we will encounter below.  The Jacobsthal
numbers\index{sequence!Jacobsthal numbers}\index{Jacobsthal numbers}
are sequence A001045 of the Online Encyclopedia of Integer Sequences
(OEIS).  We will also encounter the sequence given by
\((2,3,0,0,1)\) which is sequence A015518 of the OEIS.

We include the possibility of non-zero values for \(k\) because of
the following which is easily proven by induction.

\begin{lemma}\mylabel{LucasSumLem} Let \(t(n)\) be determined by
\((p,q,0,0,b)\), and let 
\[
s(n) = \sum_{i=0}^n t(i)
\]
be the sequence of partial sums of the \(t(i)\).  Then the sequence
\(s(n)\) is determined by the quintuple \((p,q,b,0,b)\).
\end{lemma}

In particular the partial sums \(g(n)\) of the Fibonacci numbers
satisfy \(g(0)=0\), \(g(1)=1\), and \(g(n+1) = g(n)+g(n-1)+1\) for
\(n\ge 2\).  Writing out the first few of these numbers shows that
they are closely related to the Fibonacci numbers themselves.  This
is not surprising since this fact is true of exponential sequences
and the Fibonacci numbers are linear combination of two exponential
sequences.

In general the behavior of the sequence corresponding to
\((p,q,0,a,b)\) is determined by the eigenvalues \(\lambda_1\) and
\(\lambda_2\) of the matrix 
\(
\begin{bmatrix} p&q \\ 1&0 \end{bmatrix}
\).  
The terms of the sequence are then linear combinations of like
powers of \(\lambda_1\) and \(\lambda_2\).

The eigenvalues are roots of \(\lambda(\lambda-p)-q=
\lambda^2-p\lambda-q\).  So if \(q\) is already in the form
\(q=m(m-p)\), then the eigenvalues are \(\lambda_1=m\) and
\(\lambda_2=p-m\).  Thus we get particularly nice behavior from the
sequences if \(m=\pm1\).  For \(m=-1\), we get \(q=p+1\) and for
\(m=1\), we get \(q=1-p\).  

In the cases we said we would encounter, we have \(q=p+1\) and the
following is easy to prove by induction.

\begin{lemma}\mylabel{LucasSpecialLem} Let \(t(n)\) be determined by
\((p,p+1,0,0,1)\) and let \(s(n)\) be the sequence of partial sums
of the \(t(i)\).  Then 
\[
\begin{split}
(p+2)t(n) & = (p+1)^n-(-1)^n, \\
ps(n) &= t(n+1) - \frac12(1+(-1)^n).
\end{split}
\]
for all \(n\ge0\).  \end{lemma}

Thus we see that the values of \(t(n)\) and the partial sums of the
\(t(i)\) are dominated by powers of \(p+1\).

For other values of \(p\) and \(q\), the linear combinations of
powers of the eigenvalues are more complicated.

It is because of Lemma \ref{LucasSpecialLem} that we say that the
recursive formulas are not necessary.  However, the recursive
formulas are a pleasant way to discover some of the enumeration
formulas.

\subsection{Acceptable colorings}

Proposition \ref{ValidCharProp} characterizes acceptable colorings
of length at least two as those that are not constant and do not sum
to zero.  Every non-constant color vector of length \(n+1\) is of
the form \(1^i2[123]^{n-i}\) if we mod out by the action of \(S^3\)
on the colors.  We compute \(c(n)\)\index{\protect\(c(n)\protect\)},
the number of these that do not sum to zero.

We have \(c(0)=0\) and \(c(1)=1\).

We calculate \(c(n+1)\) by breaking the vectors of length \(n+2\)
into three classes.

({\bfseries I}) There is one vector of the form \(1^{i+1}2\).

({\bfseries II}) There are \(3c(n)\) vectors of the form
\((1^i2[123]^{n-i})x\), \(x\in \{1,2,3\}\), where the sum of
\((1^i2[123]^{n-i})\) is not zero.  But only two values of \(x\)
give the whole vector a non-zero sum, so this contributes \(2c(n)\)
to \(c(n+1)\).

({\bfseries III}) This leaves vectors of the form
\((1^i2[123]^{n-i})x\) where the sum of \((1^i2[123]^{n-i})\) is
zero.  But \((1^i2[123]^{n-i})\) has the form \(\mathbf py\) where
\(\mathbf p\) must have sum \(y\).  Now \((1^i2[123]^{n-i})\) cannot
have the form \(1^n2\) since \(1^n2\) does not have sum zero.  Thus
\(\mathbf p\) has the form \(\mathbf p = (1^i2[123]^{n-i-1})\) and
there are \(c(n-1)\) of these.  Since we assume \(\mathbf py\) has
sum zero, there is only one choice for \(y\) given \(\mathbf p\).
Since the sum of \(\mathbf py\) is zero, all three values of \(x\)
will give \(\mathbf pyx\) a non-zero sum.  So this contributes
\(3c(n-1)\) to \(c(n+1)\).

We have shown the following.

\begin{prop}\mylabel{ColorCountProp} Modulo the action of \(S_3\) on
the colors, there are \(c(n)\) acceptable color vectors for the set of
trees with \(n\) carets, where \(c(0)=0\), \(c(1)=1\), and \(c(n+1)
= 2c(n)+3c(n-1)+1\).  \end{prop}

From Lemma \ref{LucasSumLem}, this is the sequence of partial sums
of the sequence determined by \((2,3,0,0,1)\).

\subsection{Rigid colorings}\mylabel{RigidCountSec}

To count rigid colorings, we count the color vectors that are
described in Proposition \ref{RigidCharLem}.  These are (modulo
permutations of the colors) the non-constant vectors that sum to 1
with no prefix that sums to 3.

\begin{prop}\mylabel{RigidCountLem} Let
\(r(n)\)\index{\protect\(r(n)\protect\)} be the number (modulo the
action of \(S_3\) on the colors) of positive rigid color vectors
valid for trees with \(n\) carets.  Then \(r(1)=1\), \(r(2)=2\) and
\(r(n+1) = r(n)+2r(n-1)+1\).  \end{prop}

\begin{proof}  The cases \(n=1\) and \(n=2\) are done by direct
checking.

Consider the set \(S\) of vectors \(\mathbf c\) that are color
vectors that each positively and rigidly color some tree with
\(n+1\) carets with \(n\ge2\).

We break our set into three disjoint sets.  The first consists of
all vectors in \(S\) of the form \(2s\), the second consists of all
vectors in \(S\) of the form \(13s\) and the third consists of all
vectors in \(S\) of the form \(11s\).  In these, \(s\) is the
remainder of the color vector.  In the first two cases, the sum of
\(s\) is 3, and in the third case the sum of \(s\) is 1.

We first assume that \(s\) is a constant sequence.  If we have the
form \(11s\), the sum of the full vector is 0 if \(s\) has even
length, so it must have odd length.  But then the sum is 1 and \(s\)
is the constant vector 1 making the full vector constant.  So the
form \(11s\) cannot have \(s\) constant.

In the other two cases, the sum of \(s\) is 3 and must be the
constant sequence 3 of odd length.  We get one valid, positive,
rigid vector with \(s\) constant from the form \(2s\) if the length
of the full vector is even, and we get one valid, positive rigid
vector with \(s\) constant from the form \(13s\) if the length of
the full vector is odd.  This contributes 1 to the number of valid
vectors no matter what the parity of the length of the full vector
is.

From now on we assume that \(s\) is not constant.

In form \(2s\), we know that \(s\) is non-constant, sums to \(3\)
and cannot have a prefix summing to 1.  This is the criterion for
Proposition \ref{RigidCharLem} if we interchange the roles of 1 and
3.  Thus the number of such \(s\) is equal to \(r(n)\).

In form \(13s\), we know that \(s\) is non-constant, sums to \(3\)
and cannot have a prefix summing to 1.  Once again, we refer to
Proposition \ref{RigidCharLem} with the roles of 1 and 3
interchanged and we have that the number of such \(s\) is
\(r(n-1)\).

In form \(11s\), we know that \(s\) is non-constant, sums to \(1\)
and cannot have a prefix summing to 3.  From Proposition
\ref{RigidCharLem}, the number of such \(s\) is \(r(n-1)\).

Thus we have \(r(n+1)=r(n)+2r(n-1)+1\).
\end{proof}

From Lemma \ref{LucasSumLem}, this is the sequence of partial sums
of the Jacobsthal numbers.  Note that in this case, the value of
\(p\) in Lemma \ref{LucasSpecialLem} is 1.  So the values of the
\(r(n)\) are never far from the values of the Jacobsthal numbers.

\subsection{Flexible colorings}

Modulo the action of \(S_3\) on the colors, the number of acceptable
flexible colorings for trees of \(n\) internal vertices
\(f(n)\)\index{\protect\(f(n)\protect\)} satisfies \(c(n) =
f(n)+r(n)\).  We have the following.

\begin{lemma}  The count \(f(n)\) satisfies \(f(1)=0\), \(f(2)=1\)
and \(f(n+1) = 3f(n)+r(n)\).  \end{lemma}

\begin{proof} The known values of \(c(n)\) and \(r(n)\) verify the
base cases and the formula for \(f(n+1)\) for \(n=1\) and \(n=2\).

From Propositions \ref{ColorCountProp} and
\ref{RigidCountLem}, we have 
\[
\begin{split}
f(n+1)
&=
c(n+1)-r(n+1) \\
&=
2c(n)+3c(n-1)-(r(n)+2r(n-1)) \\
&=
2(c(n)-r(n))+r(n)+3(c(n-1)-r(n-1))+r(n-1) \\
&=
2f(n)+r(n)+3f(n-1)+r(n-1) \\
&=
2f(n)+r(n)+f(n) \\
&=
3f(n)+r(n).
\end{split}
\]

\end{proof}

\subsection{Highly colorable maps}

These last sections have many questions and few results.

We found much about map colorings by writing computer programs to
look for colorings.  It was natural to count the colorings as they
were found.  We discovered that certain maps had many more colorings
than all other maps.

Let \(n\) be at least 2.  We define \(m_1(n)\), \(m_2(n)\),
\(m_3(n)\) and \(m_4(n)\)\index{\protect\(m_i(n)\protect\)} as those
positive integers that are the largest possible with
\(m_4(n)<m_3(n)<m_2(n)< m_1(n)\) so that for each \(i\in
\{1,2,3,4\}\) there is a tree pair \((D_i,R_i)\) with \(n\) internal
vertices each, corresponding to a map in \(\mathfrak W\), and having
\(m_i(n)\) different valid color vectors (modulo the action of
\(S_3\) on the colors).  Stated differently, \(m_1(n)\) is the
largest number of colorings of maps in \(\mathfrak W\) of that size
(\(n+2\) faces), \(m_2(n)\) is the second largest, and so forth.

We can ask what the \(m_i(n)\) are.  We are unable to say what they
are, but we can say what we suspect they are.  We can also ask what
maps achieve the numbers that we think the \(m_i(n)\) are.  We are
also unable to say what these maps are, but we can give some
candidates.  (The value \(m_1(n)\) and corresponding map are now
known.  See the note after Question \ref{BiwheelQuestion}.)

\subsubsection{The estimates}\mylabel{EstimatePara}

We give our observations.  Counts depend heavily on the parity of
\(n\).  For this reason we use the notation \(f(n)+(e,o)\) to denote
\(f(n)+e\) when \(n\) is even and \(f(n)+o\) when \(n\) is odd.  The
Jacobsthal numbers \(J(n)\) are as introduced in Section
\ref{RecurSeqSec}.  From Lemma \ref{LucasSpecialLem}, we have
\[
J(n) = \frac13(2^n+(-1)^{n+1}).
\]

Calculations support the following.

\begin{conj}[Color count]\mylabel{ColCountConj} The following
hold.\index{conjecture!Color count}
\mymargin{TheEstimatesI-IV}\begin{alignat}{3}
\label{TheEstimatesI}
m_1(n) &= J(n-3) &+(1,0) &, &\quad n&\ge5, \\
\label{TheEstimatesII}
m_2(n) &= J(n-4) &+(7,5) &=m_1(n-1)+(7,4), &\quad n&\ge7, \\
\label{TheEstimatesIII}
m_3(n) &= J(n-4) &+(1,0) &=m_1(n-1), &\quad n&\ge7, \\
\label{TheEstimatesIV}
m_4(n) &= J(n-4) &-(1,2) &=m_1(n-1)-(1,3), &\quad n&\ge7.
\end{alignat}
\end{conj}

Given the rapid growth of \(J(n)\), it is seen that there is a huge
gap in the number of colorings between the largest and second
largest number of colorings.

\subsubsection{The maps} It is easier to discuss and calculate
numbers of vertex colorings of maps than face colorings because of
the availability of the chromatic polynomial.  The duals to the maps
in \(\mathfrak W\) are triangulations of the 2-sphere.  We are thus
interested in ways to describe and refer to triangulations of the
2-sphere.  We describe some terminology and notation.

We let \(C_n\) be the polygon of \(n\) vertices and edges.  If
\(\Gamma\) and \(G\) are graphs, then \(\Gamma G\) is the join of
\(\Gamma\) and \(G\).   This is the graph whose vertex set is
\(V(\Gamma)\cup V(G)\) and whose edge set is \(E(\Gamma)\cup
E(G)\cup T\) where \(T\) has an edge joining \(v\) to \(w\) for
every \((v,w)\) in \(V(\Gamma)\times V(G)\).  Then \(C(\Gamma)\) is
the ``cone on \(\Gamma\)'' and is the join of \(\Gamma\) and the one
point graph with no edges.  We then get \(\Sigma(\Gamma)\), the
``suspension of \(\Gamma\),'' which is the join of \(\Gamma\) and
the two point graph with no edges.

The graph we are most interested in is \(W_{n+2} = \Sigma(C_n)\) or
the biwheel\index{biwheel} of \(n+2\) vertices.  This graph can be
embedded in the 2-sphere to give a triangulation of the 2-sphere.
The dual of \(W_{10}\) is shown below.
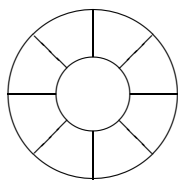
\[
\xy
(0,0)*{\cir<15pt>{}};
(0,0)*{\cir<40pt>{}};
(3.55,3.55); (7.8, 7.8)**@{-};
(-3.55,3.55); (-7.8, 7.8)**@{-};
(3.55,-3.55); (7.8, -7.8)**@{-};
(-3.55,-3.55); (-7.8, -7.8)**@{-};
(0,5); (0,11.3)**@{-};
(5,0); (11.3,0)**@{-};
(0,-5); (0,-11.3)**@{-};
(-5,0); (-11.3,0)**@{-};
\endxy
\]

We next describe four variations of \(W_n\).  The first is the
one-bar variation \(\Theta_{n+1}\) of \(W_n\).  Its name derives from the
appearance of its dual.  The dual of \(\Theta_{11}\) is below.
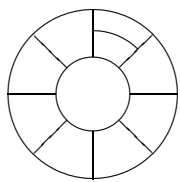
\[
\xy
(0,0)*{\cir<15pt>{}};
(0,0)*{\cir<40pt>{}};
(0,0)*{\cir<25pt>{ul^l}};
(3.55,3.55); (7.8, 7.8)**@{-};
(-3.55,3.55); (-7.8, 7.8)**@{-};
(3.55,-3.55); (7.8, -7.8)**@{-};
(-3.55,-3.55); (-7.8, -7.8)**@{-};
(0,5); (0,11.3)**@{-};
(5,0); (11.3,0)**@{-};
(0,-5); (0,-11.3)**@{-};
(-5,0); (-11.3,0)**@{-};
\endxy
\]

The difference between the graph \(\Theta_{n+1}\) and \(W_n\) is
explained by the transition pictured below.
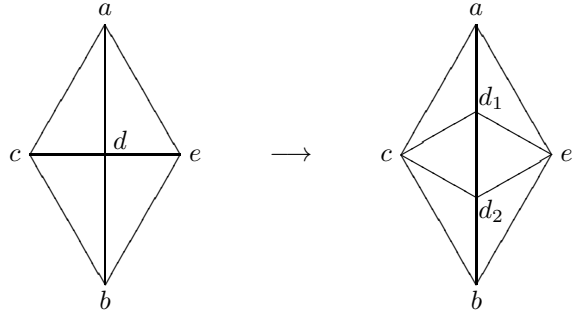
\[
\xymatrix{
{\xy
(0,0); (20,0)**@{-}; (10,17.32)**@{-}; (0,0)**@{-};
(10,-17.32)**@{-}; (20,0)**@{-};
(10,-17.32); (10,17.32)**@{-};
(10,19.32)*{a};
(10,-19.32)*{b};
(-2,0)*{c}; (22,0)*{e};
(12,2)*{d};
\endxy}
}
\qquad\longrightarrow\qquad
\xymatrix{
{\xy
(20,0); (10,17.32)**@{-}; (0,0)**@{-};
(10,5.77)**@{-}; (20,0)**@{-};
(10,5.77); (10,17.32)**@{-};
(20,0); (10,-17.32)**@{-}; (0,0)**@{-};
(10,-5.77)**@{-}; (20,0)**@{-};
(10,-5.77); (10,-17.32)**@{-};
(10,-5.77); (10,5.77)**@{-};
(10,19.32)*{a};
(10,-19.32)*{b};
(-2,0)*{c}; (22,0)*{e};
(12,7.5)*{d_1};
(12,-7.5)*{d_2};
\endxy}
}
\]

The left figure shows four triangles in in \(W_n\).  The points
\(a\) and \(b\) are the ``suspension points'' and the points \(c\),
\(d\) and \(e\) are three points in the polygon \(C_{n-2}\).  The
graph \(\Theta_{n+1}\) is created by replacing the left figure in \(W_n\)
by the right figure and leaving the rest of \(W_n\) the same.  Thus
the transition raises the number of vertices by one and the number
of triangles by two.

The second is the two-bar variation \(\Xi_{n+2}\) of \(W_n\).  The
dual of \(\Xi_{12}\) is shown below.
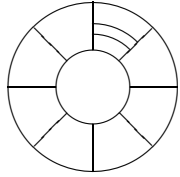
\[
\xy
(0,0)*{\cir<15pt>{}};
(0,0)*{\cir<40pt>{}};
(0,0)*{\cir<23pt>{ul^l}};
(0,0)*{\cir<27pt>{ul^l}};
(3.55,3.55); (7.8, 7.8)**@{-};
(-3.55,3.55); (-7.8, 7.8)**@{-};
(3.55,-3.55); (7.8, -7.8)**@{-};
(-3.55,-3.55); (-7.8, -7.8)**@{-};
(0,5); (0,11.3)**@{-};
(5,0); (11.3,0)**@{-};
(0,-5); (0,-11.3)**@{-};
(-5,0); (-11.3,0)**@{-};
\endxy
\]

The graph \(\Xi_{n+2}\) is obtained by a replacement similar to that
used to obtain \(\Theta_{n+1}\).  The figure below left in \(W_n\) is
replaced by the figure below right.
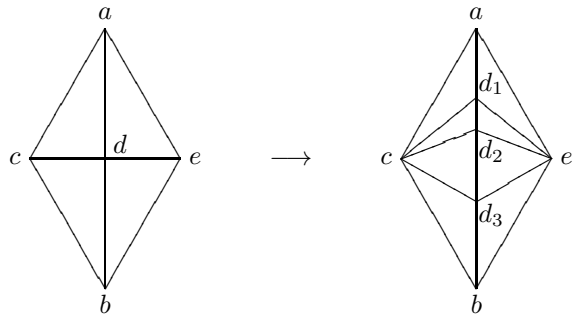
\[
\xymatrix{
{\xy
(0,0); (20,0)**@{-}; (10,17.32)**@{-}; (0,0)**@{-};
(10,-17.32)**@{-}; (20,0)**@{-};
(10,-17.32); (10,17.32)**@{-};
(10,19.32)*{a};
(10,-19.32)*{b};
(-2,0)*{c}; (22,0)*{e};
(12,2)*{d};
\endxy}
}
\qquad\longrightarrow\qquad
\xymatrix{
{\xy
(20,0); (10,17.32)**@{-}; (0,0)**@{-};
(10,3.85)**@{-}; (20,0)**@{-}; (10, 8.06)**@{-};
(10,17.32)**@{-}; (10,3.85);
(10, 8.06)**@{-}; (0,0)**@{-};
(20,0); (10,-17.32)**@{-}; (0,0)**@{-};
(10,-5.77)**@{-}; (20,0)**@{-};
(10,-5.77); (10,-17.32)**@{-};
(10,-5.77); (10,5.77)**@{-};
(10,19.32)*{a};
(10,-19.32)*{b};
(-2,0)*{c}; (22,0)*{e};
(12,10)*{d_1};
(12,1)*{d_2};
(12,-7.5)*{d_3};
\endxy}
}
\]

The next two variations on \(W_n\) are best described without
reference to the original dual map.  Each is obtained by adding new
edges to a single triangular face in \(W_n\).  The graph
\(Y_{n+1}\) is obtained by the following modification.
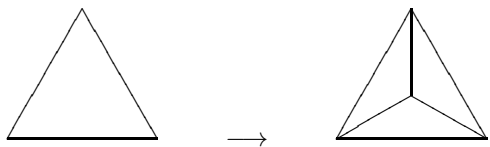
\[
\xymatrix{
{\xy
(0,0); (20,0)**@{-}; (10,17.32)**@{-}; (0,0)**@{-};
\endxy}
}
\qquad\longrightarrow\qquad
\xymatrix{
{\xy
(0,0); (20,0)**@{-}; (10,17.32)**@{-}; (0,0)**@{-};
(10,5.77)**@{-}; (20,0)**@{-};
(10,5.77); (10,17.32)**@{-};
\endxy}
}
\]

The graph \(\nabla_{n+3}\) is obtained by the following
modification.
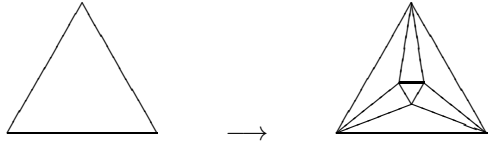
\[
\xymatrix{
{\xy
(0,0); (20,0)**@{-}; (10,17.32)**@{-}; (0,0)**@{-};
\endxy}
}
\qquad\longrightarrow\qquad
\xymatrix{
{\xy
(0,0); (20,0)**@{-}; (10,17.32)**@{-}; (0,0)**@{-};
(10,3.85)**@{-}; (20,0)**@{-}; (11.66, 6.72)**@{-};
(10,17.32)**@{-}; (8.34,6.72)**@{-}; (0,0)**@{-};
(10,3.85); (11.66, 6.72)**@{-}; (8.34,6.72)**@{-};
(10,3.85)**@{-};
\endxy}
}
\]

\subsubsection{The counts}

We are interested in the number of proper, vertex 4-colorings of
\(W_n\), \(\Theta_n\), \(\Xi_n\), \(Y_n\) and \(\nabla_n\).
Available to us is the chromatic polynomial \(P(\Gamma,
t)\)\index{chromatic polynomial} of a finite graph \(\Gamma\) and
the standard techniques for calculating the polynomial.  We are thus
interested in \(P(\Gamma,4)\) for \(\Gamma\) one of the five graphs
listed.  Note that we have been counting colorings modulo the
permutations of the colors while \(P(\Gamma, t)\) considers
colorings different even if one is obtained from the other by
permuting the colors.  This is not a problem since none of our
graphs has a proper, 2-coloring.  It follows from this that the
number we want is 1/24 of the number that the chromatic polynomial
gives.

We state certain results without giving details.  The calculations
using chromatic polynomials are straightforward.  They are a bit
longer in the case of \(\nabla_n\) and \(\Xi_n\), but there are no
surprises.  We have
\mymargin{TheCountsIandV}\begin{align}
\label{TheCountsI}\frac1{24}P(W_n,4) 
&= 
\frac13(2^{n-3}+(-1)^n) + \frac12(1+(-1)^n), \\
\frac1{24}P(\Theta_n,4) 
&= 
\frac13(2^{n-5}+(-1)^n) + \frac12(4+4(-1)^{n-1}), \\
\frac1{24}P(\Xi_n,4) 
&= 
\frac13(2^{n-4}+(-1)^{n-1}) + \frac12(5+9(-1)^n), \\
\frac1{24}P(Y_n,4) 
&= 
\frac1{24}P(W_{n-1},4), \\
\label{TheCountsV}\frac1{24}P(\nabla_n,4) 
&= 
\frac13(2^{n-3}+(-1)^{n-1}) + \frac12(4+6(-1)^{n-1}).
\end{align}

The fourth equality is a triviality given the nature of the relation
between \(Y_n\) and \(W_{n-1}\).

It is a bit more revealing to compare each quantity on the left with
a corresponding quantity of some \(W_j\) as is done with \(Y_n\).
We do this below where, as in Paragraph \ref{EstimatePara}, we
replace expressions such as \(\frac12(a+b(-1)^n)\) with a pair of
numbers \((e,o)\) where \(e\) is the value of the expression when
\(n\) is even and \(o\) is the value when \(n\) is odd.  Certain
symmetries wreak havoc with comparisons for low values of \(n\), and
so restricting to \(n\ge7\) gives
\[
\begin{split}
\frac1{24}P(\Theta_n,4)
&=
\frac1{24}P(W_{n-2},4) + (-1,4), \\
\frac1{24}P(\Xi_n,4)
&=
\frac1{24}P(W_{n-1},4) + (7,-3), \\
\frac1{24}P(Y_n,4)
&=
\frac1{24}P(W_{n-1},4) + (0,0), \\
\frac1{24}P(\nabla_n,4)
&=
\frac1{24}P(W_{n-1},4) + (-1,4).
\end{split}
\]

Given the exponential growth of \(P(W_n,4)\), it is seen that the
number of colorings of the \(\Theta_n\) are much smaller than the
others.  Comparing the above with
\tref{TheEstimatesI}--\tref{TheEstimatesIV}, we see that \(\Xi_n\)
gives what we guess to be the second largest value when \(n\) is
even and the fourth largest value when \(n\) is odd.  The same role
with the parity of \(n\) reversed is played by \(\nabla_n\).  The
graph \(Y_n\) always gives our guess at the third largest value.

The exponential growth of \(P(W_n,4)\) also shows that if our
guesses at the \(m_i(n)\) are correct, then there is a huge gap
between largest number \(m_1(n)\) of colorings of an \(n\)-vertex
triangulation of the 2-sphere and the next number of colorings
encountered which cluster around approximately \(\frac12 m_1(n)\).
Computer calculations hint that there might be other gaps among the
number of colorings that grow with \(n\), but they are no where near
as large.  For \(n=12\), the difference of \(m_1(12)\) and
\(m_2(12)\) is 80 while the next largest gap is 12.

\subsubsection{Questions}

The formulas for the \(m_i(n)\) in
\tref{TheEstimatesI}--\tref{TheEstimatesIV} are guesses and have
been verified by us for values of \(n\) through \(n=12\) and by
others through \(n=15\).  The counts of the number of colorings of
the graphs \(W_n\), \(\Theta_n\), \(\Xi_n\), \(Y_n\) and \(\nabla_n\) in
\tref{TheCountsI}--\tref{TheCountsV} have been formally derived.

\begin{question}\mylabel{MiQuestion}
Are the values for the \(m_i(n)\), \(1\le i\le 4\), given in 
\tref{TheEstimatesI}--\tref{TheEstimatesIV} correct?
\end{question}

Computational evidence suggests that the biwheel has the most
proper, vertex 4-colorings of any \(n\)-vertex triangulation of the
2-sphere.  Through \(n=12\), no other triangulations have as many
colorings.

\begin{question}\mylabel{BiwheelQuestion}
Is \(W_n\) the only triangulation of the 2-sphere with
\(J(n-3)+\frac12(1+(-1)^n)\) colorings?
\end{question}

{\itshape Note:} Question \ref{MiQuestion} for \(i=1\) and Question
\ref{BiwheelQuestion} have been answered in the affirmative by Paul
Seymour \cite{pdseymour:biwheel}.

We have no evidence one way or another that the triangulations
\(\Xi_n\), \(Y_n\) and \(\nabla_n\) are the only triangulations
having the numbers of proper, vertex 4-colorings that they do.  But
we can still ask the question.

\begin{question}
Are there \(n\)-vertex triangulations of the 2-sphere not isomorphic
to \(\Xi_n\), \(Y_n\) or \(\nabla_n\) with \(m_i(n)\) colorings for
\(2\le i\le 4\)?
\end{question}

\newcommand{\eye}
{\xy(-1.7,3.5);(1.7,2.5)**@{-};(0,1)*{\textstyle{\mathrm{mu}}};\endxy
}

\subsection{Zero sets and color graphs}

This section also raises more questions than it answers.

If \(\mathbf c\) is a color vector of length \(n+1\), then it has a
zero set and complementary color graph on \(A_d\) with \(d=n-1\).
Thus small zero sets should go with large color graphs and vice
versa.  We can measure zero sets directly by the number of vertices
of \(A_d\) that are in it and indirectly by the number of
codimension-1 faces that are in it.  We have few examples from
calculations, but we have observed that diameter as well as size
enters into this discussion.  The color graphs of highest diameter
we have seen are the color graphs of Section \ref{LongPathSec}.  It
turns out that that the complementary zero sets of these color
graphs are the smallest we have seen measured in terms of
codimension-1 faces.  On the other hand, color
graphs of smallest diameter have complementary zero sets that have
the largest number of codimension-1 faces.  We now describe some
specifics.

We get two numbers to consider from Section \ref{ZeroSetSec}.  Let
\(n=d+1\).  If \(\mathbf c\) is a color vector of length \(n+1\),
then it has a zero set which is determined by certain set
\(Z_\mathbf c\) of intervals in \([0,n]\).  An interval \(J\) is in
\(Z_\mathbf c\) if the sum of the \(\mathbf c_i\) for \(i\in J\) is
zero.  Each such \(J\in Z_\mathbf c\) determines a codimension-1
face of \(A_d\) all of whose vertices are trees for which \(\mathbf
c\) is not valid.  We define \(u(n)\) to be the smallest integer and
\(l(n)\) to be the largest integer so that for any acceptable color
vector \(\mathbf c\) of length \(n+1\), we have \(l(n)\le |Z_\mathbf
c|\le u(n)\).

We claim that \(u(n) = \lfloor \frac{n^2}{4}\rfloor\), and we think
that \(l(n) = \lfloor \frac{n^2}{8}\rfloor\).  We indicate the
argument for \(u(n)\) and give support for \(l(n)\).

\begin{lemma}\mylabel{IKConstLem} If \(\mathbf k\) is a constant vector of
length \(n\), then \(|Z_\mathbf k| = \lfloor \frac{n^2}{4}
\rfloor\).  \end{lemma}

This follows from direct calculation and the fact that the intervals
in \(Z_\mathbf k\) are just the intervals of even length.  Of course
\(\mathbf k\) is not acceptable, but the fact is useful.

\begin{prop} With definitions as above, we have \(u(n)=\lfloor
\frac{n^2}{4}\rfloor\).  \end{prop}

\begin{proof} We only sketch the proof.  We take an acceptable
vector \(\mathbf c\) of length \(n+1\) and let \(\mathbf p\) and
\(\mathbf s\) be prefix and suffix of \(\mathbf c\) such that
\(\mathbf c=\mathbf p\mathbf s\) and so that \(\mathbf p\) is the
longest constant prefix.  We can assume that \(\mathbf p=1^j\) and
that \(\mathbf s\) starts with 2 and has length \(k+1\) with
\(j+k=n\).

Since \(\mathbf p\) and \(\mathbf s\) are adjacent, we are
interested in terminal intervals for \(\mathbf p\), those intervals
in \(\mathbf p\) that include the last position in \(\mathbf p\),
and initial intervals for \(\mathbf s\), those intervals in
\(\mathbf s\) that include the first position of \(\mathbf s\).

By Lemma \ref{IKConstLem} and an inductive hypothesis, we have
\[
|Z_\mathbf c| \le \left\lfloor\frac{j^2}{4}\right\rfloor 
+ \left\lfloor \frac{k^2}{4}\right\rfloor +
AB + CD
\]
where \(A\) is the number of terminal intervals in \(\mathbf p\)
summing to zero, \(B\) is the number of initial intervals in
\(\mathbf s\) summing to zero, \(C\) is the number of terminal
intervals in \(\mathbf p\) summing to one, and \(D\) is the number
of initial intervals \(\mathbf s\) summing to one.  By taking into
account the parities of \(j\) and \(k\), and the acceptability of
\(\mathbf c\), one shows on a case by case basis that \(AB+CD\) is
no more than \(\frac{2jk}{4}\).  \end{proof}

That \(l(n)=\lfloor \frac{n^2}{8}\rfloor\) is supported by computer
calculation for small values of \(n\) and by the following fact.  If
we use \(a\), \(b\) and \(c\) to represent the colors 1, 2, 3 in
some ordering, then in any interval of length 4 one of the patterns
\(aa\), \(abc\) or \(abab\) will occur.  Thus in any interval of
length 4 there is a subinterval of length at least two that sums to
zero.  

Lemma \ref{IKConstLem} immediately gives us examples of extremes.
Let \(\mathbf c=1^n2\), \(\mathbf d=1^k21^k\) and \(\mathbf
e=1^k21^{k+1}\).  Let \(n=2k\) for \(\mathbf d\) and let \(n=2k+1\)
for \(\mathbf e\).  It follows directly from Lemma \ref{IKConstLem}
that the following are true.
\[
|Z_\mathbf c| = \left\lfloor \frac{n^2}{4} \right\rfloor, \qquad 
|Z_\mathbf d| = \left\lfloor \frac{n^2}{8} \right\rfloor, \qquad
|Z_\mathbf e| = \left\lfloor \frac{n^2}{8} \right\rfloor.
\]

The relevance of these examples to this section is that \(\mathbf
c\) is valid for only one tree (the right vine), and \(\mathbf d\)
and \(\mathbf e\) are examples from Section \ref{LongPathSec} of
color vectors whose color graphs have very high diameter.

Other vectors that are valid for only one tree are of the form
\(1^j231^{n-j-1}\) whose trees are root shifts of the right vine,
and \(12^n\) whose tree is the left vine and also a root shift of
the right vine.  Thus all these have zero sets with \(u(n)\)
codimension-1 faces.

Another color vector that seems to realize \(l(n)\) is \(\mathbf f =
(12)^k\).  With \(n=2k-1\), preliminary calculations hint that
\(|Z_\mathbf f| = \lfloor \frac{n^2}{8}\rfloor\).  We have not
investigated the color graph of \(\mathbf f\) in detail, but (for
various \(k\)) they are intriguing.

We are left with the following questions.

\begin{question}  With \(n=d+1\), is \(\lfloor\frac{n^2}{2}\rfloor\)
the highest possible diameter of a color set in \(A_d\)?
\end{question}

\begin{question} Is \(l(n)=\lfloor \frac{n^2}{8}\rfloor\)?
\end{question}

\begin{question} Are there other color vectors realizing either
\(u(n)\) or \(l(n)\)?  \end{question}

We end with even vaguer questions.

\begin{question} What is the nature of the color graph of
\((12)^k\)?  \end{question}

\begin{question} What is the relationship between the shape of the
color graph and the number and arrangement of the codimension-1
faces of the zero set of a color vector?  \end{question}

For rigid coloring vectors (such as \(1^n2\)), we can ask the
following.

\begin{question} What is the relationship between the number of
vertices in the (totally disconnected) color graph of a rigid color
vector and the number of codimension-1 faces of its zero set?
\end{question}

We end this section with one minor comment.  The only vertices
completely interior to the color graph of a color vector \(\mathbf
c\) are those trees that are either completely positively signed by
\(\mathbf c\) or are completely negatively signed by \(\mathbf c\).
If we accept that such signings of a tree by a given color vector
are rare, then we have that most vertices in a color graph are on
its boundary.

\section{The end}\mylabel{EndSec}

This section represents the ill defined frontier of this paper.  It
gathers questions that fit nowhere else.

\begin{question} Given vertices \(u\) and \(v\) in a color graph
\(\Gamma\) in \(A_d\), is there a pair of codimension-1 faces in
\(A_d\) that contains a path in \(\Gamma\) between \(u\) and \(v\)?
Can the path and faces be chosen so that there is an intermediate
vertex \(w\) in the path with the path from \(u\) to \(w\) in one
face and the path from \(w\) to \(v\) in the other?  \end{question}

Recall from Section \ref{ZeroSetSec} that for a color vector
\(\mathbf c\) with \(n+1\) entries, the set \(Z_\mathbf c\) is the
set of intervals \([m,n]\) in \([0,n]\) for which the sum of the
\(\mathbf c_i\) with \(i\in [m,n]\) is zero.

\begin{question} Is there a way to characterize those sets of
intervals in \([0,n]\) that are of the form \(Z_\mathbf c\) for some
acceptable color vector \(\mathbf c\) with \(n+1\) entries?
\end{question}

\begin{question} Is there any significance to the inconsistent
elements of \(E\)?  \end{question}

It is not clear what is meant by this last question.  Perhaps there
are imaginary colorings of some tree pairs.

Another question is motivated by our construction of \(\Sigma(w)\).
Its relevance to the paper is highly questionable.  Let \(\Gamma\)
be a finite graph.  If the edges of \(\Gamma\) are numbered \(e_1\),
\(e_2\), \dots, \(e_n\) with \(n=|E(\Gamma)|\), then we turn
\(\Gamma\) into a signed graph much as we do for \(\Sigma(w)\).  We
let \(\Gamma_i\) be the subgraph of \(\Gamma\) having exactly the
edges \(\{e_k\mid 1\le k\le i\}\).  We say that edge \(e_{i+1}\)
with endpoints \(a_i\) and \(b_i\) is positive if the parities in
\(\Gamma_i\) of \(a_i\) and \(b_i\) agree, and is negative
otherwise.  Obviously, the signing of the edges depends on the
numbering of the edges.  We are interested in whether the signed
graph \(\Gamma\) is balanced.  The question we raise is as follows.

\begin{question} Which connected, finite graphs become balanced for
every numbering of their edges, which connected, finite graphs
become balanced for no numbering of their edges, and which
connected, finite graphs are not in the two classes just described?
\end{question}

Of course trees, having no interesting closed walks, are in the
first class.  Eliminating trees, a minor amount of experimentation
hints that in the first class are the boundaries of polygons with an
even number of edges, in the second class are the boundaries of
polygons with an odd number of edges, and in the third class is
everything else.


\providecommand{\bysame}{\leavevmode\hbox to3em{\hrulefill}\thinspace}
\providecommand{\MR}{\relax\ifhmode\unskip\space\fi MR }
\providecommand{\MRhref}[2]{%
  \href{http://www.ams.org/mathscinet-getitem?mr=#1}{#2}
}
\providecommand{\href}[2]{#2}

\noindent Department of Mathematics

\noindent State University of New York at Oneonta

\noindent Oneonta, NY 

\noindent USA

\noindent email: gsbowlin@gmail.com

\vspace{20pt}

\noindent Department of Mathematical Sciences

\noindent State University of New York at Binghamton

\noindent Binghamton, NY 13902-6000

\noindent USA

\noindent email: matt@math.binghamton.edu


\begin{theindex}

  \item \protect\(A_d\protect\), 17
  \item \protect\(E\protect\), 34
  \item \protect\(F\protect\), 27
  \item \protect\(F^+\protect\), 55
  \item \protect\(F_4\protect\), 40
  \item \protect\(P^+\protect\), 54
  \item \protect\(P^r\protect\), 54
  \item \protect\(P_u\protect\), 54
  \item \protect\(S\caret T\protect\), 14
  \item \protect\(T^\sigma\protect\), 22
  \item \protect\(T_v\protect\), 14
  \item \protect\(V\protect\), 57
  \item \protect\(Z_\mathbf c\protect\), 41
  \item \protect\(\Sigma(w)\protect\), 44
  \item \protect\(\overline{u}\protect\), 18
  \item \protect\(\protect\mathcal T\protect\), 13
  \item \protect\(\protect\mathfrak W\protect\), 3
  \item \protect\(\rot{u}\protect\), 18, 25
  \item \protect\(c(n)\protect\), 60
  \item \protect\(f(n)\protect\), 62
  \item \protect\(m_i(n)\protect\), 62
  \item \protect\(r(n)\protect\), 61
  \item 4-tree, 51

  \indexspace

  \item above
    \subitem for subtrees, 46
  \item acceptable
    \subitem color vector, 21
      \subsubitem characterization, 53
  \item address
    \subitem of vertex in tree, 14
  \item adjacent
    \subitem subtrees, 46
  \item ancestor
    \subitem in tree, 12
  \item associahedron, 16
    \subitem colored, 27
    \subitem diameter, 7
    \subitem dihedral group action, 20
    \subitem dimension, 16
    \subitem edge, 17
    \subitem face, 17

  \indexspace

  \item balanced
    \subitem signed graph, 44
  \item binary
    \subitem tree, 12
      \subsubitem as union of carets, 29
      \subsubitem standard model, 13
  \item binary trees
    \subitem union, intersection, difference, 30
  \item biwheel, 63

  \indexspace

  \item caret, 29
    \subitem center vertex, 29
    \subitem exposed, 29
  \item carets
    \subitem tree union of, 29
  \item Catalan number, 17
  \item Catalan numbers, 59
  \item center vertex
    \subitem caret, 29
  \item chain
    \subitem of increasing multiplications, 37
    \subitem of minimally increasing multiplications, 37
  \item children
    \subitem in tree, 12
  \item chromatic polynomial, 65
  \item code
    \subitem for colors, 20
  \item color
    \subitem code for, 20
    \subitem element of group, 20
    \subitem of leaf in tree, 20
    \subitem of vertex in tree, 20
    \subitem permutation, 23
      \subsubitem effect on sign assignment, 23
    \subitem root
      \subsubitem of tree, 20
    \subitem vector, 20
      \subsubitem acceptable, 21
      \subsubitem color graph of, 27
      \subsubitem valid for pair, 21
      \subsubitem valid for tree, 21
      \subsubitem zero set of, 41
  \item color graph
    \subitem of color vector, 27
  \item colored
    \subitem associahedron, 27
    \subitem vertices, 27
  \item coloring
    \subitem compatible with sign structure, 45
    \subitem normal, 23
    \subitem of element of \protect\(F\protect\), 35
    \subitem of tree from vector, 21
  \item Compatibility Lemma, 35
  \item conjecture
    \subitem Color count, 63
    \subitem signed path
      \subsubitem first, 26
      \subsubitem second, 46

  \indexspace

  \item deferment
    \subitem of element of \protect\(F\protect\), 55
  \item depth
    \subitem vertex in tree, 39
  \item descendant
    \subitem in tree, 12
  \item diameter
    \subitem of assciahedron, 7
  \item dihedral group
    \subitem action on associahedron, 20
  \item dimension
    \subitem associahedron, 16
      \subsubitem face, 17
  \item drawing
    \subitem of tree, 13
  \item dual
    \subitem polygon to tree, 15
    \subitem tree to polygon, 15
  \item dyadic
    \subitem rational, 28

  \indexspace

  \item edge
    \subitem associahedron, 17
    \subitem external in tree, 12
    \subitem internal in tree, 12
    \subitem leaf in tree, 12
    \subitem root in tree, 12
    \subitem rotation along, 18
  \item edge disjoint
    \subitem subtrees, 15
  \item edge-face
    \subitem coloring correspondence, 20
  \item exposed caret, 29

  \indexspace

  \item face
    \subitem of associahedron, 17
      \subsubitem dimension, 17
      \subsubitem structure, 17
  \item face-edge
    \subitem coloring correspondence, 20
  \item Fibonacci numbers, 59
  \item flexible
    \subitem sign assignment, 26

  \indexspace

  \item graph
    \subitem color
      \subsubitem of color vector, 27
  \item group
    \subitem of colors, 20
  \item group \protect\(E\protect\), 34

  \indexspace

  \item increasing
    \subitem chain of multiplications, 37
    \subitem multiplication, 37
  \item induction
    \subitem Noetherian, 14
    \subitem well founded, 14
  \item infix order
    \subitem on tree, 13

  \indexspace

  \item Jacobsthal numbers, 59

  \indexspace

  \item leaves
    \subitem left-right order, 13
    \subitem tree, 12
      \subsubitem color, 20
  \item left
    \subitem subtree, 14
    \subitem vine, 36
  \item left-right order
    \subitem on leaves, 13
  \item lemma
    \subitem Compatibility, 35
  \item level
    \subitem vertex in tree, 39
  \item local order
    \subitem from planar embedding, 12
    \subitem in tree, 12

  \indexspace

  \item map
    \subitem prime, 10
  \item minimally increasing
    \subitem chain of multiplications, 37
    \subitem multiplication, 37
  \item move
    \subitem pentagonal, 48
    \subitem square, 48
  \item multiplication
    \subitem increasing, 37
    \subitem minimally increasing, 37

  \indexspace

  \item negative
    \subitem normal coloring, 23
    \subitem rigid
      \subsubitem pattern, 54
  \item neighborhood
    \subitem positive
      \subsubitem of vertex, 55
  \item Noetherian
    \subitem induction, 14
  \item normal
    \subitem coloring, 23
      \subsubitem negative, 23
      \subsubitem positive, 23

  \indexspace

  \item order
    \subitem local in tree, 12
    \subitem well founded, 14

  \indexspace

  \item P-compatible
    \subitem pair of pairs, 54
    \subitem tree pair, 53
  \item pair
    \subitem of finite trees, 11
      \subsubitem P-compatible, 53
      \subsubitem parity condition, 39
      \subsubitem reduced, 29
      \subsubitem reduction, 29
    \subitem triangulated
      \subsubitem polygon, 11
  \item parent
    \subitem in tree, 12
  \item parity condition
    \subitem on pair of trees, 39
  \item path
    \subitem sign consistent, 46
    \subitem sign structure of, 46
    \subitem square equivalence, 50
    \subitem valid
      \subsubitem of rotations, 25
  \item pattern, 53
    \subitem all positive, 54
    \subitem negative rigid, 54
    \subitem positive rigid, 54
    \subitem shadow
      \subsubitem of tree, 41
  \item pentagonal
    \subitem move, 48
    \subitem relation, 34
  \item permuting
    \subitem colors, 23
  \item Petersen graph, 58
  \item pivot vertex
    \subitem of rotation, 18
  \item polygon
    \subitem dual to tree, 15
    \subitem triangulated
      \subsubitem pair, 11
  \item positive
    \subitem element of \protect\(F\protect\), 38
    \subitem monoid of \protect\(F\protect\), 38
    \subitem neighborhood
      \subsubitem of vertex, 55
    \subitem normal coloring, 23
    \subitem pattern, 54
    \subitem rigid
      \subsubitem pattern, 54
  \item prefix order
    \subitem on tree, 13
  \item prime
    \subitem finite tree pair, 11
    \subitem map, 10
  \item projection
    \subitem of tree, 15
  \item projective plane
    \subitem map needing six colors, 58

  \indexspace

  \item rational
    \subitem dyadic, 28
  \item reduced
    \subitem tree pair, 29
  \item reduction
    \subitem of tree pair, 29
  \item reflection
    \subitem tree, 16
  \item regular expressions, 14
  \item relation
    \subitem pentagonal, 34
    \subitem square, 34
  \item right
    \subitem subtree, 14
    \subitem vine, 36
  \item rigid
    \subitem pattern
      \subsubitem negative, 54
      \subsubitem positive, 54
    \subitem sign assignment, 26
  \item root
    \subitem color, 20
    \subitem of tree, 12
  \item root shift
    \subitem of tree, 16
  \item rooted
    \subitem tree, 12
  \item rotation
    \subitem action of, 32
      \subsubitem on colored pairs, 36
    \subitem along edge, 18
    \subitem pivot vertex, 18
    \subitem signed, 25
    \subitem tree, 18
    \subitem valid, 25
    \subitem valid path, 25

  \indexspace

  \item sequence
    \subitem Catalan numbers, 59
    \subitem Fibonacci numbers, 59
    \subitem Jacobsthal numbers, 59
  \item shadow
    \subitem interval
      \subsubitem of vertex in tree, 41
    \subitem of vertex in tree, 41
    \subitem pattern of tree, 41
  \item sign
    \subitem at vertex, 22
  \item sign assignment
    \subitem compatible with sign structure, 45
    \subitem effect of permutation, 23
    \subitem flexible, 26
    \subitem of tree, 22
    \subitem rigid, 26
  \item sign consistent path, 46
  \item sign structure
    \subitem compatible with coloring, 45
    \subitem compatible with sign assignment, 45
    \subitem of path, 46
    \subitem vertex set, 46
  \item signed
    \subitem rotation, 25
  \item square
    \subitem equivalence of paths, 50
    \subitem move, 48
    \subitem relation, 34
  \item standard model
    \subitem binary tree, 13
  \item subpattern, 54
  \item subtree, 13
    \subitem above another subtree, 46
    \subitem adjacent to a subtree, 46
    \subitem left, 14
    \subitem right, 14

  \indexspace

  \item theorem
    \subitem Trichotomy, 26
  \item Thompson's group \protect\(F\protect\), 27
    \subitem action on trees, 31
    \subitem coloring of element, 35
    \subitem deferment of element, 55
    \subitem multiplication, 29
    \subitem positive monoid, 38
    \subitem presentations, 33
  \item Thompson's group \protect\(V\protect\), 57
  \item torus
    \subitem map needing seven colors, 56
  \item tree
    \subitem ancestor, 12
    \subitem binary, 12
      \subsubitem as union of carets, 29
    \subitem children, 12
    \subitem coloring from vector, 21
    \subitem definition, 12
    \subitem descendant, 12
    \subitem drawing, 13
    \subitem dual to polygon, 15
    \subitem external edge, 12
    \subitem finite, 12
    \subitem finite pair, 11
      \subsubitem prime, 11
      \subsubitem valid color vector, 21
    \subitem infix order, 13
    \subitem internal edge, 12
    \subitem internal vertex, 12
    \subitem leaf
      \subsubitem color, 20
      \subsubitem edge, 12
    \subitem leaves, 12
    \subitem locally finite, 12
    \subitem locally ordered, 12
    \subitem parent, 12
    \subitem prefix order, 13
    \subitem projection, 15
    \subitem reflection, 16
    \subitem root, 12
    \subitem root edge, 12
    \subitem root shift, 16
    \subitem rooted, 12
    \subitem rotation, 18
    \subitem shadow
      \subsubitem of vertex, 41
      \subsubitem pattern, 41
    \subitem shadow interval
      \subsubitem of vertex, 41
    \subitem sign assignment, 22
    \subitem trivial, 12
    \subitem valid color vector, 21
    \subitem vertex
      \subsubitem color, 20
      \subsubitem depth, 39
      \subsubitem level, 39
      \subsubitem shadow interval of, 41
      \subsubitem shadow of, 41
      \subsubitem sign, 22
  \item trees
    \subitem binary
      \subsubitem union, intersection, difference, 30
  \item triangulated
    \subitem polygon
      \subsubitem pair, 11
  \item Trichotomy Theorem, 26
  \item trivial
    \subitem tree, 12

  \indexspace

  \item valid
    \subitem path of rotations, 25
    \subitem rotation, 25
    \subitem vector
      \subsubitem for pair, 21
      \subsubitem for tree, 21
  \item vector
    \subitem color, 20
      \subsubitem acceptable, 21
      \subsubitem color graph of, 27
      \subsubitem zero set of, 41
    \subitem valid for pair, 21
    \subitem valid for tree, 21
  \item vertex
    \subitem center of caret, 29
    \subitem of tree
      \subsubitem depth, 39
      \subsubitem internal, 12
      \subsubitem level, 39
      \subsubitem shadow, 41
      \subsubitem shadow interval, 41
      \subsubitem sign, 22
    \subitem positive neighborhood, 55
    \subitem tree
      \subsubitem color, 20
  \item vertex set
    \subitem of sign structure, 46
  \item vertices
    \subitem colored, 27
  \item vine, 36
    \subitem left, 36
    \subitem right, 36

  \indexspace

  \item well founded
    \subitem induction, 14
    \subitem order, 14
  \item Whitney
    \subitem theorem, 3

  \indexspace

  \item zero set
    \subitem of color vector, 41

\end{theindex}

\end{document}